\documentclass[10pt]{article}
\usepackage[english]{babel}
\usepackage{natbib}
\usepackage{url}
\usepackage{inputenc}
\usepackage{amsfonts}
\usepackage{amsmath}
\usepackage{empheq}
\usepackage{graphicx}
\graphicspath{{images/}}
\usepackage{parskip}
\usepackage{fancyhdr}
\usepackage{bbold}
\usepackage{amssymb}
\usepackage{vmargin}
\usepackage{array}
\usepackage{amsthm}
\usepackage[ruled,vlined]{algorithm2e}
\usepackage[colorlinks=true,citecolor=ForestGreen, linkcolor=black]{hyperref}
\usepackage{caption, subcaption}
\usepackage{float}
\usepackage{fourier-orns}
\usepackage{listings}
\usepackage{epsfig}
\usepackage{amsmath}               
\newtheorem{assumption}{Assumption}
\newtheorem{theorem}{Theorem}[section]

\newtheorem{lemma}[theorem]{Lemma}
\newtheorem{remark}{Remark}[section]
\newtheorem{proposition}{Proposition}[section]

\theoremstyle{definition}

\DeclareMathAlphabet\mathbfcal{OMS}{cmsy}{b}{n}

\usepackage{listings}
\usepackage[scr=boondoxo,scrscaled=1.05]{mathalfa}
\usepackage{makecell}

\usepackage[shortlabels]{enumitem}
\mathtoolsset{showonlyrefs=true}

\usepackage[bbgreekl]{mathbbol}

\newcommand{\cha}{\mathbb{1}_{\Omega_0}}

\newcommand{\om}{{\Omega_{0}}}

\newcommand{\R}{\mathbb{R}}

\newcommand{\tbu}{\tilde{\mathbf{u}}}
\newcommand{\bbu}{\bar{\mathbf{u}}}
\newcommand{\bu}{\mathbf{u}}
\newcommand{\bU}{\mathbf{U}}
\newcommand{\bydef}{\stackrel{\mbox{\tiny\textnormal{\raisebox{0ex}[0ex][0ex]{def}}}}{=}}
\newcommand{\bgam}{\text{\boldmath{$\gamma$}}}
\newcommand{\bGam}{\text{\boldmath{$\Gamma$}}}

\newcommand{\bpi}{\text{\boldmath{$\Pi$}}}

\newcommand{\obpi}{\overline{\bpi}}

\newcommand{\og}{\overline{\bgam}}
\newcommand{\ogd}{\overline{\bgam}^\dagger}
\newcommand{\oGG}{\overline{\bGam}}
\newcommand{\oGGD}{\overline{\bGam}^\dagger}
\newcommand{\out}{\mathbb{1}_{\mathbb{R}\setminus\Omega_0}}
\newcommand{\mbf}[1]{\mathbf{#1}}
\newcommand{\mbb}[1]{\mathbb{#1}}

\usepackage[dvipsnames]{xcolor}

\title{Proving the existence of localized patterns and saddle node bifurcations in 1D activator-inhibitor type models
}
\author{
Dominic Blanco \footnote{McGill University, Department of Mathematics and Statistics, 805 Sherbrooke Street West, Montreal, QC, H3A 0B9, Canada. {\tt dominic.blanco@mail.mcgill.ca}}
\and
Matthieu Cadiot
\footnote{CMAP, CNRS, Ecole polytechnique, Institut Polytechnique de Paris, 91120 Palaiseau, France. {\tt matthieu.cadiot@polytechnique.edu}}
\and
Daniel Fassler \footnote{Concordia University, Department of Mathematics and Statistics, 1400 De Maisonneuve Blvd. W. Montreal, QC H3G 1M8 Canada. {\tt daniel.fassler@mail.concordia.ca}}}
\begin{document}

\maketitle
\begin{abstract}
    In this paper, we develop a methodology for constructively proving the existence of saddle-node bifurcations on unbounded domains. In particular, we combine several previous results in order to formulate a Newton-Kantorovich argument to obtain our results. We demonstrate our approach on a class of 1D activator-inhibitor models. The first ingredient is for us to prove the existence of stationary localized 1D solutions. Our approach relies on having an approximate, compactly supported solution, $\overline{\mathbf{u}}$. We construct this solution using Fourier series and then define a well-chosen fixed point map that is contracting on a neighborhood around $\overline{\mathbf{u}}$. For this matter, we construct an approximate inverse of the linearization around $\overline{\mathbf{u}}$ and establish sufficient conditions under which the contraction is achieved. The second ingredient is to derive an augmented zero finding problem whose solutions correspond to a saddle-node bifurcations. Extending the methodology for proving the existence of localized solutions to this augmented problem, we are able to constructively obtain the existence of saddle-nodes on unbounded domains. The third ingredient is for us to verify that we obtain a non-degenerate saddle-node bifurcation. To do so, we control the spectrum of the linearization around the solution to the augmented zero finding problem. This allows us to rule out a degenerate Hopf bifurcation, as well as controlling exactly the number of unstable directions at the bifurcation. Quantitative explicit estimates as well as the requisite codes for the rigorous numerics are provided.
\end{abstract}
\begin{center}
\end{center}
\section{Introduction}\label{sec : introduction}
 In this paper, we consider the following class of activator-inhibitor Partial Differential Equations (PDEs) posed on the real line:
\begin{equation}\label{eq:original_system}
    \begin{split}
        \partial_t u &= \delta^2\Delta u - cu + dv + u^2v + a,  \\
        \partial_t v &= \Delta v + hu -jv -u^2v + b \\
    \end{split}
    \quad ~~ (u,v) = \left(u(x,t), v(x,t)\right), ~~ x \in \R,
\end{equation}
where $\delta$, $c$, $d$, $h$, $j$, $a$, and $b$ are positive  parameters, and $\Delta$ is the Laplacian operator. 
These equations arise naturally in various fields such as biology and chemistry, where they model a chemical reaction between two substances $U$ and $V$. The first species is the activator ($U$) with concentration $u$ and the other one is an inhibitor ($V$) with concentration $v$~\citep{GrayScott1985}. This system is set in the regime $0 < \delta \ll 1$, where the inhibitor diffuses more than the activator. This class of PDEs encompasses many reaction-diffusion systems that have been studied extensively in the literature, such as the Schnakenberg model~\citep{schankenberg_og} ($c = 1, d= h=j=0$), the Sel'kov Schnakenberg model~\citep{selkov_schankenberg_og} ($c = 1, d= j, h=0$), the Brusselator model~\citep{brusselator_og}($h = c-1 > 0, d=j, h=0$), the Root hair model~\citep{root_hair_og}($a = 0, d = j, h > c > 0$), and the Glycolysis model~\citep{glycolysis_og}($c = 1, d=j, a=h=0$). 
In this paper, we wish to answer some of the questions asked by the authors of \citep{AlSaadi_unifying_framework} regarding the class of activator-inhibitor models \eqref{eq:original_system}, and the existence of saddle-node bifurcations for localized states. For that matter, we develop a general computer-assisted methodology for constructively proving the existence of stationary localized solutions and (localized) saddle-node bifurcations in \eqref{eq:original_system}. Our approach relies on combining various ingredients, including computer-assisted analysis on unbounded domains, augmented maps for proving bifurcations and a control of the spectrum of linearized operators. Combining these features will allow us to derive our framework and to produce novel results in the field of localized patterns. Before delving into the details of our analysis, we present a non-exhaustive review of related results on localized solutions, as well as a review on computer-assisted techniques for differential equations on unbounded domains.

\subsection{Stationary localized solutions}

The study of pattern formation originates in the seminal work of Alan Turing~\citep{turing_morphogenesis}. In this paper, Turing demonstrated how patterns can emerge from perturbations of a steady-state solution through what is now referred to as a Turing instability, or diffusion-driven instability. In one spatial dimension, this mechanism typically manifests as a supercritical pitchfork bifurcation from the steady-state solution~\citep{Meron2015}.
Reaction–Diffusion (RD) systems have attracted significant interest due to their physical relevance, with applications spanning biology~\citep{root_hair_og, selkov_schankenberg_og}, chemistry~\citep{GrayScott1985, schankenberg_og, brusselator_og}, and ecology~\citep{SITEUR201481, ZELNIK201727}. These systems are particularly compelling because they often exhibit rich and complex behaviors while maintaining a relatively simple mathematical structure. Their nonlinearities are typically polynomial, as in~\eqref{eq:original_system}, which includes a simple auto-catalytic term of the form $\pm u^2v$ in each equation. RD systems have been studied extensively using a variety of analytical approaches. For example, geometric perturbation theory and spatial dynamics have been employed to establish the existence of localized and periodic solutions on long one-dimensional domains in the Gray–Scott model~\citep{Doelman1997}, later extended to higher dimensions~\citep{MuratovOsipov2000} and to more general RD systems~\citep{doelman_geometric}. The stability of localized patterns can be investigated using asymptotic expansions, wherein the solution is analyzed through an inner expansion near the pattern and a far-field expansion in the tail. Matching these expansions provides approximations of the eigenvalues of the linearized PDE operator, yielding stability criteria. This method has been successfully applied in the one-dimensional Gierer–Meinhardt model~\citep{iron_semistrong}, the Schnakenberg model~\citep{Iron2004}, and the Gray–Scott model~\citep{muratov_semistrong}. Another central phenomenon in RD systems is the connection between localized pattern formation and homoclinic snaking arising from a Hamiltonian–Hopf bifurcation, which has been investigated in~\citep{knobloch_dissipative, VERSCHUEREN2021132858}.

In recent work, Champneys et al.\citep{Champneys_Bistability} outlined several mechanisms for pattern formation in RD systems. They investigated the bifurcation structure of the Brusselator model, constructing approximate solutions via matched asymptotic expansions and numerically continuing them to reveal homoclinic snaking. A similar analysis was carried out for the Schnackenberg model in~\citep{alsaadi_spikes}, where the same approach was employed to build approximate solutions and to numerically observe the system’s snaking behavior, alongside a stability analysis of spike solutions. It is important to note that the validity of matched asymptotic expansions is restricted to specific parameter regimes: for the Brusselator, this requires $\delta \ll 1$, while for the Schnackenberg model the appropriate scaling is $a = \epsilon\alpha$, $b = \epsilon^2 \beta$, with $\epsilon \ll 1$. Building on these results,~\citep{AlSaadi_unifying_framework} extended the analysis to demonstrate the ubiquity of such behavior in systems of the form~\eqref{eq:original_system}. Using weakly nonlinear analysis, they characterized the bifurcation structure of~\eqref{eq:original_system}, identifying parameter regions where homoclinic snaking is expected to occur for both large and small parameter values. They then applied matched asymptotic expansions to numerically explore the snaking region and further examined the stability of solutions within this region as the diffusion coefficient $\delta$ increases. It should be emphasized, however, that while their theoretical analysis is posed on $\mathbb{R}$, the corresponding numerical results were obtained on long one-dimensional domains.

In their conclusion, the authors of~\citep{AlSaadi_unifying_framework} raise the critical open question of providing a rigorous proof for their results. Addressing this challenge is essential, as it bridges the gap between formal asymptotics, numerical exploration, and mathematical rigor. In the next subsection, we present a computer-assisted approach that delivers such a rigorous justification, thereby validating and strengthening their numerical observations. Constructive existence proofs of patterns, similar to Figure~\ref{fig : intro solutions}, as well as a constructive proof of a saddle node will be provided. In particular, our method is based on the computer-assisted analysis presented in~\citep{gs_cadiot_blanco}. As these models arise from physical context, the question of the stability of these solutions is also studied by rigorously controlling the spectrum of the jacobian of the linearization around an approximate solution using the approach of~\citep{cadiot2025stabilityanalysislocalizedsolutions}

\begin{figure}[H]
    \centering
    \begin{subfigure}[b]{0.49\textwidth}
        \centering
        \epsfig{figure=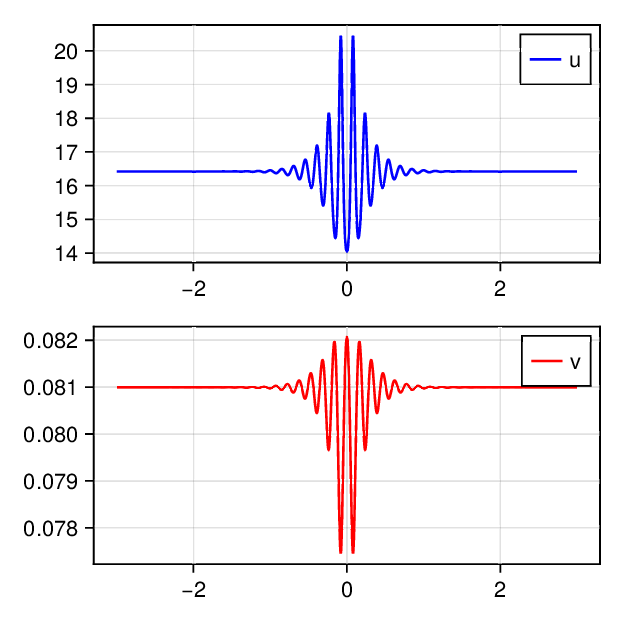, width=\textwidth}
        \caption{Approximation of a localized solution in the Brusselator equation}\label{fig : multi spike brus intro}
    \end{subfigure}
    \hfill
    \begin{subfigure}[b]{0.49\textwidth}
        \centering
        \epsfig{figure=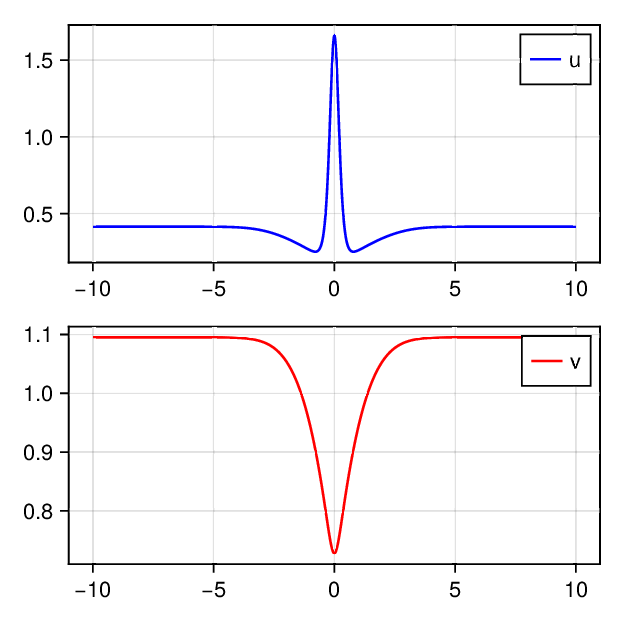, width=\textwidth}
        \caption{Approximation of a saddle node bifurcation  in the Glycolysis equation}\label{fig : saddle gly intro}
    \end{subfigure}
    \caption{Approximation of a steady-state and saddle node bifurcation of the activator-inhibitor models for different parameter regimes.}\label{fig : intro solutions}
\end{figure}

\subsection{A Computer-assisted approach}

One way to approach solving semilinear PDEs on $\mathbb{R}^m$ is to use computer-assisted proofs (CAPs) techniques. CAPs can help derive a constructive methodology that relies on a Newton-Kantorovich argument. It involves the computation of specific bounds, which can be evaluated rigorously on the computer by the use of interval arithmetic for instance \citep{julia_interval}.  Traditionally, computer-assisted methods have split into two main categories. On one hand, spectral approaches (like Fourier series on toroidal domains or Chebyshev polynomials on intervals) have provided highly accurate constructive proofs in non-trivial problems. Even though the spectral accuracy is always desirable, its applicability might be compromise by the complexity of the PDE domain. On the other hand, finite-elements approaches have provided highly versatile means to tackle PDE problems posed on complex geometries, at the cost of a limited accuracy.  In what follows, we present a literature review of both worlds, highlighting the flexibility offered for deepening PDE analysis.

\par We begin on bounded domains with periodic boundary conditions. Here, the theory is well developed. As the solution is periodic, one can represent it using Fourier series. One will turn the PDE is into a zero finding problem $F(U) = 0$ where $U$ is in a sequence space. One can then compute an approximate solution $\overline{U}$ using Fourier sequences. In order to handle $DF(\overline{U})^{-1}$, one can exploit the
fact that the Fréchet derivatives are (asymptotically) diagonally
dominant. Indeed, the linear part
$L = DF(0)$ dominates for high-order modes. As a result, the tail of $DF(\overline{U})$
can be estimated as a diagonally-dominant infinite dimensional matrix. Since compactness justifies approximating infinite dimensional operators with matrices, the common approach is to approximate $DF(\overline{U})^{-1}$ as a finite matrix stored on the computer and a diagonal tail which can be controlled theoretically. We refer the interested reader to  \citep{JB_symmetries_1,JB_symmetries_2,breden2022computer,period_kuramoto,MR3633778,Sander_equilibrium,gabriel_pfc,lessard_1,lessard_2}
for some illustrations. There are also approaches when the tail is not diagonally dominant, and we refer the interested reader to \citep{cyranka2018construction,MR4379799,MR3392647} for examples. Additionally, the approach is applicable when the nonlinear part is nonpolynomial in certain cases. We refer the interested reader to the thesis \citep{marco_thesis} for a general summary, and the recent papers \citep{maxime_paper_continuation,olivier_kevin_paper,matthieu_lindsay,lindsay_suspensionbrdige} for illustrations. Note that the methods for handling $DF(\overline{U})^{-1}$ are unchanged as a nonpolynomial nonlinearity would only affect the nonlinear part.
\par Still on bounded domains, but we now move to those with non-periodic boundary conditions. Many
different approaches have been used to establish the norm of $D\mathbb{F}^{-1}(u_{0}):Y\to X$
or its weak formulation where $X$ and $Y$ are Banach spaces. The first method we discuss was developed by the authors of \citep{nakao1990numerical,nakao2001numerical}. The authors approximate the inverse of $DF(u_{0})$
by the inverse of a finite-dimensional projection on a finite elements. Under certain assumptions on the finite element space, the authors justify inversion of the linearization and then use CAPs to perform rigorous proofs. This method was recently improved through the use of the Schur complement by the authors of
\citep{MR4182090}. Another extension of this result was developed by the authors of  \citep{watanabe2005verify_invert}. In particular, the authors
developed specific tools in order to invert $D\mathbb{F}(u_{0})$
in the case of an elliptic PDE. The authors use a finite dimensional subspace of
Sobolev spaces to obtain invertibility criteria and
the value of the upper bound. The works \citep{watanabe2015estimation,watanabe2019improved,watanabe2013posteriori,plum_wata2016norm}
present results which further extend the method and give sharper
criteria regarding the inverse. Also, on bounded domains,
the authors of  \citep{liu_self_adjoint,liu_inverse,liu_laplacian,liu_elliptic_boundary} developed a method to rigorously estimate
eigenvalues of the linearized operator.
By again considering a finite dimensional subspace, the eigenvalues of the
linearized operator are bounded, from above and below. This is done through the use of the eigenvalues
of a finite-dimensional operator. Then, using the theory of symmetric
operators, the norm of $D\mathbb{F}(u_{0})^{-1}$ is determined through
the minimal eigenvalue. This method is particularly efficient if $D\mathbb{F}(u_{0})$
is self-adjoint itself. It can also be applied to a general operator
by considering $D\mathbb{F}(u_{0})^{*}D\mathbb{F}(u_{0})$ where $D\mathbb{F}(u_{0})^{*}$
is the adjoint operator. Finally, the method by Plum et al. \citep{plum_numerical_verif,plum1990eigenvalue_homotopy,plum1991bounds} follows a similar approach to those done in \citep{liu_self_adjoint,liu_inverse,liu_laplacian,liu_elliptic_boundary}. Indeed, the authors use an approximation of the
eigenvalues of $D\mathbb{F}(u_{0})$ to compute the norm of its inverse.
Moreover, as $\Omega$ is bounded, they use the compactness of the
resolvant to build a basis of eigenfunctions of the resolvant. They are able to find a lower bound for the minimal eigenvalue $|\lambda|$ and use results on equivalence of norms to get a bound for the norm of $D\mathbb{F}^{-1}(u_{0}):L^{2}(\Omega)\to H^{2}(\Omega)\bigcap H_{0}^{1}(\Omega)$.

In the present paper, we choose to develop a spectral approach for studying \eqref{eq:original_system}, and the existence of saddle-node bifurcations.  For that matter, we rely on the analysis developed in \citep{unbounded_domain_cadiot} for semi-linear problems posed on $\R^m$, providing highly accurate existence proofs. Before delving into the details of our set-up, we present a focused review of the existing CAP approaches for studying PDEs on unbounded domains.

\par For the  ODE case, a major framework for studying localized states is  the parameterization method \citep{MR1976079, MR1976080, MR2177465}. This method relies on the rigorous parametrization of invariant manifolds at the zero solution. Then, this allows one to formulate a projected boundary value problem which one can treat using Chebyshev series or splines. The approach can also be applied to nonautonomous ODEs through the use of automatic differentiation (cf. \citep{automatic_differentiation}) and was demonstrated in \citep{miguel_soliton} for the Schrodinger equation. Additionally, the approach has been used to prove branches of solutions, and has been recently used to prove homoclinic snaking in the 1D Swift Hohenberg equation by the authors of \citep{gabriel_snaking}. Still for ODEs, the analysis derived in \citep{breden2025solutionsdifferentialequationsfreudweighted, breden2024constructiveproofssemilinearpdes} allows to produce constructive existence proofs in a well-chosen class of equations. Indeed, such equations possess a confining potential, allowing the resolvent of the differential operator to be compact on a well-chosen space of functions. Note that the method derived in \citep{breden2024constructiveproofssemilinearpdes} also applies to PDEs possessing a linear part of the form $\Delta u + \frac{x}{2} \cdot \nabla u$. 
Moving to PDEs, but in the weak formulation, the authors of \citep{plum_numerical_verif} presented a method for proving weak solutions to second and fourth-order PDEs. Their approach relies on the rigorous control of the spectrum of the linearization around an approximate solution. Then, they use a homotopy argument and the Temple-Lehmann-Goerisch method (see Section 10.2 in \citep{plum_numerical_verif}). This approach was then applied by Wunderlich in \citep{plum_thesis_navierstokes}. Wunderlich proved the existence of a weak solution to the Navier-Stokes equations defined on an infinite strip with an obstacle. As far as strong solutions are concerned, \citep{olivier_radial} provides a methodology for proving the existence of radially symmetric solutions. Using the radially symmetric ansatz,  the PDE is transformed into an ordinary differential equation. The approach then relies on  a rigorous enclosure of the center stable manifold by using a Lyapunov-Perron operator. This allows one to solve a boundary value problem on $(0,\infty)$. Removing the radially symmetric assumption, the authors of \citep{unbounded_domain_cadiot} provide a general method for proving the existence and local uniqueness of localized solutions to semilinear autonomous PDEs on $\mathbb{R}^m$. The approach is based on Fourier series and describes the necessary tools to construct an approximate solution, $u_0$, and an approximate inverse $\mathbb{A}$ of the linearization $D\mathbb{F}(u_0)$ about $u_0$. An application to the 1D Kawahara equation was provided in the same paper. In \citep{sh_cadiot}, the method of \citep{unbounded_domain_cadiot} was applied to the planar Swift Hohenberg PDE. The method was then extended further by the authors of \citep{gs_cadiot_blanco}. In \citep{unbounded_domain_cadiot}, the method was presented for scalar PDEs. In \citep{gs_cadiot_blanco}, the authors generalized the approach to systems of PDEs. This allows for rigorous proofs of localized solutions to systems of PDEs. The approach was demonstrated on the 2D Gray-Scott system of equations. Then, in \citep{symmetry_blanco_cadiot}, the authors extended further this methodology to treat existence proofs of symmetric solutions, where the symmetry might be given by a general group of symmetries. The 2D planar Swift Hohenberg equation was again used as an illustration, allowing for the authors to obtain existence  proofs of solutions with various dihedral symmetries.

\par Now that we have discussed the existing literature for stationary localized patterns and state of the art techniques in rigorous numerics, we outline how we will mesh these ideas together for our analysis.

\subsection{Ingredients of our Approach}
\par Our goal is to provide rigorous validations to some of the observations made by the authors of \citep{AlSaadi_unifying_framework}. As mentioned previously, we will need to develop a methodology for constructively proving the existence of saddle-node bifurcations on unbounded domains. Our approach will rely on meshing previously developed ingredients. The first of which is a rigorous study on the existence of localized patterns. We base our analysis on the method developed in \citep{gs_cadiot_blanco}.  Going back to the study of  \eqref{eq:original_system}, we assume in the rest of the paper that $j=d$ and that $c>h$. Then, we use the following change of variables and re-normalisations
 \begin{align*}
     u_1(x,t) \bydef u\left( x,t\right) - \lambda_1 \text{ and } u_2(x,t) \bydef  v(x,t) - \lambda_2, \text{ where }\\
      \lambda_1 \bydef \frac{a+b}{c-h} \text{ and } \lambda_2 \bydef \frac{c(c-h)(a+b) - a(c-h)^2}{(a+b)^2 + d(c-h)^2}.
 \end{align*}
Note that $(\lambda_1,\lambda_2)$ is a steady state of \eqref{eq:original_system}. Consequently, using the above definitions, we can look for $(u_1,u_2)(x,t) \to 0$ as $|x| \to \infty$. In fact, we obtain that $(u_1,u_2)$ equivalently solve
\begin{equation}\label{eq : gen_model}
\begin{split}
   \partial_t u_1 &= \rho \Delta u_1 + \nu_1u_1 + \nu_2 u_2 + u_1^2 u_2 + \nu_4 u_1^2 + \nu_5 u_1 u_2,  \\
      \partial_t u_2 &=\Delta u_2 - \nu_2 u_2 + \nu_3 u_1 - u_1^2 u_2 - \nu_4 u_1^2 - \nu_5 u_1 u_2
      \end{split}
\end{equation}
where
\begin{align}
    &\rho \bydef \delta^2,~\nu_1 \bydef  2\lambda_1 \lambda_2 - c,~\nu_2 \bydef d + \lambda_1^2,~\nu_3 \bydef h - 2\lambda_1 \lambda_2,~\nu_4 \bydef \lambda_2,~\nu_5 \bydef 2\lambda_1.
\end{align}
Using the above, $(u,v)$ solves \eqref{eq:original_system} if and only if $(u_1,u_2)$ solves \eqref{eq : gen_model}. Let us denote $\mbf{u}=(u_1, u_2)$ and define $\mathbb{f}$ as follows
\begin{align}
    \mbb{f}(\mbf{u}) \bydef \mathbb{l} \bu + \mathbb{g}(\bu),
\end{align}
where
\begin{equation}\label{def : def l and g intro}
    \mathbb{l} \bydef \begin{bmatrix}
        \rho \Delta + \nu_1 I & \nu_2 I \\
        \nu_3 I & \Delta - \nu_2
    \end{bmatrix}  = \begin{bmatrix}
        \mathbb{l}_{11} & \mathbb{l}_{12} \\ \mathbb{l}_{21} & \mathbb{l}_{22}
    \end{bmatrix}, ~~  \mathbb{g}(\mathbf{u}) \bydef \begin{bmatrix}
        u_1^2 u_2 + \nu_4 u_1^2 + \nu_5 u_1 u_2 \\
        - u_1^2 u_2 - \nu_4 u_1^2 - \nu_5 u_1 u_2
    \end{bmatrix}.
\end{equation}
In fact, \eqref{eq : gen_model} has the condensed parabolic form
\begin{align}\label{eq : parabolic form of the problem}
    \partial_t \bu  = \mathbb{f}(\bu).
\end{align}

To begin our study, we will focus our attention on stationary localized solutions, that is $\bu : \R \to \R^2$ such that $\mathbb{f}(\bu) = 0$ and  satisfying $\bu(x) \to 0$ as $|x| \to \infty.$  Given a well-chosen Hilbert space $\mathcal{H}_e$ (see \eqref{def : Hilbert space Hl}), containing even localized functions, we look for zeros of $\mathbb{f}$ in $\mathcal{H}_e$. In fact, we demonstrate that $\mathbb{f} : \mathcal{H}_e \to L^2_e$ is a smooth operator, where $L^2_e$ is the restriction of even functions of $L^2(\R) \times L^2(\R)$. Then, following the analysis derived in \citep{unbounded_domain_cadiot, gs_cadiot_blanco}, we prove the existence of zeros of $\mathbb{f}$ using a Newton-Kantorovich approach (cf. Section \ref{sec : patterns}). In fact, our analysis relies on the construction of an approximate solution $\bar{\bu} \in \mathcal{H}_e$ and of an approximate inverse $\mathbb{A} : {L}^2_e \to \mathcal{H}_e$ for the Fréchet derivative $D\mathbb{f}(\bar{\bu})$. Indeed, we can then define a Newton-like fixed point operator $\mathbb{T}$ given as 
\begin{align*}
    \mathbb{T}(\bu) \bydef \bu - \mathbb{A}\mathbb{f}(\bu)
\end{align*}
and we prove that there exists $r>0$ such that $\mathbb{T}$ is contracting from $\overline{B_r(\bar{\bu})}$ to itself, where $\overline{B_r(\bar{\bu})}$ is the closed ball in $\mathcal{H}_e$ of radius $r$ centered at $\bar{\bu}$. Using the Banach-fixed point theorem, this implies the existence of a unique zero $\tilde{\bu}$ of $\mathbb{f}$ in $\overline{B_r(\bar{\bu})}$. The verification of the contractivity of $\mathbb{T}$ on $\overline{B_r(\bar{\bu})}$ involves the rigorous computation of different quantities, which we expose in Section \ref{sec : Bounds of patterns}. This allows us to provide novel constructive existence proofs of localized solutions in \eqref{eq : gen_model}. 
\par Building upon existence proofs of solutions, we investigate the existence of saddle-node bifurcations. {In fact, we aim at applying similar ideas as in \citep{gs_cadiot_blanco} to a well-chosen augmented system in order to prove the existence of saddle node bifurcations. This provides a novel computer-assisted approach  for proving the existence of saddle-node bifurcations for localized solutions.}  We now write $\mathbb{f}(\bu) = \mathbb{f}(\rho,\bu)$ to explicit the dependency of $\mathbb{f}$ on the parameter $\rho$, as given in \eqref{eq : gen_model}.
By definition, a saddle-node bifurcation at $(\tilde{\rho}, \tilde{\bu})$ satisfies the following
\begin{align*}
    \mathbb{f}(\tilde{\rho},\tilde{\bu}) = 0, ~ D_u\mathbb{f}(\tilde{\rho},\tilde{\bu})\mathbf{w} =0 
\end{align*}
and some extra non-degeneracy conditions (given in \citep{jp_saddle_node} for instance). In particular, \citep{jp_saddle_node} provides an equivalent definition of a saddle-node bifurcation, thanks to isolated zeros of a well-chosen map. For the second ingredient of our approach, we introduce the  Hilbert spaces $H_1 = \mathbb{R}\times \mathcal{H}_e \times \mathcal{H}_e$ and $H_2 = \mathbb{R}\times L_e^2 \times L^2_2$ (see \eqref{def : Hilbert space Hl}) and  we define the so-called  \emph{Saddle-Node Map} $\mathbb{F} : H_1 \to H_2$ as follows
\begin{equation}\label{eq:saddle_node_map}
  \mathbb{F}(\mathbf{x}) \bydef \begin{bmatrix}
\phi(\mathbf{w})  -\alpha \\\mathbb{f}(\rho,\mbf{u}) \\
      D_{\mbf{u}}\mathbb{f} (\rho,\mbf{u})\mbf{w} 
  \end{bmatrix}, \text{ with } \mathbf{x} = (\rho, \mathbf{u},\mathbf{w}) \in H_1,
\end{equation}
and where $\alpha \in \mathbb{R}$, $\phi : \mathcal{H}_e \to \R$ is a non-trivial linear functional. The additional function, $\mathbf{w}$ denotes the eigenvector corresponding to the zero eigenvalue of $D_{\mathbf{u}}\mathbb{f}(\rho,\mathbf{u}).$ Then, using \citep{jp_saddle_node}, we have the following sufficient conditions for a saddle-node bifurcation.
\begin{lemma}\label{lem : saddle node def}
    Suppose that $\tilde{\mathbf{x}} = (\tilde{\rho}, \tilde{\mathbf{u}},\tilde{\mathbf{w}}) \in H_1$ is a non-degenerate zero of $\mathbb{F}$, that is $\mathbb{F}(\tilde{\mathbf{x}}) = 0$ and $D\mathbb{F}(\tilde{\mathbf{x}})$ is invertible. Moreover, assume that $D_u\mathbb{f}(\tilde{\rho},\tilde{\mathbf{u}})$ is a Fredholm operator of index zero. Finally, assume that the spectrum of $D_u\mathbb{f}(\tilde{\rho},\tilde{\mathbf{u}})$ does not intersect the imaginary axis, except at zero.  Then, the problem $\mathbb{f}(\rho,\mathbf{u}) = 0$ undergoes a saddle-node bifurcation at  $(\tilde{\rho}, \tilde{\mathbf{u}})$.
\end{lemma}

\begin{remark}
    In the previous lemma, the assumption that the spectrum of $D_u\mathbb{f}(\tilde{\rho},\tilde{\mathbf{u}})$ does not intersect the imaginary axis allows to eliminate possible degenerate saddle-node bifurcations at which a Hopf bifurcation simultaneously occurs. This assumption was not necessary in the applications of \citep{jp_saddle_node} because of self-adjointness. 
\end{remark}

In practice, non-degenerate zeros of $\mathbb{F}$ in $H_1$ are obtained using a Newton-Kantorovich approach, similarly as for the existence proofs of localized solutions. Our approach is computer-assisted and depends on the rigorous computation of various quantities, explicited in Section \ref{sec : Bounds for Saddle Nodes}. Moreover, the Fredholm property of $D_u\mathbb{f}(\tilde{\rho},\tilde{\mathbf{u}})$ was established in \citep{unbounded_domain_cadiot}. In order to control the spectrum, we require a third ingredient developed in \citep{cadiot2025stabilityanalysislocalizedsolutions}. More specifically, the control of the spectrum of $D_u\mathbb{f}(\tilde{\rho},\tilde{\mathbf{u}})$ is done a posteriori (that is after the existence proof of a zero of  $\mathbb{F}$), and is presented in Section \ref{sec : control of the spectrum}.  In fact, this approach also allows us to determine the stability (with respect to even perturbations) of localized solutions. We provide such a stability analysis in Section \ref{sec : control of the spectrum} as well. By combining the ingredients for proving localized patterns, zeros of the augmented system \eqref{eq:saddle_node_map}, and the control of the spectrum, we are able to obtain the existence of saddle-node bifurcations on unbounded domains. Furthermore, we aim to provide such estimates in such a way that the results are widely applicable. That is, all of our estimates are quantitative and explicitly provided. Similarly, we provide the requisite codes for the related rigorous computations on the GitHub repository \citep{julia_blanco_fassler}.

This manuscript is organized as follows. Section~\ref{sec : setup} introduces the notation, assumptions, and restrictions used throughout the paper. In Section~\ref{sec : patterns}, we introduce a Newton-Kantorovich approach associated to the construction of an approximate solution $\bar{\mathbf{u}}$ and of an approximate inverse $\mathbb{A}$ for $D\mathbb{f}(\bar{\mathbf{u}})$. The specific construction of our approximate objects allow us to derive explicit formulas in Section~\ref{sec : Bounds of patterns} for the application of the Newton-Kantorovich approach. In the same section, we present applications of our analysis and derive multiple constructive existence proofs of localized solutions. In Section~\ref{sec : saddle node}, we conduct a similar analysis for the constructive proofs of solutions to the augmented problem \eqref{eq:saddle_node_map}. Finally, we derive in Section \ref{sec : control of the spectrum} a computer-assisted approach for controlling the spectrum of the Jacobian at a given localized solution. This allows us to establish the existence of a saddle-node in the Glycolysis model, as well as stability of some localized solutions.

\section{Setup of the Problem}\label{sec : setup}

Recall that we  wish to prove the existence of localized solutions to \eqref{eq : gen_model} for the spatial domain $\R$. For this matter, we introduce usual notations of functional analysis on $\R$. We first define the Lebesgue notation on a product space as $L^2 = L^2(\mathbb{R}) \times L^2(\mathbb{R})$ and $L^2(\om)$ on a bounded domain $\om$ in $\mathbb{R}$. More generally, $L^p$ denotes the usual $p$ Lebesgue product space with two components on $\mathbb{R}$ associated to its norm $\|\cdot\|_{p}$ defined as
$\|\mathbf{u}\|_{p} \bydef (\|u_1\|_{p}^p + \|u_2\|_{p}^p)^{\frac{1}{p}}$
where $\mathbf{u} \bydef (u_1,u_2).$ Moreover, given $s \in \mathbb{R}$, denote by $H^s \bydef H^s(\mathbb{R}) \times H^s(\mathbb{R})$ the usual Sobolev space on $\mathbb{R}$.  For a bounded linear operator $\mathbb{K} : L^2 \to L^2$, denote by $\mathbb{K}^*$ the adjoint of $\mathbb{K}$ in $L^2$. Moreover, if $\mathbf{u} \in L^2$ where $\mathbf{u} = (u_1,u_2)$, denote by $\hat{\mathbf{u}} \bydef \mathcal{F}(\mathbf{u}) \bydef (\mathcal{F}(u_1),\mathcal{F}(u_2))$ the Fourier transform of $\mathbf{u}$. More specifically,  $\displaystyle \hat{u}_{j}(\xi)  \bydef  \int_{\mathbb{R}}u_j(x)e^{-i2\pi x \cdot \xi}dx$ for all $\xi \in \mathbb{R}$ and $j \in \{1,2\}$. 

Adopting the notation of \citep{unbounded_domain_cadiot}, we look for $\mbf{u}=(u_1, u_2)$ such that
\begin{align}
    \mbb{f}(\mbf{u})  = 0
\end{align}
where
$\mathbb{f}$ is given in \eqref{def : def l and g intro}. Solving this equation requires us to choose a natural space of functions for $\mathbf{u}$ to live. In fact this function space is of main importance as it determines the required computations in our analysis. For that matter, we follow the construction introduced in \citep{unbounded_domain_cadiot}, constructing the space of solutions thanks to the linear operator $\mathbb{l}$. Note that $\mathbb{l}$ possesses a symbol (Fourier transform) $\mathscr{l}: \mathbb{R} \to \mathbb{R} \times \mathbb{R}$ given as 
\begin{equation}\label{def : definition of symbol l}
    \mathscr{l}(\xi) \bydef \begin{bmatrix}
        \mathscr{l}_{11}(\xi) & \mathscr{l}_{12}(\xi) \\
        \mathscr{l}_{21}(\xi) & \mathscr{l}_{22}(\xi)
    \end{bmatrix}
    = \begin{bmatrix}
         -\rho|2\pi\xi|^2 + \nu_1 & \nu_2 \\
        \nu_3 & -|2\pi\xi|^2 - \nu_2
    \end{bmatrix}
\end{equation}
for all $\xi \in \mathbb{R}$. We now recall a necessary assumption from \citep{gs_cadiot_blanco} that our framework requires. We will later verify it for a class of  Activator-Inhibitor models \eqref{eq : gen_model}.

\begin{assumption}\label{assumption : fourier}
  Given $\mathscr{l}$ as in \eqref{def : definition of symbol l},  assume there exists $\sigma_0 > 0$ such that 
    \begin{equation}\label{eq : fourier transform bounded away from 0}
        |\det(\mathscr{l}(\xi))| \geq \sigma_0 \text{ for all } \xi \in \mathbb{R}.
    \end{equation}
    That, is, $\det(\mathscr{l}(\xi))$ is bounded away uniformly from $0$.
\end{assumption}
We now provide the values of the parameters $\nu_j$ ($j\in \{1, \dots, 5\})$ such that $\mathscr{l}$ satisfies Assumption \ref{assumption : fourier}. 
\begin{lemma}\label{lem : l_invertible}$\mathscr{l}$ is invertible if
    \begin{equation}
        (\rho\nu_2-\nu_1)^2 + 4\rho\nu_2(\nu_1 + \nu_3) < 0,
    \end{equation}\label{eq : invertible condition 1}
    or if
    \begin{equation}
        \begin{cases}
            (\rho\nu_2 - \nu_1)^2 + 4\rho\nu_2(\nu_1 + \nu_3) \geq 0 ~~ \text{ and } \\
            \rho\nu_2 - \nu_1 > \sqrt{(\rho \nu_2 - \nu_1)^2 + 4\rho \nu_2 (\nu_1 + \nu_3)}
        \end{cases}
    \end{equation}\label{eq : invertible condition 2}
\end{lemma}
\begin{proof}
    $\mathscr{l}$ is invertible if and only if $\mathscr{l}_{11}(\xi)\mathscr{l}_{22}(\xi) - \mathscr{l}_{12}(\xi)\mathscr{l}_{21}(\xi) \neq 0, \forall \xi \in \mathbb{R}$. 
    This is equivalent to studying the roots of  the second order polynomial $x \mapsto \rho x^2 + (\rho\nu_2 - \nu_1)x - \nu_2(\nu_1 + \nu_3) $. The proof is then obtained using basic properties of second order polynomials.
\end{proof}
From now on, we assume that the parameters $\rho$ and $\nu_j$ ($j\in \{1, \dots, 5\})$   satisfy the conditions of Lemma \ref{lem : l_invertible}, yielding that $\mathscr{l}$ is invertible. This allows to define the  following norm and inner product \begin{align}\label{def : definition of the norm and inner product Hl}
    \|\mathbf{u}\|_{\mathcal{H}} \bydef \|\mathbb{l}\mathbf{u}\|_{2} ~~ \text{ and } ~~(\mathbf{u},\mathbf{v})_{\mathcal{H}} \bydef (\mathbb{l}\mathbf{u},\mathbb{l}\mathbf{v})_2 
\end{align}
for all $\mathbf{u}, \mathbf{v} \in \mathcal{H}$, where   $\mathcal{H}$ is the Hilbert space
\begin{align}\label{def : Hilbert space Hl}
    \mathcal{H} \bydef \{ \mathbf{u} \in L^2, \|\mathbf{u}\|_{\mathcal{H}} < \infty \}.
\end{align} 
Using Plancherel's identity, we have 
\begin{align*}
(\mathbf{u}, \mathbf{v})_{\mathcal{H}} = (\mathscr{l}\hat{\mathbf{u}},\mathscr{l}\hat{\mathbf{v}})_2 
\end{align*}
for all $\mathbf{u}, \mathbf{v} \in \mathcal{H}.$ In particular, note that $\mathbb{l} : \mathcal{H} \to L^2$ is a well-defined bounded linear operator, which is actually an isometric isomorphism. We now want to prove that the operator $\mathbb{g}: \mathcal{H} \to L^2$ is a smooth operator. For this matter, we recall two spaces from \citep{gs_cadiot_blanco} :  $\mathcal{M}_1$ and $\mathcal{M}_2$ given as
\begin{align}
 &\mathcal{M}_1 \bydef \left\{ M = \left(M_{i,j}\right)_{i,j \in \{ 1,2 \} }, \text{ where } M_{i,j} \in L^{\infty}(\mathbb{R}) \text{ and } ~ \|M\|_{\mathcal{M}_1} < \infty \right\}  \\
&\mathcal{M}_2 \bydef \left\{ M = \left(M_{i,j}\right)_{i,j \in \{ 1,2 \} }, \text{ where } M_{i,j} \in L^{2}(\mathbb{R}) \text{ and } ~ \|M\|_{\mathcal{M}_2} < \infty \right\}
\end{align}
with their associated norms
\\ \noindent\begin{minipage}{.4\linewidth}
\begin{equation}\label{M_2_norm_def}
 \|M\|_{\mathcal{M}_1} \bydef \sup_{\xi \in \mathbb{R}} \sup\limits_{\substack{x \in \mathbb{R}^2\\ |x|_2=1}}  |M(\xi)x|_2\end{equation}\end{minipage}%
\begin{minipage}{.6\linewidth}
\begin{equation}\label{M_21_norm_def}
\vspace{0.15cm}\|M\|_{\mathcal{M}_2} \bydef \max_{i \in \{1,2\}} \left\{\left(\sum_{j = 1}^2 \|M_{i,j}\|_{L^2(\mathbb{R})}^2\right)^{\frac{1}{2}} \right\}\end{equation}
\end{minipage}
In fact, using $\mathcal{M}_1$ and $\mathcal{M}_2$, \citep{gs_cadiot_blanco} provides sufficient conditions for which products on $\mathcal{H}\times \mathcal{H} \to L^2$ are well-defined. Using this result, we obtain the following lemma.
\begin{lemma}\label{corr : banach algebra}
Let $\mathbf{u} = (u_1,u_2), \mathbf{v} = (v_1,v_2),$ and $\mathbf{w} = (w_1,w_2)$. Suppose Assumption~\ref{assumption : fourier} is verified and let $\kappa \bydef \|\mathscr{l}^{-1}\|_{\mathcal{M}_1} \|\mathscr{l}^{-1}\|_{\mathcal{M}_2}$. Then 
\begin{align}
\|u_iv_j\|_2 \leq \kappa \|\mathbf{u}\|_\mathcal{H} \|\mathbf{v}\|_\mathcal{H} ~~ \text{ and } ~~ \|u_iv_jw_k\|_2 \leq \kappa \|\mathscr{l}^{-1}\|_{\mathcal{M}_2} \|\mathbf{u}\|_\mathcal{H} \|\mathbf{v}\|_\mathcal{H} \|\mathbf{w}\|_\mathcal{H}
\end{align}
for any $i,j,k \in \{1,\dots,2\}$. 
\end{lemma}
\begin{proof}
    To begin, observe that
    \begin{align}
        \nonumber \|u_i v_j\|_{2} &\leq \|u_i\|_{\infty} \|v_j\|_{2} \leq \|\mathbf{u}\|_{\infty} \|\mathbf{v}\|_{2}\leq \|\hat{\mathbf{u}}\|_{1} \|\hat{\mathbf{v}}\|_{2}
    \end{align}
where we used Plancherel's theorem.
But now, using the proof of Lemma 2.1 in \citep{gs_cadiot_blanco}, we get
\begin{align}\label{eq : inequality l norm}
    \|\hat{\mathbf{v}}\|_{2} &= \|\mathscr{l}^{-1} \mathscr{l} \hat{\mathbf{v}}\|_{2} \leq \|\mathscr{l}^{-1}\|_{\mathcal{M}_1} \|\mathscr{l}\hat{\mathbf{v}}\|_{2} = \|\mathscr{l}^{-1}\|_{\mathcal{M}_1} \|\mathbf{v}\|_{\mathcal{H}}.
\end{align}
Finally, let us analyze $\|\hat{\mathbf{u}}\|_{1}$. Using the proof of Lemma 2.1 in \citep{gs_cadiot_blanco} again, we get
\begin{align}
    \nonumber \|\hat{\mathbf{u}}\|_{1} &= \|\mathscr{l}^{-1} \mathscr{l} \hat{\mathbf{u}}\|_{1}\leq \|\mathscr{l}^{-1}\|_{\mathcal{M}_2} \|\mathscr{l}\hat{\mathbf{u}}\|_{2} \leq \|\mathscr{l}^{-1}\|_{\mathcal{M}_2} \|\mathbf{u}\|_{\mathcal{H}}
\end{align}
where we again used Plancheral's theorem on the final step. All together, we have
\begin{align}
    \nonumber \|u_i v_j\|_2 &\leq \|\mathscr{l}^{-1}\|_{\mathcal{M}_1}\|\mathscr{l}^{-1}\|_{\mathcal{M}_2} \|\mathbf{u}\|_{\mathcal{H}} \|\mathbf{v}\|_{\mathcal{H}}.
\end{align}
The case $\|u_i v_j w_k\|_{2}$ is treated similarly observing that
\begin{align}
    \nonumber \|u_i v_j w_k\|_{2}  \leq  \|u_i \|_{\infty} \|w_k\|_{\infty} \|v_j\|_{2}. ~~~~~~~~~~~ \qedhere
\end{align}
\end{proof}

Using Lemma \ref{corr : banach algebra}, we obtain that $\mathbb{g} : \mathcal{H} \to L^2$ is a smooth operator. In particular, this yields that $\mathbb{f} : \mathcal{H} \to L^2$ is smooth as well. In fact, Proposition 2.1 in \citep{gs_cadiot_blanco} yields that solutions to \eqref{eq : gen_model} in $\mathcal{H}$ are actually infinitely differentiable, meaning that our analysis on $\mathcal{H}$ will yield strong solutions to \eqref{eq : gen_model}.

Supposing that $\mathbf{u} = (u_1,u_2)$ is a solution of \eqref{eq : gen_model}, then any translation in space provides a new solution. Therefore, in order to isolate a localized solution in the set of solutions, we choose to look for even solutions, that is we enforce that $\mathbf{u}(x) = \mathbf{u}(-x)$ for all $x \in \R$.   With this in mind, we introduce the following even restriction $\mathcal{H}_{e} \subset \mathcal{H}$
\begin{equation}\label{H_l_e_definition}
         \mathcal{H}_{e} \bydef \{ \mathbf{u} \in \mathcal{H} \ | \ \mathbf{u}(x) = \mathbf{u}(-x)\}.
    \end{equation}
Similarly, denote $L_{e}^2$ the Hilbert subspace of $L^2$ satisfying the even symmetry.
In particular we notice that if $\mathbf{u} \in \mathcal{H}_{e}$, then $\mathbb{l}\mathbf{u} \in L^2_{e}$ and $\mathbb{g}(\mathbf{u}) \in L^2_{e}$ ; hence,  $\mathbb{l}$ and $\mathbb{g}$ are well-defined as operators from $\mathcal{H}_{e}$ to $L^2_{e}$. 
 Finally, we look for solutions of the following problem
\begin{equation}\label{eq : f(u)=0 on H^l_e}
    \mathbb{f}(\mathbf{u}) = 0 ~~ \text{ and } ~~ \mathbf{u} \in \mathcal{H}_{e}.
\end{equation}
\subsection{Periodic Spaces}\label{sec : periodic spaces}
In this section, we recall some notations introduced in Section 2.4 of \citep{unbounded_domain_cadiot}. In particular, the objects defined in the above sections have a corresponding representation in Fourier series when restricted to a bounded domain. Specifically, we define our bounded domain of interest $\Omega_0 \bydef (-d,d)$  where $1 \leq d<\infty$.
Moreover, denote $\tilde{n} = \frac{n}{2d} \in \mathbb{R}$ for all $n \in \mathbb{Z}$. As for the continuous case, we want to restrict to Fourier series representing even functions. Let $\mathbf{u}_{n} = ((u_1)_{n},(u_2)_{n})$ be the $n$ Fourier coefficient of $\mathbf{u}$. In terms of Fourier coefficients, the even  restriction reads
\begin{align*}
    \mathbf{u}_{n} = \mathbf{u}_{-n} \text{ for all } n \in \mathbb{Z}.
\end{align*}
In particular, when enforcing the even symmetry, we can restrict the indices of Fourier coefficients from $\mathbb{Z}$ to the reduced set
$
 \mathbb{N}_0 \bydef \left\{n \in \mathbb{Z} \ | \ 0 \leq n \right\}.$
Now, let $(\alpha_n)_{n \in \mathbb{N}_0}$ be defined by 
\begin{equation}\label{def : alpha_n}
    \alpha_n \bydef \begin{cases}
        1 &\text{ if } n=0\\
       2 &\text{ if } n > 0. 
    \end{cases}
\end{equation}
In particular, $\alpha_n$ is the size of $\mathrm{orb}_{\mathbb{Z}_2}(n)$.  Next, let $\ell^p(\mathbb{N}_0)$ denote the following Banach space
\begin{align}
    \ell^p(\mathbb{N}_0) \bydef \left\{U = (u_n)_{n \in \mathbb{N}_0}: ~ \|U\|_p \bydef \left( \sum_{n \in \mathbb{N}_0} \alpha_n|u_n|^p\right)^\frac{1}{p} < \infty \right\}.
\end{align}
In the above, note that $U$ is a sequence of scalars, and not a sequence of vectors.
In particular, $\ell^2(\mathbb{N}_0)$ is an Hilbert space associated to its natural inner product $(\cdot, \cdot)_2$ given by
\[
(U,V)_2 \bydef \sum_{n \in \mathbb{N}_0} \alpha_n u_n v_n^*
\]
for all $U = (u_n)_{n \in \mathbb{N}_0}, V = (v_n)_{n \in \mathbb{N}_0} \in \ell^2(\mathbb{N}_0)$ and $*$ denotes complex conjugation.
Then, we define the Banach space $\ell^p_{e}$  as 
\begin{align}
   \ell^p_{e} \bydef \ell^p(\mathbb{N}_0) \times \ell^p(\mathbb{N}_0), ~ \text{ with norm } ~ \|\mathbf{U}\|_{p} = (\|U_1\|_{p}^p + \|U_2\|_{p}^p)^{\frac{1}{p}}.
\end{align}
 For a bounded operator $K : \ell^2_{e} \to \ell^2_{e}$ (resp. $K : \ell^2(\mathbb{N}_0) \to \ell^2(\mathbb{N}_0)$), $K^*$ denotes the adjoint of $K$ in $\ell^2_{e}$ (resp. $\ell^2(\mathbb{N}_0)$). Now, similarly as what is achieved \citep{gs_cadiot_blanco}, we  define $\gamma~:~L^2_{e}(\mathbb{R}) \to \ell^2(\mathbb{N}_0)$ as
\begin{align}
    \left(\gamma(u)\right)_n \bydef  \frac{1}{|\om|}\int_\om u(x) e^{-2\pi i \tilde{n}\cdot x}dx\label{single_gamma}
\end{align}
for all $n \in \mathbb{N}_0$ and all $u \in L^2_{e}(\R)$. Similarly, we define $\gamma^\dagger : \ell^2(\mathbb{N}_0) \to L^2_{e}(\mathbb{R})$ as 
\begin{align}
    \gamma^\dagger\left(U\right)(x) \bydef \cha(x) \sum_{n \in \mathbb{N}_0} \alpha_n u_n \cos(2\pi \tilde{m} x)\label{single_gamma_dagger}
\end{align}
for all $x \in \mathbb{R}$ and all $U =\left(u_n\right)_{n \in \mathbb{N}_0} \in \ell^2(\mathbb{N}_0)$, where $\cha$ is the characteristic function on $\om$. We now introduce a similar notation for functions in the product space $L^2_{e}$. In particular, define $\bgam : L^2_{e} \to \ell^2_{e}$ as
\begin{align}
\bgam(\mathbf{u}) = (\gamma(u_1),\gamma(u_2))\label{def : gamma}
\end{align}
and 
$\bgam^\dagger : \ell^2_{e} \to L^2_{e}$ as
\begin{align}
\bgam^\dagger(\mathbf{U}) = (\gamma^\dagger(U_1),\gamma^\dagger(U_2)).\label{def : gamma dagger}
\end{align}
More specifically, given $\mathbf{u} \in L^2_{e}$, $\bgam(\mathbf{u})$ represents the Fourier coefficients indexed on $\mathbb{N}_0$ of the restriction of $\mathbf{u}$ on $\om$. Conversely, given a sequence $\mathbf{U}\in \ell^2_{e}$, $\bgam^\dagger\left(\mathbf{U}\right)$  is the function representation of $\mathbf{U}$ in $L^2_{e}.$ In particular, notice that $\bgam^\dagger\left(\mathbf{U}\right)(x) =0$ for all $x \notin \om.$  

Then, recall similar notations from \citep{unbounded_domain_cadiot}
\begin{align}
    L^2_{e,\om} \bydef \left\{\mathbf{u} \in L^2_{e} : \text{supp}(\mathbf{u}) \subset \overline{\om} \right\}~~ \text{ and } ~~
   \mathcal{H}_{e,\om} \bydef \left\{\mathbf{u} \in \mathcal{H}_{e} : \text{supp}(\mathbf{u}) \subset \overline{\om} \right\}.
\end{align}
Moreover, define $\mathcal{B}(L^2_{e})$ (respectively $\mathcal{B}(\ell^2_{e})$) as the space of bounded linear operators on $L^2_{e}$ (respectively $\ell^2_{e}$) and denote by $\mathcal{B}_\om(L^2_{e})$ the following subspace of $\mathcal{B}(L^2_{e})$
\begin{equation}\label{def : Bomega}
    \mathcal{B}_\om(L^2_{e}) \bydef \{\mathbb{K} \in \mathcal{B}(L^2_{e}) :  \mathbb{K} = \cha \mathbb{K} \cha\}.
\end{equation}
Finally, define $\bGam : \mathcal{B}(L^2_{e}) \to \mathcal{B}(\ell^2_{e})$ and $\bGam^\dagger : \mathcal{B}(\ell^2_{e}) \to \mathcal{B}(L^2_{e})$ as follows
\begin{equation}\label{def : Gamma and Gamma dagger}
    \bGam(\mathbb{K}) \bydef \bgam \mathbb{K} \bgam^\dagger ~~ \text{ and } ~~  \bGam^\dagger(K) \bydef \bgam^\dagger {K} \bgam 
\end{equation}
for all $\mathbb{K} \in \mathcal{B}(L^2_{e})$ and all $K \in \mathcal{B}(\ell^2_{e}).$

The maps defined above in \eqref{def : gamma}, \eqref{def : gamma dagger} and \eqref{def : Gamma and Gamma dagger} are fundamental in our analysis as they allow to pass from the problem on $\mathbb{R}$ to the one in $\ell^2_{e}$ and vice-versa. Furthermore, the following lemma, provides that this passage is an isometric isomorphism when restricted to the relevant spaces.

\begin{lemma}\label{lem : gamma and Gamma properties}
    The map $\sqrt{|\om|} \bgam : L^2_{e,\om} \to \ell^2_{e}$ (respectively $\bGam : \mathcal{B}_\om(L^2_{e}) \to \mathcal{B}(\ell^2_{e})$) is an isometric isomorphism whose inverse is given by $\frac{1}{\sqrt{|\om|}} \bgam^\dagger : \ell^2_{e} \to L^2_{e,\om}$ (respectively $\bGam^\dagger :   \mathcal{B}(\ell^2_{e}) \to \mathcal{B}_\om(L^2_{e})$). In particular,
    \begin{align}\label{eq : parseval's identity}
        \|\mathbf{u}\|_2 = \sqrt{\om}\|\mathbf{U}\|_2 \text{ and } \|\mathbb{K}\|_2 = \|K\|_2
    \end{align}
    for all $\mathbf{u} \in L^2_{e,\om}$ and $\mathbb{K} \in \mathcal{B}_\om(L^2_{e})$ where $\mathbf{U} \bydef \bgam(\mathbf{u})$ and $K \bydef \bGam(\mathbb{K})$.
\end{lemma}
\begin{proof}
The proof is obtained following similar steps as the ones of Lemma 3.2 in \citep{unbounded_domain_cadiot}.
\end{proof}
The above lemma not only provides a one-to-one correspondence between the elements in $L^2_{e,\om}$ (respectively $\mathcal{B}_\om(L^2_{e})$) and the ones in $\ell^2_{e}$ (respectively $\mathcal{B}(\ell^2_{e})$) but it also provides an identity on norms. This property is essential in our construction of an approximate inverse  in Section \ref{sec : A}.

Now, we define the Hilbert space $\mathscr{h}$ as 
\begin{align*}
    \mathscr{h} \bydef \left\{ \mathbf{U}  \in \ell^2_{e} \text{ such that } \|\mathbf{U}\|_{\mathscr{h}} < \infty \right\}
\end{align*} associated to its inner product $(\cdot,\cdot)_{\mathscr{h}}$ and norm $\|\cdot\|_{\mathscr{h}}$ defined as 
\begin{align*}
    (\mathbf{U},\mathbf{V})_{\mathscr{h}} \bydef \sum_{n \in \mathbb{N}_0}  \alpha_n\left(l(\tilde{n})\mathbf{u}_n\right)\cdot \left(l(\tilde{n})\overline{\mathbf{v}_n}\right),~
    \|\mathbf{U}\|_{\mathscr{h}} \bydef \sqrt{(\mathbf{U},\mathbf{U})_{\mathscr{h}}}
\end{align*}
for all $\mathbf{U}=(\mathbf{u}_n)_{n \in \mathbb{N}_0}, \mathbf{V}=(\mathbf{v}_n)_{n \in \mathbb{N}_0} \in \mathscr{h}$ and $\alpha_n$ is defined as in \eqref{def : alpha_n}.
Denote by $l : \mathscr{h} \to \ell^2_{e}$ and $g : \mathscr{h} \to \ell^2_{e}$ the Fourier coefficients representation of $\mathbb{l}$ and $\mathbb{g}$ respectively. More specifically,
\begin{align*}
    &l \bydef \begin{bmatrix}
        l_{11} & l_{12}\\
        l_{21} & l_{22}
    \end{bmatrix},~g(\bU) \bydef \begin{bmatrix}
       U_1 * U_1 * U_2 + \nu_4 U_1 * U_1 + \nu_5 U_1 * U_2\\
        -U_1 * U_1 * U_2 - \nu_4 U_1 * U_1 - \nu_5 U_1 * U_2
    \end{bmatrix}
\end{align*}
for all $\bU \in \mathscr{h}.$ For all $(i,j) \in \{(1,1),(1,2),(2,1),(2,2)\}$, 
$l_{ij}$ is an infinite diagonal matrix with coefficients $\left(\mathscr{l}_{ij}(\tilde{n})\right)_{n\in \mathbb{N}_0}$ on its diagonal. Moreover $U*V \bydef \gamma(\gamma^\dagger(U)\gamma^\dagger(V))$ is defined as the discrete convolution (under even symmetry). In particular, notice that Young's inequality for convolution is applicable 
\begin{align}
    \|U*V\|_{2} \leq \|U\|_{2} \|V\|_{1}\label{young_inequality}
\end{align}
for all $U \in \ell^2(\mathbb{N}_0), V \in \ell^1(\mathbb{N}_0)$.
Furthermore, using the definition of $l$, notice that \[\|\mathbf{U}\|_{\mathscr{h}} = \|l\mathbf{U}\|_{2}.\] Finally, we define $f(\mathbf{U}) \bydef l\mathbf{U} + g(\mathbf{U})$ and introduce  
\begin{equation}\label{eq : F(U)=0 in X^l_e}
    f(\mathbf{U}) =0 ~~ \text{ and } ~~ \mathbf{U} \in \mathscr{h}
\end{equation}
as the periodic equivalent on $\om$ of \eqref{eq : f(u)=0 on H^l_e}.
\section{Setting up the Computer-Assisted Approach for Localized Patterns}\label{sec : patterns}\
The goal of this section is to set-up a Newton-Kantorovich approach for proving the existence of solutions to \eqref{eq : f(u)=0 on H^l_e} thanks to a fixed-point argument. We  state the necessary theorem  to perform such an analysis, and construct the related objects to apply it.
\subsection{Radii-Polynomial Theorem}\label{sec : radii poly theorem}
 Given $\overline{\mathbf{u}} \in \mathcal{H}_{e}$, an approximate solution to \eqref{eq : f(u)=0 on H^l_e}, and $\mathbb{A} : L^2_{e} \to \mathcal{H}_{e}$, an approximate inverse to $D\mathbb{f}(\overline{\mathbf{u}})$, we want to prove that there exists $r>0$ such that $\mathbb{T} : \overline{B_r(\overline{\mathbf{u}})} \to \overline{B_r(\overline{\mathbf{u}})}$ given by
\[
\mathbb{T}(\mathbf{u}) \bydef \mathbf{u} - \mathbb{A}\mathbb{f}(\mathbf{u})
\]
is well-defined and a contraction. In order to determine a possible value for $r>0$ that would provide the contraction, we wish to use a Radii-Polynomial theorem.  In particular, we  build $\mathbb{A} : L_{e}^2 \to \mathcal{H}_{e}$, $\mathcal{Y}_0, \mathcal{Z}_1 >0$ and  $\mathcal{Z}_2 : (0, \infty) \to [0,\infty)$ in such a way that the hypotheses of the following theorem are satisfied.
\begin{theorem}\label{th: radii polynomial}
Let $\mathbb{A} : L_{e}^2 \to \mathcal{H}_{e}$ be a bounded linear operator. Moreover, let $\mathcal{Y}_0, \mathcal{Z}_1$ be non-negative constants and let  $\mathcal{Z}_2 : (0, \infty) \to [0,\infty)$ be a non-negative function  such that for all $r>0$
  \begin{align}\label{eq: definition Y0 Z1 Z2}
    \|\mathbb{A}\mathbb{f}(\overline{\mathbf{u}})\|_{\mathcal{H}} \leq &\mathcal{Y}_0\\
    \|I - \mathbb{A}D\mathbb{f}(\overline{\mathbf{u}})\|_{\mathcal{H}} \leq &\mathcal{Z}_1\\
    \|\mathbb{A}\left({D}\mathbb{f}(\overline{\mathbf{u}} + \mathbf{h}) - D\mathbb{f}(\overline{\mathbf{u}})\right)\|_{\mathcal{H}} \leq &\mathcal{Z}_2(r)r, ~~ \text{for all } \mathbf{h} \in B_r(0)
\end{align}  
If there exists $r_0>0$ such that
\begin{equation}\label{condition radii polynomial}
    \frac{1}{2}\mathcal{Z}_2(r_0)r_0^2 - (1-\mathcal{Z}_1)r_0 + \mathcal{Y}_0 <0, \ and \ \mathcal{Z}_1 + \mathcal{Z}_2(r_0)r_0 < 1 
 \end{equation}
then there exists a unique $\tilde{\mathbf{u}} \in \overline{B_{r_0}(\overline{\mathbf{u}})} \subset \mathcal{H}_{e}$ such that $\mathbb{f}(\tilde{\mathbf{u}})=0$, where $B_{r_0}(\overline{\mathbf{u}})$ is the open ball of radius $r_0$ in $\mathcal{H}_{e}$ and centered at $\overline{\mathbf{u}}$. 
\end{theorem}
\begin{proof}
The proof can be found in \citep{gs_cadiot_blanco}.
\end{proof}
In order to apply Theorem \ref{th: radii polynomial}, we need to construct explicitly $\overline{\mathbf{u}} \in \mathcal{H}_e$ and $\mathbb{A} : \mathcal{H}_e \to L^2_e$. These are the topics of the next sections.
\subsection{Construction of \texorpdfstring{$\overline{\mathbf{u}}$}{mbfu0}}\label{sec : u0}
In this section, we discuss the  construction of $\overline{\mathbf{u}}$, which is an approximate solution to \eqref{eq : f(u)=0 on H^l_e}. This is generally a challenging problem. In the specific case of \eqref{eq:original_system}, there is a rich variety of work on computing numerical solutions to reaction diffusion systems (e.g \citep{alsaadi_spikes, AlSaadi_unifying_framework, Champneys_Bistability}). Our approach relies on $\overline{\mathbf{u}}$ being constructed numerically on $\om = (-d,d)$ thanks to its Fourier coefficients representation. Fix $N \in \mathbb{N}$ to be the size of our numerical approximation for linear operators (i.e. matrices) and $N_0 \in \mathbb{N}$ to be the one of our Fourier coefficients approximations (i.e. vectors). Now, given $\mathcal{N} \in \mathbb{N}$, let us introduce the following projection operators 
 \begin{align}
 \nonumber
    (\Pi^{\leq\mathcal{N}}(V))_n  =  \begin{cases}
          v_n,  & n \in I^{\mathcal{N}} \\
              0, &n \notin I^{\mathcal{N}}
    \end{cases} ~~ \text{ and } ~~
     (\Pi^{>\mathcal{N}}(V))_n  =  \begin{cases}
          0,  & n \in I^{\mathcal{N}} \\
              v_n, &n \notin I^{\mathcal{N}}
    \end{cases}\label{def : piN and pisubN}
 \end{align}
where $I^{\mathcal{N}} \bydef \{n \in \mathbb{N}_0, ~ n \leq \mathcal{N}\}$ for all $V = (v_n)_{n \in  \mathbb{N}_0} \in \ell^2_e.$ Then, we define
\begin{align}
    (\bpi^{\leq\mathcal{N}}\mathbf{U})_n \bydef ((\Pi^{\leq\mathcal{N}}U_1)_n, (\Pi^{\leq\mathcal{N}}U_2)_n) ~~\text{and}~~(\bpi^{>\mathcal{N}}\mathbf{U})_n \bydef ((\Pi^{>\mathcal{N}}U_1)_n, (\Pi^{>\mathcal{N}}U_2)_n)
\end{align}
for all $\mathbf{U} = \mathbf{u}_{n \in \mathbb{N}_0} \in \ell^2_{e}$.
 To obtain a numerical approximation of a pattern, we looked at the numerical solutions found in \citep{AlSaadi_unifying_framework, Champneys_Bistability} and tried to recover them as a Fourier series representation using guesses of the form:
 \begin{equation*}
    \begin{cases}
        u_{1,0}(x) = \gamma_1 \frac{\beta_1}{\cosh(\alpha_1x)} + \lambda_1 \\
        u_{2,0}(x) = \gamma_2 \frac{\beta_2}{\cosh(\alpha_2x)} + \lambda_2
    \end{cases}
\end{equation*}

Then we tuned each parameter to replicate the localized solutions. $\alpha_i, i = 1,2$ controls the stiffness of each spike, and $\beta_i$ their amplitude. Once we found a promising candidate, we applied Newton's method to obtain a better approximate solution of \eqref{eq : F(U)=0 in X^l_e}. Then we applied the change of variable introduced in Section~\ref{sec : introduction} to get a candidate solution of \eqref{eq : gen_model}. The solution in Figure \ref{fig : bruss pattern} was obtained through numerical continuation of a candidate in the Brusselator equation, using the AUTO continuation software \citep{auto_article}.
From here, we compute a cosine Fourier sequence approximation of $\mathbf{u}_0 = (\bar{u}_{1,0}, \bar{u}_{2,0})$, $\mbf{U}_0 = \gamma(\bar{\mbf{u}}_0) \in \ell^2_e$. We note that $\mbf{U}_0$ satisfies the identity $\mbf{U}_0 = \bpi^{\leq N_0}\mbf{U}_0$ as it is a Fourier sequence with finitely many coefficients. \\
At this point, we now have a vector representation on $\Omega_0$, $\mbf{u}_0 = \mathbf{\gamma}^\dagger(\mbf{U}_0)$ that we extend by $0$ to have $\mbf{u}_0 \in L^2_e$; however, such a function is not necessarily in $\mathcal{H}_e$. To ensure $\mbf{u}_0 \in \mathcal{H}_e$, we need to impose that $\mbf{u}_0$ and some of its derivatives to vanish at $\pm d$. 

In fact, note that $\mathcal{H}_e$ contains the even functions of $H^2(\R) \times H^2(\R)$. Consequently, in order to have $\mbf{u}_0 \in \mathcal{H}_e$, we simply need 
\begin{align*}
    \mbf{u}_0(\pm d) = \partial_x \mbf{u}_0(\pm d) = 0.
\end{align*}
Note that $\partial_x \mbf{u}_0(\pm d) = 0$ is automatically satisfied since $\mbf{u}_0(\pm d)$ admits a cosine series representation on $\om$. Then, using our Fourier representation of $\mbf{u}_0$, we have that for $i = 1,2$
\begin{equation*}
    u_{i,0}(d) = u_{i,0}(-d) = \sum_{n=0}^{N_0} (U_{i, 0})_n \cos\left(\frac{n\pi x}{d}\right) = \sum_{n=0}^{N_0} (U_{i, 0})_n (-1)^n
\end{equation*}
We can ensure that $u_{i,0}(d) = 0$ by setting $(U_{i,0})_0 = \sum_{n=1}^{N_0} (U_{i,0})_n(-1)^{n+1}$. Hence we can define the operator $\mathcal{P}: \ell^2_e \rightarrow \ell_e^2$,
\begin{equation*}
    (\mathcal{P}V)_n = \begin{cases}
        \sum_{n=1}^{N_0} (V)_n(-1)^{n+1} & \text{if } n = 0 \\
        (V)_n & \text{if } n \neq 0
    \end{cases}\label{def : operator P}
\end{equation*}
Moreover, we let $\mathbfcal{P}\mathbf{V} = (\mathcal{P}V_1, \mathcal{P}V_2)$. This allows us to define 
$$\overline{\mathbf{U}} = \mathbfcal{P}\mathbf{U}_0 \text{ and } \overline{\mathbf{u}} = \gamma^\dagger(\overline{\mathbf{U}}).$$ By construction of $\mathcal{P}$, one can verify that $\overline{\mathbf{u}}(\pm d)=0$, and consequently $\overline{\mathbf{u}} \in \mathcal{H}_e$. Moreover, we have that $\overline{\mathbf{U}} = \Pi^{\leq N_0}\overline{\mathbf{U}}$ by construction. Overall, we have constructed our approximate solution $\overline{\mathbf{u}}$ as 
\begin{equation}\label{eq : approximate solution}
    \overline{\mathbf{u}} = \gamma^\dagger(\overline{\mathbf{U}}) \in \mathcal{H}_e, ~  \text{ and } \overline{\mathbf{U}} = \Pi^{\leq N_0}\overline{\mathbf{U}}.
\end{equation}
\subsection{The Operator \texorpdfstring{$\mathbb{A}$}{A}}\label{sec : A}
In this section, we focus our attention to the operator $\mbb{A}: L^2_e \rightarrow \mathcal{H}_e$ required by Theorem \ref{th: radii polynomial}. $\mbb{A}$ is an approximate inverse of $D\mbb{f}(\overline{\mathbf{u}})$. We now present its construction following a similar argument as in \citep{gs_cadiot_blanco}. First we observe that $\mbb{L}: \mathcal{H}_e \rightarrow L^2_e$ is an isometric isomorphism. Therefore building $\mbb{A}: L^2_e \rightarrow \mathcal{H}_e$ is equivalent to building $\mbb{B}: L^2_e \rightarrow L^2_e$ approximating the inverse of $D\mbb{F}(\overline{\mathbf{u}})\mbb{L}^{-1}$, and setting $\mbb{A} \bydef \mbb{L}^{-1}\mbb{B}$. We will then construct $\mbb{B}$ using Lemma \ref{lem : gamma and Gamma properties} and a bounded linear operator on Fourier coefficients. 
First we construct $B^N$ a numerical approximate inverse to $\bpi^{\leq N}Df(\overline{\mathbf{U}})\mathscr{l}^{-1}\bpi^{\leq N}$. We choose $B^N$ satisfying $B^N = \bpi^{\leq N} B^N \bpi^{\leq N}$, hence it admits a matrix representation. We now describe further its construction. Due to the form of $g$, we obtain
{\small\begin{align}
    &Df(\overline{\mathbf{U}})l^{-1} = 
     \begin{bmatrix}
        I_d + Dg_{11}(\overline{\mathbf{U}})l_{\text{den}}^{-1}l_{22} -Dg_{12}(\overline{\mathbf{U}})l_{\text{den}}^{-1}l_{21} & -Dg_{11}(\overline{\mathbf{U}})l_{\text{den}}^{-1}l_{12} + Dg_{12}(\overline{\mathbf{U}})l_{\text{den}}^{-1}l_{11} \\ - Dg_{11}(\overline{\mbf{U}}) l_{\text{den}}^{-1} l_{22} + Dg_{12}(\overline{\mbf{U}}) l_{\text{den}}^{-1}l_{21} & I_d +Dg_{11}(\overline{\mathbf{U}})l_{\text{den}}^{-1}l_{12} - Dg_{12}(\overline{\mathbf{U}})l_{\text{den}}^{-1}l_{11}
        \end{bmatrix}\label{eq: B^N matrix}
\end{align}}
where 
\begin{align} 
l_{\text{den}} \bydef  l_{11}l_{22} - l_{12}l_{21}.\label{def : l_den}
\end{align}  Notice that since \eqref{eq: B^N matrix} is a full matrix, we will choose $B^N$ to be a full matrix as well.  More specifically, we construct $B^N$ and $B$ as follows
\begin{align} 
B^N = \begin{bmatrix}
    B_{11}^N & B_{12}^N \\  B_{21}^N & B_{22}^N
\end{bmatrix},~
B = \bpi^{>N} + B^N \bydef \begin{bmatrix}
    B_{11} & B_{12} \\ B_{21} & B_{22}
\end{bmatrix}\label{def : B_finite_infinite}
\end{align}
with $B_{11}^N, B_{12}^N, B_{21}^N,$ and $B_{22}^N$ constructed numerically and satisfying $B_{1j}^N = \Pi^{\leq N} B_{1j}^N\Pi^{\leq N}$. Moreover, we define $B_{11} = B_{11}^N + \Pi^{>N}$, $B_{12} = B_{12}^N, B_{21} = B_{21}^N, B_{22} = B_{22}^N + \Pi^{>N}$ for convenience.
Finally, we get $\mbb{B}: L_e^2 \rightarrow L_e^2$ and $\mbb{A}: L_e^2 \rightarrow \mathcal{H}_e$ as 
\begin{equation}\label{eq:def_A}
    \mbb{B}\bydef \begin{bmatrix}
        \mbb{1}_{\mathbb{R}\backslash \Omega_0} & 0 \\ 0 & \mbb{1}_{\mathbb{R} \backslash \Omega_0}
    \end{bmatrix} + \mbf{\Gamma}^{\dagger}(B)   = \begin{bmatrix}
        \mbb{B}_{11} & \mbb{B}_{12} \\ \mbb{B}_{21} & \mbb{B}_{22}
    \end{bmatrix},~ \mbb{A}\bydef \mbb{l}^{-1}\mbb{B}.
\end{equation}
where $\mbb{B}_{ij} = \Gamma^{\dagger}(B_{ij})$ for $i,j \in \{1,2\}$. We will justify later on such a choice for $\mathbb{A}$ by computing explicitly the defect $\|I - \mathbb{A}D\mathbb{f}(\bar{\mathbf{u}})\|_{\mathcal{H}}$. In fact, we already obtain the following result, allowing to compute the operator norm of $\mathbb{A}$.
\begin{lemma}\label{lem:bound_B}
    Let $\mbb{A}: L^2_e \rightarrow \mathcal{H}_e$ be as in \eqref{eq:def_A}. Then,
    \begin{equation}\label{eq:B_bound}
        \|\mbb{A}\|_{2, \mathcal{H}} = \|\mbb{B}\|_2 = \max\{1, \|B^N\|_2\}.
    \end{equation}
\end{lemma}
\begin{proof}
The proof can be found in \citep{unbounded_domain_cadiot}.
\end{proof}
With $\mathbb{A}$ now built, we can begin to compute the bounds necessary to apply Theorem \ref{th: radii polynomial}. This will be the goal of the next section.
\section{Computing the Bounds for Localized Patterns}\label{sec : Bounds of patterns}
In this section, we provide explicit formulas for the bounds $\mathcal{Y}_0,\mathcal{Z}_1,$ and $\mathcal{Z}_2$ in Theorem \ref{th: radii polynomial}.
Before we begin the computation of the bounds, we recall some preliminary notations and results from \citep{gs_cadiot_blanco}. First, given $u \in L^\infty(\mathbb{R}^2)$,  denote by 
\begin{align}\label{def : multiplication operator}
    \mathbb{u}  \colon L^2(\mathbb{R}^2) &\to L^2(\mathbb{R}^2)\\
    v &\mapsto uv
\end{align}
 the multiplication operator associated to $u$.
Similarly, given $U= (u_n)_{n \in \mathbb{N}_0} \in \ell^1(\mathbb{N}_0)$,  denote by
\begin{align}\label{def : discrete conv operator}
    \mathbb{U} : \ell^2(\mathbb{N}_0) &\to \ell^2(\mathbb{N}_0) \\
    V &\mapsto  U * V
\end{align}
 the discrete convolution operator associated to $U$. We now define the following map $\varphi: \mathbb{R}^4 \to \mathbb{R}$ as 
\begin{equation}\label{definition_of_phi}
    \varphi(x_1, x_2, x_3, x_4) \bydef \min\left\{\max \{x_1, x_4\} + \max \{x_2, x_3\},~ \sqrt{x_1^2 + x_2^2 + x_3^2 + x_4^2}\right\}
\end{equation}
for a given $(x_1,x_2,x_3,x_4) \in \mathbb{R}^4$.
For the purpose of our analysis, we recall Lemma 4.2 from \citep{gs_cadiot_blanco} which allows us to estimate the norm of an operator by block in terms of the norms of its individual blocks.
\begin{lemma}\label{lem : full_matrix_estimate}
Let $K_1,K_2,K_3,K_4 \in \mathcal{B}(X)$ for $X \in \left\{L_e^2(\mathbb{R}^2),\ell_e^2\right\}$ where $\mathcal{B}(X)$ is the set of bounded linear operators on $X$. Then,
    \begin{align}  
    \nonumber &\left\| \begin{bmatrix}
        K_1 & K_2 \\ 
        K_3 & K_4 \\
    \end{bmatrix} \right\|_2 \leq \varphi(\left\|K_1\right\|_{2},\left\|K_2\right\|_{2},\left\|K_3\right\|_{2},\left\|K_4\right\|_{2}),
    \end{align}
    where $\varphi: \mathbb{R}^4 \to \mathbb{R}$ is defined as in \eqref{definition_of_phi}. 
\end{lemma}
\begin{proof}
The proof can be found in \citep{gs_cadiot_blanco}.
\end{proof}

Now, we provide explicit formulas for the bounds of Theorem \ref{th: radii polynomial} in the following list of lemmas. In particular, each formula relies on finite dimensional computations (that is vector or matrix norms), which can be rigorously evaluated thanks to computer-assisted techniques (cf. \citep{julia_blanco_fassler}). This justifies how specific choices for the approximate objects $\bar{\mathbf{u}}$ and $\mathbb{A}$.
We begin with the $\mathcal{Y}_0$ bound. In the following lemma, we describe the necessary steps to estimate this bound required for Theorem \ref{th: radii polynomial}. Its proof will involve splitting the estimate using the projection operators $\bpi^{\leq N}$ and $\bpi^{> N}$. We then rewrite the bounds in such a way that they all involve finite quantities which can be rigorously evaluated on the computer.
\begin{lemma}\label{lem : bound Y_0}
Let $\mathcal{Y}_0 >0$ be defined as  
\begin{equation}
    \mathcal{Y}_0 \bydef |\om|^{\frac{1}{2}}\left(\|B^Nf(\overline{\mathbf{U}})\|_2^2 + \|(\bpi^{\leq N_0}-\bpi^{\leq N})l\overline{\mathbf{U}} + (\bpi^{\leq 3N_0}-\bpi^{\leq N}) g(\overline{\mathbf{U}})\|_2^2 \right)^{\frac{1}{2}}.
\end{equation}
Then $\|\mathbb{A}\mathbb{f}(\overline{\mathbf{u}})\|_{\mathcal{H}} \leq \mathcal{Y}_0.$
\end{lemma}

\begin{proof}
We follow the same steps as those of Lemma 4.3 in \citep{gs_cadiot_blanco}. In particular, we combine \eqref{def : definition of the norm and inner product Hl},the operator $\mathbb{A}$, and the properties of $\bgam$ to get
\begin{align}
    \|\mathbb{A}\mathbb{f}(\overline{\mathbf{u}})\|_{\mathcal{H}} &= \|\mathbb{B}\mathbb{f}(\overline{\mathbf{u}})\|_{2} = |\om|^{\frac{1}{2}}\|\left(Bf(\overline{\mathbf{U}})\right)\|_2.
\end{align}
We then obtain the result using the steps of Lemma 4.11 in \citep{unbounded_domain_cadiot}.
\end{proof}
%
%
%
We now estimate the $\mathcal{Z}_2(r)$ bound. We exploit the specific structure of $\mathbb{B}$ to obtain a similar result to the one obtained in \citep{gs_cadiot_blanco}. We use properties of multilinear operators to write the estimate in terms of the Hessian, and then rely on Parseval's identity to obtain quantities which can be estimated on the computer.
\begin{lemma}\label{lem : Z2 patterns}
    Let $\kappa > 0$ be defined as in Lemma \ref{corr : banach algebra}, Moreover, define $q_1, q_2 \in L^\infty(\mbb{R})\cap L^2(\mbb{R})$  as
    \begin{align*}
        &q_1 \bydef 2\overline{u}_2 + 2\nu_4 \mathbb{1}_{\om},~
        q_2 \bydef 2\overline{u}_1 + \nu_5 \mathbb{1}_{\om}
    \end{align*}
    and their associated Fourier Coefficients $Q_1 = \gamma(q_1), Q_2 = \gamma(q_2)$.
    Let $\mathcal{Z}_2 : [0, \infty) \rightarrow [0, \infty)$ be defined as
    $$\mathcal{Z}_2(r) \bydef \sqrt{5}\kappa\sqrt{\varphi\left(\mathcal{Z}_{2,1},\mathcal{Z}_{2,2},\mathcal{Z}_{2,2},\mathcal{Z}_{2,3}\right)} + 3\kappa \|\mathscr{l}^{-1}\|_{\mathcal{M}_2}\max\left\{1, \left\| \begin{bmatrix}
        B_{11}^N - B_{12}^N \\
        B_{21}^N - B_{22}^N
    \end{bmatrix}\right\|_{2}\right\}r$$
    for all $r\geq 0$, where
    \begin{align*}
        &\mathcal{Z}_{2,1} \bydef  \left\|\ \begin{bmatrix}
            B_{11}^N - B_{12}^N \\
            B_{21}^N - B_{22}^N
        \end{bmatrix}\left(\mbb{Q}_1^2 + \mbb{Q}_2^2\right) \begin{bmatrix}
            (B_{11}^N - B_{12}^N)^* & (B_{21}^N - B_{22}^N)^* 
        \end{bmatrix} \right\|_2,\\
        &\mathcal{Z}_{2,2} \bydef \sqrt{\left\|  \begin{bmatrix}
            B_{11}^N - B_{12}^N \\
            B_{21}^N - B_{22}^N
        \end{bmatrix}\left(\mbb{Q}_1^2 + \mbb{Q}_2^2\right) \Pi^{>N} \left(\mbb{Q}_1^2 + \mbb{Q}_2^2\right)\begin{bmatrix}
            (B_{11}^N - B_{12}^N)^* & (B_{21}^N - B_{22}^N)^* 
        \end{bmatrix} \right|_{2}}, \\
    &\mathcal{Z}_{2,3} \bydef \| Q_1^2 + Q_2^2 \|_1
    \end{align*}
    and $\varphi$ is defined in Lemma \ref{lem : full_matrix_estimate}.
    Then $\|\mathbb{A}(D\mathbb{f}(\mathbf{h} + \bar{\mathbf{u}}) - D\mathbb{f}(\bar{\mathbf{u}})) \|_{\mathcal{H}} \leq \mathcal{Z}_2(r)r$ for all $r>0$ and $\mathbf{h} \in \overline{B_r(0)} \subset H_e^l$
\end{lemma}

\begin{proof}
Let $\mbf{h} \in \overline{B_r(0)}, \mbf{z} \in \overline{B_1(0)} \subset \mathcal{H}_e$. 
Let $g_3 : \left(\mathcal{H}_e\right)^3 \rightarrow \mathcal{H}_e$ be defined as the following multilinear operator
    \begin{align}
        g_3(\mbf{u}, \mbf{v}, \mbf{w}) \bydef \begin{bmatrix}
        u_1 v_1 w_2 \\ -u_1 v_1 w_2
    \end{bmatrix} \text{ for all } \mbf{u} = (u_1, u_2), \mbf{v} = (v_1, v_2),\mbf{w} = (w_1, w_2) \in \mathcal{H}_e.\label{def : g3}
    \end{align}
    In fact, notice that $g_3(\mbf{u},\mbf{u},\mbf{u}) = \mbb{g}_3(\mbf{u})$.
    Moreover,  we have the symmetry
       $
        g_3(\mbf{u}, \mbf{v}, \mbf{w}) = g_3(\mbf{v}, \mbf{u}, \mbf{w})$
    for all $\mbf{u}, \mbf{v}, \mbf{w} \in \mathcal{H}_e$.
Following similar steps to those done in Lemma 4.5 of \citep{gs_cadiot_blanco}, we obtain
\begin{align}
    &\|\mathbb{A}(D\mathbb{f}(\mathbf{h} + \bar{\mathbf{u}}) - D\mathbb{f}(\bar{\mathbf{u}})) \|_{\mathcal{H}} \leq \|\mathbb{B} D^2\mathbb{g}(\overline{\mathbf{u}})(\mbf{h},\mbf{z})\|_{2} + \|\mathbb{B}(2g_3(\mbf{z},\mbf{h},\mbf{h}) + g_3(\mbf{h},\mbf{h},\mbf{z}))\|_{2} \\
    &\leq \left\|\mathbb{B}\begin{bmatrix}(\mathbb{q}_1 \mathbb{h}_1+ \mathbb{q}_2 \mathbb{h}_2) z_1 + \mathbb{q}_2 \mathbb{h}_1 z_2 \\ -(\mathbb{q}_1 \mathbb{h}_1+ \mathbb{q}_2 \mathbb{h}_2) z_1 - \mathbb{q}_2 \mathbb{h}_1 z_2 \end{bmatrix}\right\|_{2} +  \left\|\begin{bmatrix}
        \mathbb{B}_{11} - \mathbb{B}_{12} \\
        \mathbb{B}_{21} - \mathbb{B}_{22}
    \end{bmatrix} \begin{bmatrix}
        2 z_1 h_1 h_2 + h_1^2 z_2
    \end{bmatrix}\right\|_{2} 
    \\
    & \leq  \left\|\mathbb{B}\begin{bmatrix}
        \mathbb{q}_1 & \mathbb{q}_2 \\
        -\mathbb{q}_1  & -\mathbb{q}_2
    \end{bmatrix}  \begin{bmatrix}
        \mathbb{h}_1 & \mathbb{h}_2 \\
        \mathbb{h}_2 & \mathbb{h}_1
    \end{bmatrix}\begin{bmatrix}
        z_1 \\ z_2
    \end{bmatrix}\right\|_{2} + \left\| \begin{bmatrix}
        \mathbb{B}_{11} - \mathbb{B}_{12} \\
        \mathbb{B}_{21} - \mathbb{B}_{22}
    \end{bmatrix}\right\|_{2}\| 2z_1 h_1 h_2 + h_1^2 z_2\|_{2}.
    \end{align}
Now, observe that
\begin{align}
    \|2z_1 h_1 h_2 + h_1^2 z_2\|_{2} \leq 2\|z_1 h_2 h_2\|_{2} + \|h_1^2 z_2\|_{2} \leq 3\kappa \|\mathscr{l}^{-1}\|_{\mathcal{M}_2} r^2.
\end{align}
Hence, we obtain
{\footnotesize\begin{align}
    &\left\|\mathbb{B}\begin{bmatrix}
        \mathbb{q}_1 & \mathbb{q}_2 \\
        -\mathbb{q}_1  & -\mathbb{q}_2
    \end{bmatrix}  \begin{bmatrix}
        \mathbb{h}_1 & \mathbb{h}_2 \\
        \mathbb{h}_2 & \mathbb{h}_1
    \end{bmatrix}\begin{bmatrix}
        z_1 \\ z_2
    \end{bmatrix}\right\|_{2} + \left\| \begin{bmatrix}
        \mathbb{B}_{11} - \mathbb{B}_{12} \\
        \mathbb{B}_{21} - \mathbb{B}_{22}
    \end{bmatrix}\right\|_{2}\| 2z_1 h_1 h_2 + h_1^2 z_2\|_{2}
    \\
    &\leq \left\|\begin{bmatrix}
        \mathbb{B}_{11} - \mathbb{B}_{12} \\ \mathbb{B}_{21} - \mathbb{B}_{22}
    \end{bmatrix}\begin{bmatrix}
        \mathbb{q}_1 & \mathbb{q}_2
    \end{bmatrix}\right\|_{2} \left\| \begin{bmatrix}
        \mathbb{h}_1 & \mathbb{h}_2 \\
        \mathbb{h}_2 & \mathbb{h}_1
    \end{bmatrix}\begin{bmatrix}
        z_1 \\ z_2
    \end{bmatrix}\right\|_{2}  + 3\kappa \|\mathscr{l}^{-1}\|_{\mathcal{M}_2}\max\left\{1, \left\| \begin{bmatrix}
        B_{11}^N - B_{12}^N \\
        B_{21}^N - B_{22}^N
    \end{bmatrix}\right\|_{2}\right\} r^2 \\
    &\\
    &\leq \sqrt{5}\kappa \sqrt{\left\|\begin{bmatrix}
        \mathbb{B}_{11} - \mathbb{B}_{12} \\ \mathbb{B}_{21} - \mathbb{B}_{22}
    \end{bmatrix}(\mathbb{q}_1^2 + \mathbb{q}_2^2) \begin{bmatrix}
        (\mathbb{B}_{11}-\mathbb{B}_{12})^* & (\mathbb{B}_{21} - \mathbb{B}_{22})^*
    \end{bmatrix}\right\|_{2}}r  + 3\kappa \|\mathscr{l}^{-1}\|_{\mathcal{M}_2}\max\left\{1, \left\| \begin{bmatrix}
        B_{11}^N - B_{12}^N \\
        B_{21}^N - B_{22}^N
    \end{bmatrix}\right\|_{2}\right\} r^2
\end{align}}
where we also used Lemma \ref{corr : banach algebra} in the last inequality.
Then, using the properties of $\bGam$, we similar steps to those used in Lemma 4.5 of \citep{gs_cadiot_blanco} to estimate
\begin{align}
    \left\|\begin{bmatrix}
        \mathbb{B}_{11} - \mathbb{B}_{12} \\ \mathbb{B}_{21} - \mathbb{B}_{22}
    \end{bmatrix}(\mathbb{q}_1^2 + \mathbb{q}_2^2) \begin{bmatrix}
        (\mathbb{B}_{11}-\mathbb{B}_{12})^* & (\mathbb{B}_{21} - \mathbb{B}_{22})^*
    \end{bmatrix}\right\|_{2} &\leq \varphi(\mathcal{Z}_{2,1},\mathcal{Z}_{2,2},\mathcal{Z}_{2,2},\mathcal{Z}_{2,3})
\end{align}
and conclude the result.
\end{proof}
 \subsection{The Bound \texorpdfstring{$\mathcal{Z}_1$}{Z1}}
For now, we provided explicit formulas for the bounds $\mathcal{Y}_{0}$ and $\mathcal{Z}_2$, and it remains to  estimate $\|I - \mathbb{A}D\mathbb{f}(\bar{\mathbf{u}})\|_{\mathcal{H}} \leq \mathcal{Z}_1$. This bound requires a bit more analysis, and we present its treatment in this section. First let us define $v_1, v_2 \in L^\infty(\mbb{R}) \cap L^2(\mbb{R})$
\begin{align}\label{def : v1v2}
    v_1 \bydef 2\overline{u}_1 \overline{u}_2 + 2\nu_4 \overline{u}_1 + \nu_5 \overline{u}_2 \text{ and }v_2 \bydef \overline{u}_1^2 + \nu_5 \overline{u}_1.
\end{align}
In particular, we can write 
$
    D\mbb{g}(\bbu) = \begin{bmatrix}
    \mbb{v_1} & \mbb{v_2} \\ -\mbb{v}_1 & -\mbb{v}_2    
    \end{bmatrix},$
where $\mbb{v}_1, \mbb{v}_2$ are defined as in \eqref{def : multiplication operator}. Now we define
$$V_1\bydef \gamma(v_1), V_2\bydef \gamma(v_2), V_1^N\bydef \Pi^{\leq N}V_1, V_2^N\bydef \Pi^{\leq N}V_2 \text{ and }Dg^N(\overline{\mathbf{U}})\bydef \begin{bmatrix}
    \mathbb{V}_1^N & \mathbb{V}_2^N \\
    -\mathbb{V}_1^N & -\mathbb{V}_2^N
\end{bmatrix}$$
where $\mbb{V}_1^N, \mbb{V}_2^N$ are the discrete convolution operators associated to $V_1^N, V_2^N$. Finally we define
$$D\mbb{g}^N(\overline{\mathbf{U}})\bydef \Gamma^\dag(Dg^N(\overline{\mathbf{U}})),~~ v_1^N\bydef \gamma^\dag(V_1^N) \text{ and } v_2^N\bydef \Gamma^\dag(V_2^N).$$

This notation allows us to decompose both $D\mbb{g}(\overline{\mathbf{U}})$ and $Dg(\overline{\mathbf{U}})$ using a truncation of size $N$. We now state a lemma for the $\mathcal{Z}_1$ bound This lemma will allow us to split the analysis of the $\mathcal{Z}_1$ into two parts, which we denote $Z_1$ and $\mathcal{Z}_u$. The former is the ``usual" $Z_1$ bound one computes when proving periodic solutions using Theorem \ref{th: radii polynomial} (cf. \citep{van2021spontaneous, period_kuramoto}). The latter is due to the unboundedness of the domain as outlined in \citep{unbounded_domain_cadiot}, and will explicitly depend on the truncated domain size $d$.
\begin{lemma}\label{lem : Z_full_1}
Let $\mathcal{Z}_{u}, Z_1,\mathcal{Z}_1 > 0$ be bounds satisfying
\begin{align} 
&\mathcal{Z}_{u} \geq \|\mathbb{B}D\mathbb{g}^N(\overline{\mathbf{u}})(\mathbb{l}^{-1} - \bGam^\dagger(l^{-1})\|_{2},~
    Z_1 \geq \|I - B(I + Dg^N(\overline{\mathbf{U}})l^{-1})\|_{2} \\
    &\mathcal{Z}_1 \bydef Z_1 +  \mathcal{Z}_{u} +  \max\left\{1, \left\|\begin{bmatrix}
        B_{11}^N - B_{12}^N \\
        B_{21}^N - B_{22}^N
    \end{bmatrix}\right\|_{2}\right\}\|\mathbb{l}^{-1}\|_{2} \sqrt{\|V_1^N - V_1\|_{1}^2 + \|V_2^N - V_2\|_{1}^2}.
\end{align}
Then, it follows that $ \|I -{\mathbb{A}}D\mathbb{f}(\overline{\mathbf{u}})\|_{\mathcal{H}} \leq \mathcal{Z}_1.$
\end{lemma}
\begin{proof}
First, using the definition of the norm $\|\cdot\|_{\mathcal{H}}$ and the triangle inequality, we get
{\footnotesize\begin{align}
     \|I - \mathbb{A}D\mathbb{f}(\overline{\mathbf{u}})\|_{\mathcal{H}} 
     = \|I - \mathbb{B} ( I + D\mathbb{g}(\overline{\mathbf{u}})\mathbb{l}^{-1})\|_2,   
   &\leq \|I - \mathbb{B}(I +D\mathbb{g}^N(\overline{\mathbf{u}})\mathbb{l}^{-1})\|_{2} + \|\mathbb{B}(D\mathbb{g}^N(\overline{\mathbf{u}}) - D\mathbb{g}(\overline{\mathbf{u}}))\mathbb{l}^{-1}\|_{2} \\
   &\leq Z_1 + \mathcal{Z}_u + \left\|\begin{bmatrix}
        \mathbb{B}_{11} - \mathbb{B}_{12} \\
        \mathbb{B}_{21} - \mathbb{B}_{22} 
    \end{bmatrix}\begin{bmatrix}
        \mathbb{V}_1^N - \mathbb{V}_1 & \mathbb{V}_2^N - \mathbb{V}_2
    \end{bmatrix}\mathbb{l}^{-1}\right\|_{2}\label{ineq : first step in Zu}
\end{align}}
where the final step followed from Theorem 3.5 of \citep{unbounded_domain_cadiot} and the definition of $\mathbb{B}$. For the final term in \eqref{ineq : first step in Zu}, we use submultiplicativity to obtain
\begin{align}
     \left\|\begin{bmatrix}
        \mathbb{B}_{11} - \mathbb{B}_{12} \\
        \mathbb{B}_{21} - \mathbb{B}_{22} 
    \end{bmatrix}\begin{bmatrix}
        \mathbb{V}_1^N - \mathbb{V}_1 & \mathbb{V}_2^N - \mathbb{V}_2
    \end{bmatrix}\mathbb{l}^{-1}\right\|_{2} 
    &\leq \left\|\begin{bmatrix}
        \mathbb{B}_{11} - \mathbb{B}_{12} \\
        \mathbb{B}_{21} - \mathbb{B}_{22}
    \end{bmatrix}\right\|_{2}\|\mathbb{l}^{-1}\|_{2} \|\begin{bmatrix}
        \mathbb{V}_1^N - \mathbb{V}_1 & \mathbb{V}_2^N - \mathbb{V}_2
    \end{bmatrix}\|_{2} \\
    &\hspace{-2.3cm}\leq \max\left\{1, \left\|\begin{bmatrix}
        B_{11}^N - B_{12}^N \\
        B_{21}^N - B_{22}^N
    \end{bmatrix}\right\|_{2}\right\}\|\mathbb{l}^{-1}\|_{2} \sqrt{\|V_1^N - V_1\|_{1}^2 + \|V_2^N - V_2\|_{1}^2}
\end{align}
where the last step followed by Parseval's Identity and Lemma 4.6 of \citep{gs_cadiot_blanco}. 
\end{proof}
Now, we must compute the bounds $Z_1$ and $\mathcal{Z}_u$. We will begin with $Z_1$, which as mentioned previously, is the "usual" $Z_1$ bound that one normally computes when proving Periodic solutions. As a result, its estimate only depends on Fourier coefficient computations.
\begin{lemma}\label{lem : Z1 periodic patterns}
    Let
    \begin{align}
        Z_1 \bydef \sqrt{Z_{0}^2 + 2Z_{1,1}^2 + 2Z_{1,2}^2}
    \end{align}
    where
    \begin{align*}
        &Z_{0} \bydef \|\bpi^{\leq2N} - \bpi^{\leq 3N}B(I_d + Dg^N(\overline{\mbf{U}}) l^{-1})\bpi^{\leq2N}\|_{2}\\
        &Z_{1,1} \bydef \left(\frac{l_{22}}{l_{\mathrm{den}}}\right)_{2N} \| V_1^N\|_{1} + \left(\frac{l_{21}}{l_{\mathrm{den}}}\right)_{2N} \|V_2^N\|_{1},~~ 
        Z_{1,2} &\bydef \left(\frac{l_{12}}{l_{\mathrm{den}}}\right)_{2N} \|V_1^N\|_{1} + \left(\frac{l_{11}}{l_{\mathrm{den}}}\right)_{2N} \|V_2^N\|_{1} 
    \end{align*}
    and $\left(\frac{\mathscr{l}_{ij}}{l_{\mathrm{den}}}\right)_{2N} = \max_{n\in \mathbb{Z}\setminus I^{2N}}\frac{\mathscr{l}_{ij} (\tilde n)}{\mathscr{l}_{den}(\tilde{n})}$. Then, it follows that $\|I_d - B(I_d + Dg^N(\overline{U}) l^{-1})\|_{2} \leq Z_1$.
\end{lemma}

\begin{proof}
To begin, we introduce a truncation.
{\footnotesize\begin{align}
    \|I_d - B(I_d + Dg^N(\overline{\mbf{U}}) l^{-1})\|_{2}^2 &\leq \|\bpi^{>2N} - B(I_d + Dg^N(\overline{\mbf{U}}) l^{-1})\bpi^{>2N}\|_{2}^2 + \|\bpi^{\leq2N} - B(I_d + Dg^N(\overline{\mbf{U}}) l^{-1})\bpi^{\leq2N}\|_{2}^2 \\
    &= \|\bpi^{>2N} - B \bpi^{>2N} - BDg^N(\overline{\mbf{U}}) \bpi^{<2N} l^{-1}\|_{2}^2 + Z_{0}^2.
\end{align}}
Now, by definition of $B$, we have $B \bpi^{>2N} = \bpi^{>2N}$ and $B\bpi^{>N} = \bpi^{>N}$. Additionally, since $v_1^N$ and $v_2^N$ are of size $N$, it follows that $Dg^N(\overline{U})\bpi^{>2N} = \bpi^{>N}Dg^N(\overline{U})\bpi^{>2N}$. Hence, we have
\begin{align}
    \|\bpi^{>2N} - B \bpi^{>2N} - BDg^N(\overline{U}) \bpi^{<2N} l^{-1}\|_{2} &= \|\bpi^{>2N} -  \bpi^{>2N} - B\bpi^{>N} Dg^N(\overline{U}) \bpi^{<2N} l^{-1}\|_{2} \\
    &= \|\bpi^{>N} Dg^N(\overline{U}) l^{-1}\bpi^{<2N}\|_{2} \\
    &\hspace{-1cm}= \left\|\begin{bmatrix}
        \mathbb{V}_1^N l_{22} l_{\mathrm{den}}^{-1} - \mathbb{V}_2^N l_{21} l_{\mathrm{den}}^{-1} & -\mathbb{V}_1^N  l_{12} l_{\mathrm{den}}^{-1} + \mathbb{V}_2^N l_{11} l_{\mathrm{den}}^{-1} \\ -\mathbb{V}_1^N l_{22} l_{\mathrm{den}}^{-1} + \mathbb{v}_2^N l_{21} l_{\mathrm{den}}^{-1} & \mathbb{V}_1^N  l_{12} l_{\mathrm{den}}^{-1} - \mathbb{V}_2^N l_{11} l_{\mathrm{den}}^{-1}
    \end{bmatrix}\right\|_{2} \\
    &\hspace{-1.5cm}\leq \sqrt{2}\sqrt{\|\mathbb{V}_1^N l_{22} l_{\mathrm{den}}^{-1} - \mathbb{V}_2^N l_{21} l_{\mathrm{den}}^{-1}\|_{2}^2 + \|-\mathbb{V}_1^N  l_{12} l_{\mathrm{den}}^{-1} + \mathbb{V}_2^N l_{11} l_{\mathrm{den}}^{-1}\|_{2}^2}
\end{align}
where we used the Cauchy-Schwarz inequality. We now examine each term
\begin{align}
    &\|\mathbb{V}_1^N l_{22} l_{\mathrm{den}}^{-1} - \mathbb{v}_2^N l_{21} l_{\mathrm{den}}^{-1}\|_{2} \leq \left(\frac{l_{22}}{l_{\mathrm{den}}}\right)_{2N} \| V_1^N\|_{1} + \left(\frac{l_{21}}{l_{\mathrm{den}}}\right)_{2N} \|V_2^N\|_{1} \bydef Z_{1,1} \\
    &\|-\mathbb{V}_1^N  l_{12} l_{\mathrm{den}}^{-1} + \mathbb{V}_2^N l_{11} l_{\mathrm{den}}^{-1}\|_{2} \leq \left(\frac{l_{12}}{l_{\mathrm{den}}}\right)_{2N} \|V_1^N\|_{1} + \left(\frac{l_{11}}{l_{\mathrm{den}}}\right)_{2N} \|V_2^N\|_{1} \bydef Z_{1,2}.
\end{align}
Therefore, we have
\begin{align}
    \|I_d - B(I_d + Dg^N(\overline{U})l^{-1})\|_{2} \leq \sqrt{Z_{0}^2 + 2Z_{1,1}^2 + 2Z_{1,2}^2} \bydef Z_1
\end{align}
as desired.
\end{proof}
We now work on the bound $\mathcal{Z}_u$. Before doing so, we will derive two preliminary results. The first of which provides the exponential decay of $\mathcal{F}^{-1}(\frac{\mathscr{l}_{ij}}{\mathscr{l}_{den}})(x)$ for $i,j \in \{1,2\}$. This estimate will have a direct influence on the computation of $\mathcal{Z}_u$.
\begin{lemma}\label{lem : ift_L_inverse}
    Assume that Assumption \ref{assumption : fourier} holds, then we have the following disjunction : 
    \begin{enumerate}
        \item If $(\rho\nu_2 - \nu_1)^2 + 4\rho\nu_2(\nu_1 + \nu_3) < 0$, then  define $z_1, z_2$ as 
        \small{
        \begin{align}\label{eq: def z1 z2 and y}
        z_1 \bydef 2\pi i y \text{ and } z_2 \bydef -2\pi i \bar{y}, \text{ where }
            y \bydef \frac{1}{2\pi} \left(\frac{\nu_1 - \nu_2\rho + i \sqrt{-(\rho \nu_2 - \nu_1)^2 - 4 \rho \nu_2(\nu_1 + \nu_3)}}{2\rho}\right)^{\frac{1}{2}}.
        \end{align}
        }
        \normalsize
    \item If $(\rho\nu_2 - \nu_1)^2 + 4\rho\nu_2(\nu_1 + \nu_3) \ge 0$ and $\rho\nu_2 - \nu_1 > \sqrt{(\rho\nu_2 - \nu_1)^2 + 4\rho\nu_2(\nu_1 + \nu_3)}$, then define $z_1, z_2$ as 
  \begin{align}\label{eq : def z1 z2 second case}
    z_1 &\bydef  \left(\frac{ \rho\nu_2 - \nu_1 + \sqrt{(\rho \nu_2 - \nu_1)^2 + 4 \rho \nu_2(\nu_1 + \nu_3)}}{2\rho}\right)^{\frac{1}{2}}\\
    z_2 &\bydef  \left(\frac{ \rho\nu_2 - \nu_1 -  \sqrt{(\rho \nu_2 - \nu_1)^2 + 4 \rho \nu_2(\nu_1 + \nu_3)}}{2\rho}\right)^{\frac{1}{2}}.
\end{align}
\end{enumerate}
In both cases, for $d_1,d_2 \in \mathbb{R}$, we obtain that
    \begin{equation*}
\left|\mathcal{F}^{-1}\left(\frac{d_1|2\pi\xi|^2 + d_2}{\mathscr{l}_{den}(\xi)}\right)(x)\right|\leq C_0e^{-a|x|},
    \end{equation*}
    with 
    \begin{align}
    a = \min\{\mathrm{Re}(z_1), \mathrm{Re}(z_2)\},\label{def : a}
    \end{align}
    \begin{align} C_0(d_1,d_2) = \frac{1}{|2\rho(z_1^2 - z_2^2)|} \left( |d_1 z_2|  +  \frac{|d_2|}{|z_2|} \right).\label{def : C0}
\end{align}
\end{lemma}

\begin{proof}
Case 1: $(\rho\nu_2 - \nu_1)^2 + 4\rho\nu_2(\nu_1 + \nu_3) < 0$. In this case,  $\ell_\text{den}$ possesses 4 complex roots $\pm y$ and $\pm \overline{y}$, where $y$ is given in \eqref{eq: def z1 z2 and y}.

Case 2: $(\rho\nu_2 - \nu_1)^2 + 4\rho\nu_2(\nu_1 + \nu_3) \ge 0$ and $\rho\nu_2 - \nu_1 > \sqrt{(\rho\nu_2 - \nu_1)^2 + 4\rho\nu_2(\nu_1 + \nu_3)}$. In this case, $\ell_\text{den}$ possesses 4 imaginary roots $\pm y_1$ and $\pm y_2$ where $y_1 = \frac{z_1}{2\pi i}$ and $y_2 = \frac{z_2}{2\pi i}$.

In any case, we define 
\begin{align*}
    f(x) = \mathcal{F}^{-1}\left(\frac{1}{\ell_\text{den}(\xi)}\right)(x) = \mathcal{F}^{-1}\left(\frac{1}{((2\pi \xi)^2 + z_1^2)((2\pi \xi)^2 + z_2^2)}\right)(x)
\end{align*}
and using that 
\begin{align}
    \mathcal{F}^{-1}\left(\frac{1}{|2\pi\xi^2| + {z}^2}\right)(t) = \frac{e^{-z|t|}}{2z},
\end{align}
 we  obtain
\begin{align}
    f(t) = \frac{1}{2\rho ({z}_2^2-z_1^2)} \left(\frac{e^{-z_1|t|}}{z_1} - \frac{e^{-{z}_2|t|}}{{z}_2}\right).
\end{align}
In the case 1., we have that $z_2 = \bar{z_1}$, this implies that 
\begin{align*}
    |f(t)| \leq \frac{ e^{-a|t|}}{|2 \rho  z_2 ({z}_1^2-z_2^2)|}
\end{align*}
where $a = \min\{\mathrm{Re}(z_1), \mathrm{Re}(z_2)\}$. Now, since $\frac{(2\pi\xi)^2}{\mathscr{l}_{den}(\xi)} \in L^1(\R)$, we have that $f$ is twice continuously differentiable. In particular,  we have
\begin{align*}
    f''(t) = \mathcal{F}^{-1}\left(\frac{(2\pi\xi)^2}{\rho (|2\pi\xi^2| + {z}_1^2)(|2\pi\xi^2| + {z}_2^2)}\right)(t). 
\end{align*}
In fact, using the above, we readily have that 
\begin{align*}
    |f''(t)| \leq \frac{|z_2|e^{-a|t|}}{|\sqrt{2}\rho ({z}_1^2-z_2^2)|}.
\end{align*}
For the case 2., notice that $z_1 > z_2 >0$. This implies that 
\begin{align}
   0 < f(t) = \frac{1}{2\rho ({z}_1^2-z_2^2)} \left(\frac{e^{-z_2|t|}}{z_2} - \frac{e^{-{z}_1|t|}}{{z}_1}\right) \leq  \frac{e^{-z_2|t|}}{2\rho z_2 ({z}_1^2-z_2^2)}.
\end{align}
Similarly, we get
\begin{align}
    |f''(t)| \leq     \frac{z_2e^{-z_2|t|}}{2\rho  ({z}_1^2-z_2^2)}.
\end{align}
We obtain the desired result by observing that:
\begin{align*}
\left|\mathcal{F}^{-1}\left(\frac{d_1|2\pi\xi|^2 + d_2}{\mathscr{l}_{den}(\xi)}\right)(t)\right| \leq |d_1f''(t)| + |d_2 f(t)| 
    \leq C_0 e^{-a|t|}.
\end{align*}
\end{proof}
We now move to the second preliminary result which allows to separate the evaluation of $\mathcal{Z}_u$ into simpler elementary computations which can be evaluated using the results of \citep{unbounded_domain_cadiot}.
In fact, we follow the approach outlined in \citep{gs_cadiot_blanco}, and obtain the following lemma.
\begin{lemma}\label{lem : old Zu}
Let $\left(\mathcal{Z}_{u,k,j}\right)_{k \in \{1,2\}, j \in \{1,2,3,4\}}$ be bounds satisfying
\begin{align}
&\mathcal{Z}_{u,1,1} \geq \|\mathbb{1}_{\mathbb{R}\setminus\om} \mathbb{l}_{22}\mathbb{l}_{den}^{-1} \mathbb{v}_1^N\|_{2},~\mathcal{Z}_{u,2,1} \geq \|\mathbb{1}_{\om} (\mathbb{l}_{22}\mathbb{l}_{den}^{-1} - \Gamma^\dagger(l_{22}l_{\mathrm{den}}^{-1}))\mathbb{v}_1^N\|_{2}\\
&\mathcal{Z}_{u,1,2} \geq \|\mathbb{1}_{\mathbb{R}\setminus\om}\mathbb{l}_{21}\mathbb{l}_{den}^{-1}\mathbb{v}_2^N\|_{2},~\mathcal{Z}_{u,2,2} \geq \|\mathbb{1}_{\om} (\mathbb{l}_{21}\mathbb{l}_{den}^{-1} - \Gamma^\dagger(l_{21}l_{\mathrm{den}}^{-1}))\mathbb{v}_2^N\|_{2}
    \\
    &\mathcal{Z}_{u,1,3} \geq \|\mathbb{1}_{\mathbb{R}\setminus\om}\mathbb{l}_{12}\mathbb{l}_{den}^{-1}\mathbb{v}_1^N\|_{2},~\mathcal{Z}_{u,2,3} \geq \|\mathbb{1}_{\om}(\mathbb{l}_{12}\mathbb{l}_{den}^{-1} - \Gamma^\dagger(l_{12}l_{\mathrm{den}}^{-1}))\mathbb{v}_1^N\|_{2}\\
    &\mathcal{Z}_{u,1,4} \geq \|\mathbb{1}_{\mathbb{R}\setminus\om}\mathbb{l}_{22}\mathbb{l}_{den}^{-1}\mathbb{v}_2^N\|_{2},~\mathcal{Z}_{u,2,4} \geq \|\mathbb{1}_{\om}(\mathbb{l}_{11}\mathbb{l}_{den}^{-1} - \Gamma^\dagger(l_{11}l_{\mathrm{den}}^{-1}))\mathbb{v}_2^N\|_{2}
\end{align} 
Moreover, given $k \in \{1,2\}$, define  $\mathcal{Z}_{u,k} \bydef \sqrt{2}\sqrt{(\mathcal{Z}_{u,k,1} + \mathcal{Z}_{u,k,2})^2 + (\mathcal{Z}_{u,k,3}+\mathcal{Z}_{u,k,4})^2} $.
Then, it follows that $\mathcal{Z}_{u,1}$ and $\mathcal{Z}_{u,2}$ satisfy
\begin{align} 
\mathcal{Z}_{u,1} \geq \|\mathbb{1}_{\mathbb{R} \setminus \om} D\mathbb{g}(\overline{u})^N \mathbb{l}^{-1}\|_{2},~\mathcal{Z}_{u,2} \geq \|\mathbb{1}_{\om}D\mathbb{g}(\overline{u})^N(\bGam^\dagger(l^{-1}) - \mathbb{l}^{-1})\|_{2}.
\end{align}
Finally, defining \begin{align}
    \mathcal{Z}_u \bydef \sqrt{\mathcal{Z}_{u,1}^2 + \mathcal{Z}_{u,2}^2},
\end{align}
it follows that $\|\mathbb{B}D\mathbb{g}(\overline{u})^N(\bGam^\dagger(l^{-1}) - \mathbb{l}^{-1})\|_{2} \leq \max\left\{1,\left\|\begin{bmatrix}
    B_{11}^N - B_{12}^N \\
    B_{21}^N - B_{22}^N
\end{bmatrix}\right\|_{2}\right\}\mathcal{Z}_u$.
\end{lemma}
\begin{proof}
To begin, we use the definitions of \eqref{def : v1v2} and $\mathbb{B}$ to see that
\begin{align}
   \left\| \mathbb{B}D\mathbb{g}(\overline{u})^N(\mathbf{u}_0)\left(\bGam^\dagger\left(l^{-1}\right) - \mathbb{l}^{-1}\right)\right\|_2 &\leq \max\left\{1,\left\|\begin{bmatrix}
    B_{11}^N - B_{12}^N \\
    B_{21}^N - B_{22}^N
\end{bmatrix}\right\|_{2}\right\} \mathcal{Z}_u \\
&= \max\left\{1,\left\|\begin{bmatrix}
    B_{11}^N - B_{12}^N \\
    B_{21}^N - B_{22}^N
\end{bmatrix}\right\|_{2}\right\}\sqrt{\mathcal{Z}_{u,1}^2 + \mathcal{Z}_{u,2}^2}.
\end{align}
Now, we examine $\mathcal{Z}_{u,1}\geq \|\mathbb{1}_{\mathbb{R}\setminus \om}\mathbb{m}^N(\overline{\mbf{x}}) \mathbb{l}^{-1}\|_{2}.$ To do so,
let $\mathbf{u} = (u_1,u_2) \in L^2$ such that $\|\mathbf{u}\|_{2} = 1$. Then, observe that
{\small\begin{align}
    \nonumber \|\mathbb{1}_{\mathbb{R}\setminus \om}\mathbb{l}^{-*}(D\mathbb{g}(\overline{u})^N)^{*}\mathbf{u}\|_{2}
    &=\left\|\mathbb{1}_{\mathbb{R}\setminus \om} \begin{bmatrix}
        \mathbb{l}_{22}\mathbb{l}_{den}^{-1}& - \mathbb{l}_{21} \mathbb{l}_{den}^{-1} \\
        -\mathbb{l}_{12}\mathbb{l}_{den}^{-1} & \mathbb{l}_{11}\mathbb{l}_{den}^{-1}
    \end{bmatrix}\begin{bmatrix}
      \mathbb{v}_1^N & -\mbb{v}_1^N \\ \mathbb{v}_2^N & -\mbb{v}_2^N
    \end{bmatrix}\begin{bmatrix}
        u_1 \\ u_2
    \end{bmatrix}\right\|_{2} \\ \nonumber
    &= \left\|\mathbb{1}_{\mathbb{R}\setminus \om}\begin{bmatrix}
        \mathbb{l}_{22}\mathbb{l}_{den}^{-1}& - \mathbb{l}_{21} \mathbb{l}_{den}^{-1} \\
        -\mathbb{l}_{12}\mathbb{l}_{den}^{-1} & \mathbb{l}_{11}\mathbb{l}_{den}^{-1}
    \end{bmatrix}\begin{bmatrix}
        \mathbb{v}_1^N(u_1-u_2) \\ \mathbb{v}_2^N(u_1-u_2)
    \end{bmatrix}\right\|_{2} \\ \nonumber
    &= \left\|\mathbb{1}_{\mathbb{R}\setminus \om} \begin{bmatrix}
        \mathbb{l}_{22}\mathbb{l}_{den}^{-1}\mathbb{v}_1^N (u_1-u_2) -\mathbb{l}_{21}\mathbb{l}_{den}^{-1} \mathbb{v}_2^N (u_1-u_2) \\
        -\mathbb{l}_{12}\mathbb{l}_{den}^{-1} \mathbb{v}_1^N (u_1-u_2) + \mathbb{l}_{11}\mathbb{l}_{den}^{-1} \mathbb{v}_2^N (u_1-u_2)
    \end{bmatrix}\right\|_{2} 
    \\
    &\leq \biggl(\left( \|\mathbb{1}_{\mathbb{R}\setminus \om}\mathbb{l}_{22}\mathbb{l}_{den}^{-1} \mathbb{v}_1^N (u_1-u_2) \|_{2} + \|\mathbb{1}_{\mathbb{R}\setminus \om}\mathbb{l}_{21}\mathbb{l}_{den}^{-1}\mathbb{v}_2^N (u_1-u_2)\|_{2}\right)^2 \\
    &\hspace{+2cm}+ (\|\mathbb{1}_{\mathbb{R}\setminus \om}\mathbb{l}_{12}\mathbb{l}_{den}^{-1} \mathbb{v}_1^N (u_1-u_2)\|_{2} + \|\mathbb{1}_{\mathbb{R}\setminus \om}\mathbb{l}_{11}\mathbb{l}_{den}^{-1} \mathbb{v}_2^N (u_1-u_2)\|_{2})^2\biggr)^{\frac{1}{2}}.
\end{align}}
Now, observe that
\begin{align}
    \|\mathbb{1}_{\mathbb{R} \setminus \om} \mathbb{l}_{ij} \mathbb{l}_{den}^{-1} \mathbb{v}_k^N (u_1 - u_2)\|_{2} \leq \|\mathbb{1}_{\mathbb{R} \setminus \om} \mathbb{l}_{ij} \mathbb{l}_{den}^{-1} \mathbb{v}_k^N\|_{2} \|u_1 - u_2\|_{2}
    &\leq \|\mathbb{1}_{\mathbb{R} \setminus \om} \mathbb{l}_{ij} \mathbb{l}_{den}^{-1} \mathbb{v}_k^N\|_{2} \left\| \begin{bmatrix}
    1 & -1 \\
    0 & 0 \end{bmatrix} \begin{bmatrix}
        u_1 \\ u_2
    \end{bmatrix}\right\|_{2}\\
    &\leq \|\mathbb{1}_{\mathbb{R} \setminus \om} \mathbb{l}_{ij} \mathbb{l}_{den}^{-1} \mathbb{v}_k^N\|_{2} \left\| \begin{bmatrix}
    1 & -1 \\
    0 & 0 \end{bmatrix}\right\|_{2}\|\mbf{u}\|_{2} \\
    &= \sqrt{2}\|\mathbb{1}_{\mathbb{R} \setminus \om} \mathbb{l}_{ij} \mathbb{l}_{den}^{-1} \mathbb{v}_k^N\|_{2}.
\end{align}
Therefore, return to the previous step and obtain
\begin{align}
    \|\mathbb{1}_{\mathbb{R} \setminus \om} \mathbb{l}^{-*} (D\mathbb{g}(\overline{\mbf{u}})^N)^* \mbf{u}\|_{2} &\leq \biggl(\left( \sqrt{2}\|\mathbb{1}_{\mathbb{R}\setminus \om}\mathbb{l}_{22}\mathbb{l}_{den}^{-1} \mathbb{v}_1^N \|_{2} + \sqrt{2}\|\mathbb{1}_{\mathbb{R}\setminus \om}\mathbb{l}_{21}\mathbb{l}_{den}^{-1}\mathbb{v}_2^N \|_{2}\right)^2 \\
    &\hspace{+3cm}+ (\sqrt{2}\|\mathbb{1}_{\mathbb{R}\setminus \om}\mathbb{l}_{12}\mathbb{l}_{den}^{-1} \mathbb{v}_1^N \|_{2} + \sqrt{2}\|\mathbb{1}_{\mathbb{R}\setminus \om}\mathbb{l}_{11}\mathbb{l}_{den}^{-1} \mathbb{v}_2^N \|_{2})^2\biggr)^{\frac{1}{2}} \\&\leq \biggl(2\left( \|\mathbb{1}_{\mathbb{R}\setminus \om}\mathbb{l}_{22}\mathbb{l}_{den}^{-1} \mathbb{v}_1^N \|_{2} + \|\mathbb{1}_{\mathbb{R}\setminus \om}\mathbb{l}_{21}\mathbb{l}_{den}^{-1}\mathbb{v}_2^N \|_{2}\right)^2 \\
    &\hspace{+3cm}+ 2(\|\mathbb{1}_{\mathbb{R}\setminus \om}\mathbb{l}_{12}\mathbb{l}_{den}^{-1} \mathbb{v}_1^N \|_{2} + \|\mathbb{1}_{\mathbb{R}\setminus \om}\mathbb{l}_{11}\mathbb{l}_{den}^{-1} \mathbb{v}_2^N \|_{2})^2\biggr)^{\frac{1}{2}} \\
&\leq \sqrt{2}\sqrt{(\mathcal{Z}_{u,1,1} + \mathcal{Z}_{u,1,2})^2+(\mathcal{Z}_{u,1,3}+\mathcal{Z}_{u,1,4})^2}.
\end{align}

By performing the same steps for $\mathcal{Z}_{u,2}$, we similarly obtain
\begin{align}
    \|\mathbb{1}_{\om}D\mathbb{g}(\overline{u})^N(\bGam^\dagger(l^{-1}) - \mathbb{l}^{-1})\|_{2} \leq 2\sqrt{(\mathcal{Z}_{u,2,1} + \mathcal{Z}_{u,2,2})^2 + (\mathcal{Z}_{u,2,3}+\mathcal{Z}_{u,2,4})^2}.
\end{align}
\end{proof}
 Combining Lemmas \ref{lem : ift_L_inverse} and \ref{lem : old Zu}, we obtain an explicit formula for the bound $\mathcal{Z}_u$. The formulas, which come from computations performed for scalar PDEs in \citep{unbounded_domain_cadiot}, is composed of finite-dimensional evaluations. As a result, these can be computed using rigorous numerics.
\begin{lemma}\label{lem : Zu old exact}
Let $a$ be defined as in \eqref{def : a}. Moreover, let $E \in \ell^2$ and $C(d),C_1,C_2,C_3,C_4 > 0$ be defined as
\begin{align}
    &E \bydef \gamma(\mathbb{1}_{\om} \cosh(2ax)),~ C(d) \bydef 4d + \frac{4e^{-ad}}{a(1-e^{-\frac{3ad}{2}})} + \frac{2}{a(1-e^{-2ad})},\\
    &C_1 \bydef C_0(-1,-\nu_2),~ C_2 \bydef C_0(0,\nu_2),~C_3 \bydef C_0(0,\nu_3),~\text{and}~ C_4 \bydef C_0(-\rho,\nu_1)\label{def : E Cd Cj}
\end{align}
where $C_0(d_1,d_2)$ is defined as in \eqref{def : C0}. Then, let $(\mathcal{Z}_{u,k,j})_{k \in \{1,2\},j \in \{1,2,3,4\}} > 0$ be defined as
\begin{align}
&\mathcal{Z}_{u,1,1}^2 \bydef |\om| \frac{C_{1}^2e^{-2ad}}{a} (V_1^N, V_1^N * E)_2,~\mathcal{Z}_{u,2,1}^2 \bydef \mathcal{Z}_{u,1,1}^2 + e^{-4ad} C(d) C_1^2 |\om|(V_1^N,V_1^N * E)_2 \\
    &\mathcal{Z}_{u,1,2}^2 \bydef |\om| \frac{C_2^2 e^{-2ad}}{a} (V_2^N, V_2^N * E)_2,~\mathcal{Z}_{u,2,2}^2 \bydef \mathcal{Z}_{u,1,2}^2 + e^{-4ad} C(d) C_2^2 |\om|(V_2^N,V_2^N * E)_2 \\
    &\mathcal{Z}_{u,1,3}^2 \bydef |\om| \frac{C_3^2 e^{-2ad}}{a} (V_1^N, V_1^N * E)_2,~\mathcal{Z}_{u,2,3}^2 \bydef \mathcal{Z}_{u,1,3}^2 + e^{-4ad} C(d) C_3^2 |\om|(V_1^N,V_1^N * E)_2 \\
    &\mathcal{Z}_{u,1,4}^2 \bydef |\om| \frac{C_4^2e^{-2ad}}{a} (V_2^N,V_2^N*E)_2,~\mathcal{Z}_{u,2,4}^2 \bydef \mathcal{Z}_{u,1,4}^2 + e^{-4ad} C(d) C_4^2 |\om|(V_2^N,V_2^N * E)_2.
\end{align}
If $\mathcal{Z}_{u,1},\mathcal{Z}_{u,2},$ and $\mathcal{Z}_u$ are defined as in Lemma \ref{lem : old Zu}, then it follows that $\|\mathbb{B}D\mathbb{g}(\overline{u})^N(\bGam^\dagger(l^{-1}) - \mathbb{l}^{-1})\|_{2} \leq \max\left\{1,\left\|\begin{bmatrix}
    B_{11}^N - B_{12}^N \\
    B_{21}^N - B_{22}^N
\end{bmatrix}\right\|_{2}\right\} \mathcal{Z}_u$.
\end{lemma}
\begin{proof}
The proof follows from Lemma 6.5 of \citep{unbounded_domain_cadiot}. In particular, each $\mathcal{Z}_{u,k,j}$ can be computed using the results of the aforementioned lemma.
\end{proof}
\subsection{Constructive existence proofs of localized solutions}\label{sec : Proofs of patterns}
\begin{figure}[H]
    \centering
    \begin{subfigure}[t]{0.3\textwidth}
        \centering
        \epsfig{figure=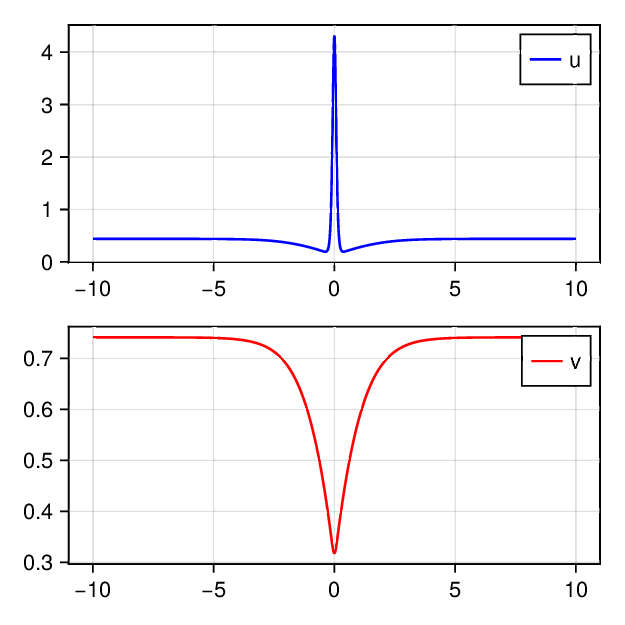, width=\textwidth}
        \caption{Approximation of a localized solution $\mbf{u}_1$ in the Glycolysis equation}\label{fig : gly pattern 1}
    \end{subfigure}
    \hfill
    \begin{subfigure}[t]{0.3\textwidth}
        \centering
        \epsfig{figure=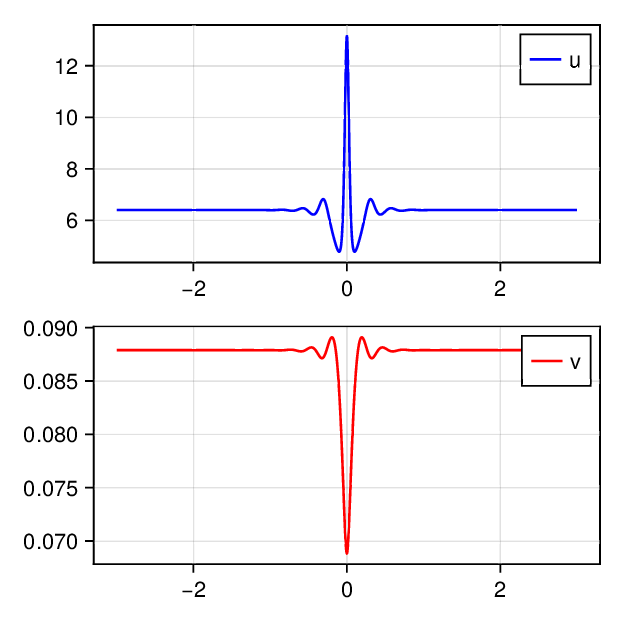, width=\textwidth}
        \caption{Approximation of a localized solution $\mbf{u}_2$ in the Schnakenberg equation}\label{fig : gly pattern 2}
    \end{subfigure}
    \hfill
    \begin{subfigure}[t]{0.3\textwidth}
    \centering
        \epsfig{figure=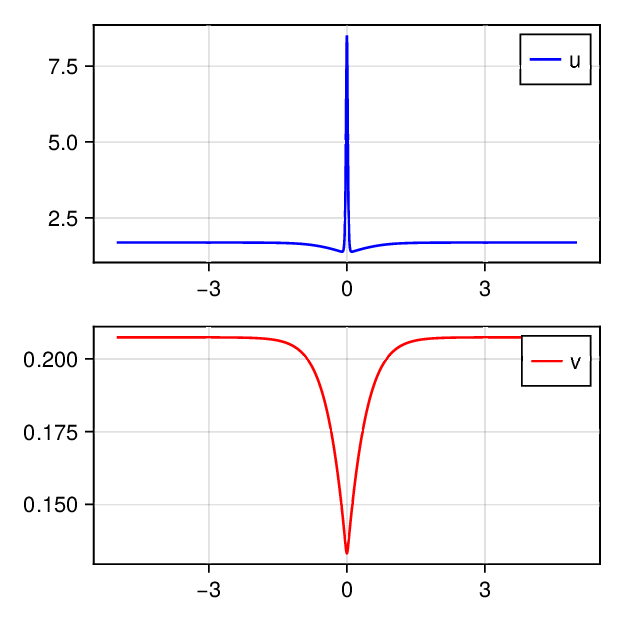, width=\textwidth}
        \caption{Approximation of a localized solution $\mbf{u}_3$ in the Selkov-Schnakenberg equation}\label{fig : sel pattern}
    \end{subfigure}
    \\
    \begin{subfigure}[t]{0.3\textwidth}
        \centering
        \epsfig{figure=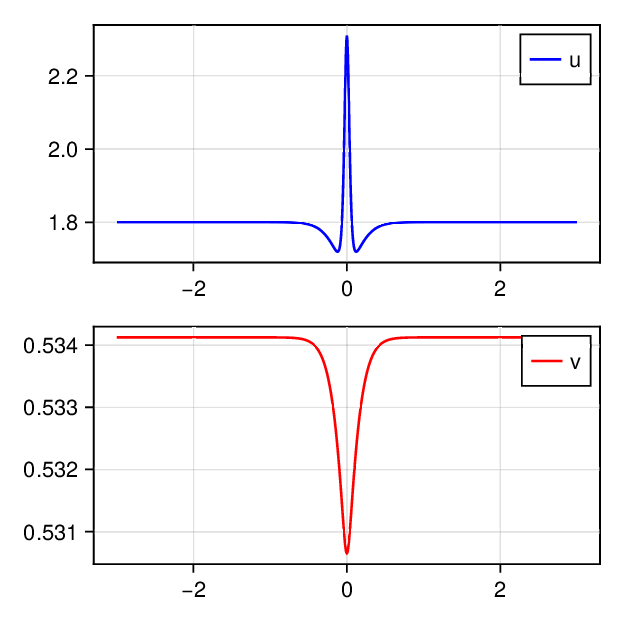, width=\textwidth}
        \caption{Approximation of a localized solution $\mbf{u}_4$ in the root-hair \\equation}\label{fig : sch pattern}
    \end{subfigure}
    \hfill
    \begin{subfigure}[t]{0.3\textwidth}
        \centering
        \epsfig{figure=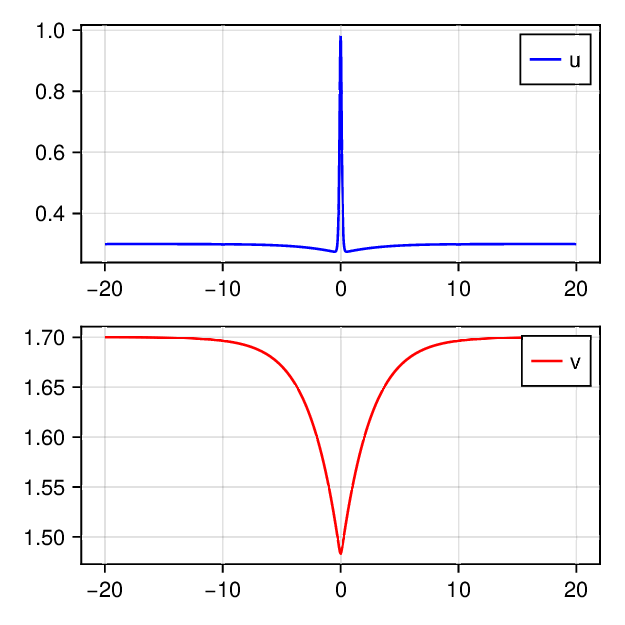, width=\textwidth}
        \caption{Approximation of a localized solution $\mbf{u}_5$ in the Brusselator equation}\label{fig : rh pattern 2}
    \end{subfigure}
    \hfill
    \begin{subfigure}[t]{0.3\textwidth}
    \centering
        \epsfig{figure=bruss_multi_spike.eps, width=\textwidth}
        \caption{Approximation of a localized solution $\mbf{u}_6$ in the Brusselator equation}\label{fig : bruss pattern}
    \end{subfigure}
    \caption{Approximations of localized solutions for general activator-inhibitor models }\label{fig:patterns}
\end{figure}

%
%
%
In this section, we apply the Newton-Kantorovich approach presented in Section \ref{sec : patterns} thanks to the explicit formulas derived in Section \ref{sec : Bounds of patterns}. In fact, the bounds are evaluated rigorously thanks to the \textit{Julia} packages \textit{IntervalArithmetic.jl} \citep{julia_interval} and \textit{RadiiPolynomial.jl} \citep{julia_olivier}. More specifically, the computer-assisted proofs details can be found on \citep{julia_blanco_fassler}. Now, we present our obtained constructive proofs of existence for localized patterns in \eqref{eq:original_system}.

\begin{table}[H]
\centering
\renewcommand{\arraystretch}{1.5}
\resizebox{\textwidth}{!}{
\begin{tabular}{|c|c|c|c|c|c|c|c|}
\hline
 & $a$ & $b$ & $c$ & $h$ & $j$ & $\delta$ & $r_0$\\
\hline

$u_1$  & $0$  & $0.44$  & $0$  & $0$  & $0.5$  & $0.45$ & $2\times 10^{-7}$  \\
\hline
$u_2$   & $2.8$  & $3.6$  & $1$  & $0$  & $0$  & $0.011$ & $3\times 10^{-10}$\\
\hline
$u_3$    & $1.095$  & $0.589818$  & $1$  & $0$  & $0.005$  & $0.01$ & $4\times 10^{-6}$ \\
\hline
$u_4$   & $0$  & $0.44$  & $0$  & $0$  & $0.5$  & $0.45$ & $6\times 10^{-7}$\\
\hline
$u_5$   & $0.3$  & $0$  & $1.51$  & $0.51$  & $0$  & $0.05$ &  $8\times 10^{-5}$\\
\hline
$u_6$    & $16.421100616$  & $0$  & $2.3299999237$  & $1.3299999237$  & $0$  & $0.0095742405424$ & $9\times 10^{-11}$\\
\hline
 \end{tabular}
}
\caption{Parameter values for the solutions of Theorem \ref{th : existence proofs}}
\label{tab:rigorous_values}
\end{table}

\begin{theorem}[\bf Proofs of Localized solutions]\label{th : existence proofs}
    Let $\mbf{u}_i, i \in\{1,2,3,4,5,6\}$ be the approximate solutions, constructed as in \eqref{eq : approximate solution}, represented in Figure~\ref{fig:patterns}, and whose Fourier coefficients can be downloaded at \citep{julia_blanco_fassler}. Then, for each $i \in \{1, \dots, 6\}$, there exists a unique solution $\tilde{\mbf{u}}_i$ in $\overline{B_{r_0}(\mbf{u}_i)}\subset \mathcal{H}_e$ to~\eqref{eq : gen_model}, where the corresponding parameters are given in  Table \ref{tab:rigorous_values}.
\end{theorem}
\begin{proof}
    The proof of this theorem is obtained by putting together Propositions~\ref{prop : spike in gyl}, \ref{prop : Pulse in schnakenberg}, \ref{prop : spike in brusselator}, \ref{prop : multispike in brusselator}, \ref{prop : spike in selkov schnakenberg} and \ref{prop : spike in rh} in Appendix~\ref{apen : proof of patterns}. The model that each solution solves is provided in Figure \ref{fig:patterns}.
\end{proof}

\begin{remark}
We acknowledge that the solution denoted $\mathbf{u}_6$ proven in Brusselator has parameters featuring many decimal places in Table~\ref{tab:rigorous_values}. This is justified by the will to produce a solution with multiple oscillations, as depicted  in Figure \ref{fig : bruss pattern}. We computed this solution using the continuation library AUTO \citep{auto_article}.
\end{remark}

\section{Saddle-Node Bifurcations : zeros of \texorpdfstring{$\mathbb{F}$}{F}}\label{sec : saddle node}
In this manuscript, we present a novel methodology for proving the existence of saddle-node bifurcations on unbounded domains. In fact, our approach  relies on the setup described in Section 5 of \citep{unbounded_domain_cadiot}, which deals with PDEs augmented with extra equations and extra parameters. In fact, in this section we focus on a computer-assisted methodology for proving constructively the existence of non-degenerate zeros of $\mathbb{F}$ given in \eqref{eq:saddle_node_map}.

Let $\bar{\rho}$ be fixed and be  such that $\mathbb{l}_{\overline{\rho}} \bydef \mathbb{l}(\overline{\rho})$ is invertible and  define the $\mathcal{H}-$norm as 
\begin{equation}\label{eq:h-norm-saddle}
    \|\mathbf{u}\|_{\mathcal{H}}^2 \bydef \int_{\mathbb{R}^2}|\hat{\mathbf{u}}(\xi)||\mathscr{l}_{\overline{\rho}}(\xi)|d\xi
\end{equation}
where $\mathscr{l}_{\overline{\rho}}(\xi)$ is the Fourier transform of $\mathbb{l}_{\overline{\rho}}$. We now define the Hilbert spaces $H_1 = \mathbb{R}\times (\mathcal{H}_e)^2$ and $H_2 = \mathbb{R}\times (L_e^2)^2$ endowed with their norms
\begin{equation}\label{eq:H12_norm}
    \|(\rho, \mathbf{u},\mathbf{w})\|_{H_1} = \left(|\rho|^2 + \|\mathbf{u}\|_{\mathcal{H}}^2 + \|\mathbf{w}\|_{\mathcal{H}}^2\right)^{\frac{1}{2}}~~\text{and}~~\|(\rho, \mathbf{u})\|_{H_2} = \left(|\rho|^2 + \|\mathbf{u}\|_2^2 + \|\mathbf{w}\|_{2}^2\right)^{\frac{1}{2}}.
\end{equation}
We now define
\begin{align}
    \mathbb{L}_{\rho} = \mathbb{L}(\rho) \bydef  \begin{bmatrix}
        1 & 0 & 0 \\
        0 & \mathbb{l}(\rho) & 0 \\
        0 & 0 & \mathbb{l}(\rho)
    \end{bmatrix}~~\mathbb{G}(\mbf{u},\mbf{w}) \bydef \begin{bmatrix}
        0 \\
        \mathbb{g}(\mbf{u}) \\
        D_{\mbf{u}}\mathbb{g}(\mbf{u})\mbf{w}
    \end{bmatrix}
\end{align}
and specify the definition of the \emph{Saddle-Node Map} $\mathbb{F} : H_1 \to H_2$  \eqref{eq:saddle_node_map} as follows
\begin{equation}
  \mathbb{F}(\mathbf{x}) \bydef \begin{bmatrix}
      (\mathbb{l}_{\overline{\rho}}\boldsymbol{\xi}, \mbf{w})_2  -\alpha \\\mathbb{f}(\rho,\mbf{u}) \\
      D_{\mbf{u}}\mathbb{f} (\rho,\mbf{u})\mbf{w} 
  \end{bmatrix}
\end{equation}
where $\boldsymbol{\xi} \bydef (\xi_1,\xi_2) \in \mathcal{H}_e$ is a non-trivial fixed function and $\alpha \in \mathbb{R}$ is a fixed constant. 

As provided in Definition \ref{lem : saddle node def}, non-degenerate zeros of $\mathbb{F}$ provide candidates for a saddle-node bifurcation. Moreover, \citep{unbounded_domain_cadiot} provides that, given $\tilde{\mathbf{u}}\in \mathcal{H}_e$, $D\mathbb{f}(\tilde{\rho}, \tilde{\mathbf{u}}) : \mathbb{R}\times \mathcal{H}_e \to \mathbb{R}\times L^2_e$ is a Fredholm operator of index zero. The spectral condition of Definition \ref{lem : saddle node def} will then be verified in the next section. For now, we  perform rigorous existence proofs of zeros of  \eqref{eq:saddle_node_map} in $H_1$ using the following theorem.
\begin{theorem}\label{th: radii polynomial saddle}
Let $\mathbb{A} : H_2 \to H_1$ be a bounded linear operator. Moreover, let $\mathcal{Y}_0, \mathcal{Z}_1$ be non-negative constants and let  $\mathcal{Z}_2 : (0, \infty) \to [0,\infty)$ be a non-negative function  such that for all $r>0$
  \begin{align}\label{eq: definition Y0 Z1 Z2 saddle}
    \|\mathbb{A}\mathbb{F}(\overline{\mathbf{x}})\|_{H_1} \leq &\mathcal{Y}_0\\
    \|I - \mathbb{A}D\mathbb{F}(\overline{\mathbf{x}})\|_{H_1} \leq &\mathcal{Z}_1\\
    \|\mathbb{A}\left({D}\mathbb{F}(\overline{\mathbf{x}} + \mathbf{h}) - D\mathbb{F}(\overline{\mathbf{x}})\right)\|_{H_1} \leq &\mathcal{Z}_2(r)r, ~~ \text{for all } \mathbf{h} \in B_r(0)
\end{align}  
If there exists $r>0$ such that \eqref{condition radii polynomial} is satisfied, then there exists a unique $\tilde{\mathbf{x}} \in \overline{B_{r}(\overline{\mathbf{x}})} \subset H_1$ such that $\mathbb{F}(\tilde{\mathbf{x}})=0$, where $B_{r}(\overline{\mathbf{x}})$ is the open ball of radius $r$ in $H_1$ and centered at $\overline{\mathbf{x}}$. 
\end{theorem}
\begin{proof}
The same proof of Theorem \ref{th: radii polynomial} applies with the changing of the map and spaces.
\end{proof}

\par The augmented zero finding problem \eqref{eq:saddle_node_map} has a periodic boundary value problem correspondence. In fact, as $\boldsymbol{\xi} \in L^2$, we can define $\boldsymbol{\Xi} \bydef \bgam(\boldsymbol{\xi}) \in \ell^2_e$. Then we define $X_1 \bydef \mathbb{R}\times (X^l_e)^2$ and $X_2 \bydef \mathbb{R} \times (\ell^2_e)^2$ the corresponding Hilbert spaces associated to their natural norms
{\small\begin{align}
    \|(\rho,\mbf{U},\mbf{W})\|_{X_1} \bydef \left( |\rho|^2 + {|\om|}\|\mbf{U}\|_l^2 + |\om|\|\mbf{W}\|_{\mathcal{H}}^2\right)^{\frac{1}{2}}, ~\|(\rho,\mbf{U},\mbf{W})\|_{X_2} \bydef \left( |\rho|^2 + {|\om|}\|\mbf{U}\|_2^2 + |\om| \|\mbf{W}\|_2^2\right)^{\frac{1}{2}}.
\end{align}}
 Note that the normalization by ${|\om|}$ is useful to obtain an isometry (cf. Lemma \ref{lem : gamma and Gamma properties}). Next, we introduce $\mbf{X} \bydef (\rho, \bgam(\mbf{u}),\bgam(\mbf{w})) = (\rho,\mbf{U},\mbf{W}) \in X_1$.
 Finally, we  define a corresponding operator $F : X_1 \to X_2$ along with $L$ and $G$ as 
 \begin{align}
    F(\mbf{X}) \bydef \begin{bmatrix} |\om|((l_{\overline{\rho}}\boldsymbol{\Xi}, \mbf{W})_2  -\alpha\\
    f(\rho,\mbf{U}) \\
    D_{\mbf{U}} f(\rho,\mbf{U})\mbf{W}\end{bmatrix},~
    L = L(\rho) \bydef \begin{bmatrix}
        1 & 0 & 0 \\
        0 & l & 0 \\
        0 & 0 & l
    \end{bmatrix},~~ G(\mbf{U},\mbf{W}) \bydef \begin{bmatrix}
        0 \\
        g(\mbf{U}) \\
        D_{\mbf{U}}g(\mbf{U})\mbf{W}
    \end{bmatrix}
\end{align}
where $f(\rho,\mbf{U}) = f(\mbf{U})$ is defined as in \eqref{eq : F(U)=0 in X^l_e}. 
Note that the construction of $\overline{\mbf{u}}$ can be done the same way as in Section \ref{sec : u0}. Furthermore, one can then construct $\overline{\mbf{w}}$ by computing the approximate eigenvector corresponding to the $0$ eigenvalue of $DF(\overline{U})$. Let us now discuss the approximate inverse, $\mathbb{A}$.

\subsection{Approximate Inverse}
We now briefly recall the setup of Section 5.2 of \citep{unbounded_domain_cadiot} to build the approximate inverse. Denote $H_{2,\om} \bydef \{(\rho,\mbf{u},\mbf{w}) \in H_2, ~~ \text{supp}(\mbf{u}) \subset \overline{\om}, \text{supp}(\mbf{w}) \subset \overline{\om}\}$ and $H_{1,\om}\bydef \{(\rho,\mbf{u},\mbf{w}) \in H_1, ~~ \text{supp}(\mbf{u}) \subset \overline{\om}, \text{supp}(\mbf{w}) \subset \overline{\om}\}$. By construction, we have $\overline{\mathbf{x}} = (\overline{\rho},\overline{\mathbf{u}},\overline{\mathbf{w}}) \in H_{1,\om}$. 
Recall that $\mathcal{B}(Y)$ is the set of bounded linear operators on some space $Y$. Denote by $\mathcal{B}_\om(H_2)$ the following  subset of $\mathcal{B}(H_2)$ 
\[
\mathcal{B}_\om(H_2) \bydef \left\{\mathbb{K} \in \mathcal{B}(H_2), ~~ \mathbb{K} = \begin{bmatrix}
        1&0&0\\
        0&\cha&0\\
        0&0&\cha
    \end{bmatrix} \mathbb{K} \begin{bmatrix}
        1&0&0\\
        0&\cha&0\\
        0&0&\cha
    \end{bmatrix}\right\}.
\]
Now,  let $\og : H_2 \to X_2$, $\ogd : X_2 \to H_2$,$\oGG : \mathcal{B}(H_2) \to \mathcal{B}(X_2)$, and $\oGGD : \mathcal{B}(X_2) \to \mathcal{B}(H_2)$ be defined as follows
\begin{align}
    &\og \bydef \begin{bmatrix}
        1 & 0 & 0\\
        0 & \bgam & 0 \\
        0 & 0 & \bgam
    \end{bmatrix}
    ,~ \ogd \bydef \begin{bmatrix}
        1 & 0 & 0\\
        0 & \bgam^\dagger & 0 \\
        0 & 0 & \bgam^\dagger
    \end{bmatrix},~
    \oGG\left(\mathbb{K}\right) \bydef \og \mathbb{K} \ogd
    ,~ \text{ and } ~ \oGGD\left(K\right) \bydef \ogd K \og
\end{align}
for all $\mathbb{K} \in \mathcal{B}(H_2)$ and all $K \in \mathcal{B}(X_2).$ 
We now recall the equivalent to Lemma \ref{lem : gamma and Gamma properties} for this case.

\begin{lemma}\label{lem : properties constraigned gamma}
    The map $\og : H_{2,\om} \to X_2$ (respectively $\oGG : \mathcal{B}_\om(H_2) \to \mathcal{B}(X_2)$) is an isometric 
    isomorphism whose inverse is $\ogd : X_2 \to H_{2,\om}$ 
    (respectively $\oGGD : \mathcal{B}(X_2) \to \mathcal{B}_\om(H_2)$). In particular,
    \begin{align*}
        \|(\rho,\mbf{u},\mbf{w})\|_{H_2} = \|(\rho,\mbf{U},\mbf{W})\|_{X_2} ~~ \text{ and } ~~ \|\mathbb{K}\|_{H_2} = \|K\|_{X_2} 
    \end{align*}
    for all $(\rho,\mbf{u},\mbf{w}) \in H_{2,\om}$ and all $\mathbb{K} \in \mathcal{B}_\om(H_2)$, where $(\rho,\mbf{U},\mbf{W}) \bydef \og\left(\rho,\mbf{u},\mbf{w}\right)$ and $K \bydef \oGG\left(\mathbb{K}\right)$.
\end{lemma}
Finally, let $N \in \mathbb{N}$ ($N$ is the size of the numerical approximation) and define $\overline{\bpi}^N$ and $\overline{\bpi}_N$ as 
\begin{align}
    \overline{\bpi}^{\leq N} \bydef \begin{bmatrix}
        1 & 0 & 0\\
        0 & \bpi^{\leq N} & 0 \\
        0 & 0 & \bpi^{\leq N}
    \end{bmatrix} 
  ~~  \text{ and } \obpi^{>N} \bydef \begin{bmatrix}
        0 & 0 & 0\\
        0 & \bpi^{>N} & 0 \\
        0 & 0 & \bpi^{>N}
    \end{bmatrix}, 
\end{align}
where $\bpi^{\leq N}$ and $\bpi^{>N}$ are given in \eqref{def : piN and pisubN}.

As before, we want to approximate the inverse of $D\mathbb{F}(\overline{\mathbf{x}})$ by some operator $\mathbb{A} : H_2 \to H_1$ through the construction of an operator $\mathbb{B} : H_2 \to H_2$, where $\mathbb{A} = \mathbb{L}_{\overline{\rho}}^{-1} \mathbb{B}.$ 
As described in Section 5 of \citep{unbounded_domain_cadiot}, we compute a numerical approximate inverse $B^N : X_2 \to X_2$ for $\overline{\bpi}^{\leq N}DF(\overline{\mathbf{X}})L_{\overline{\rho}}^{-1}\overline{\bpi}^{\leq N}$ such that $B^N  = \overline{\bpi}^{\leq N} B^N \overline{\bpi}^{\leq N}$. Then, we define $\mathbb{A} : H_2 \to H_1$ and $A : X_2 \to X_1$ as 
\begin{align}\label{def : operator A constraigned}
    \mathbb{A} \bydef \mathbb{L}_{\overline{\rho}}^{-1}\left(\begin{bmatrix}
        0&0&0\\
        0&\out&0\\
        0&0&\out
    \end{bmatrix} + \oGGD\left(\overline{\bpi}^{>N} + B^N\right) \right) ~~\text{ and } ~~ A \bydef L_{\overline{\rho}}^{-1}\left(\overline{\bpi}^{>N} + B^N\right).
\end{align}
 With $\mathbb{A}$ now constructed, we are able to apply Theorem \ref{th: radii polynomial} to \eqref{eq:saddle_node_map}. What remains is to compute the needed bounds, $ \mathcal{Y}_0,  \mathcal{Z}_1,$ and $ \mathcal{Z}_2 $ in this case. This will be the focus of the next section.
\subsection{Bounds for Saddle Nodes}\label{sec : Bounds for Saddle Nodes}
In this section, we compute the bounds of Theorem \ref{th: radii polynomial saddle} in order to prove existence and uniqueness of zeros to \eqref{eq:saddle_node_map}. As a result, we obtain a rigorous enclosure of a saddle node bifurcation to \eqref{eq : gen_model}. We will begin with the $\mathcal{Y}_0$ bound.
This can be done by combining the approach in Section 6 of \citep{unbounded_domain_cadiot} to treat a scalar PDE with an extra equation and that of Section 4.1 of \citep{gs_cadiot_blanco} where a system of PDEs was considered. 
The estimate for $\mathcal{Y}_0$ is similar to that of Lemma \ref{lem : bound Y_0}. Indeed, we introduce the projection operator, and take care of the extra equation to perform the estimate.
\begin{lemma}\label{lem : bound Y_0 saddle}
Let $\mathcal{Y}_0 >0$ be such that 
{\small\begin{align}
    &\mathcal{Y}_0 \bydef |\om|^{\frac{1}{2}} \sqrt{\left\|B^N\begin{bmatrix} ((l_{\overline{\rho}}\boldsymbol{\Xi}, \mbf{W}_{0})_2  - \alpha\\
    f(\overline{\rho},\overline{\mathbf{U}}) \\
    D_{\mbf{U}} f(\overline{\rho},\overline{\mathbf{U}})\overline{\mathbf{W}}\end{bmatrix}\right\|_{2}^2 + \left\| \begin{bmatrix}
        0 \\
       (\bpi^{\leq N_0} - \bpi^{\leq N}) l_{\overline{\rho}} \overline{\mathbf{U}} + (\bpi^{3N_0} - \bpi^{\leq N})g(\overline{\mathbf{U}})  \\
        (\bpi^{\leq N_0} - \bpi^{\leq N})l_{\overline{\rho}} \overline{\mathbf{W}} + (\bpi^{3N_0} - \bpi^{\leq N})D_{\mbf{U}}g(\overline{\mathbf{U}})\overline{\mathbf{W}}
    \end{bmatrix}\right\|_{2}^2}.
\end{align}}
Then $\|\mathbb{A}\mathbb{F}(\overline{\mathbf{x}})\|_{H_1} \leq \mathcal{Y}_0.$
\end{lemma}
\begin{proof}
Following similar to steps to those from Lemma \ref{lem : bound Y_0}, we have
{\small\begin{align}
    \|\mathbb{A}\mathbb{F}(\overline{\mathbf{x}})\|_{H_1} = \|\mathbb{B}\mathbb{F}(\overline{\mathbf{x}})\|_{H_2} = \|B F(\overline{\mathbf{X}})\|_{X_2} &\leq \left(\left\|B^NF(\overline{\mathbf{X}})\right\|_{X_2}^2 + \left\| \begin{bmatrix} 0\\
    \bpi^{>N} f(\overline{\rho},\overline{\mathbf{U}}) \\
    \bpi^{>N} D_{\mbf{U}} f(\overline{\rho},\overline{\mathbf{U}})\overline{\mathbf{W}}\end{bmatrix}\right\|_{X_2}^2\right)^{\frac{1}{2}}.
\end{align}}
Following the same steps as Lemma \ref{lem : bound Y_0} but on the larger system with more components will yield the result. 
\end{proof}
%
%
%
Next, we compute the $\mathcal{Z}_2$ bound. Due to the presence of the extra equation, the computation of $\mathcal{Z}_2$ changes significantly. We summarize the results in the following lemma.
\begin{lemma}\label{lem : bound Z_2 saddle}
Let 
\begin{align}\label{def : Z2 saddle definitions}
    &q_1 \bydef  2\overline{u}_2 + 2\nu_4 \mathbb{1}_{\om},~
        q_2 \bydef 2\overline{u}_1 +\nu_5 \mbb{1}_{\Omega_0},~ q_3 \bydef 2\overline{w}_2,~ q_4 \bydef 2\overline{w}_1 \\
    &\mathbb{B}_{1\to5,24} \bydef  \begin{bmatrix}
        0 & \mathbb{B}_{12} & 0 & \mathbb{B}_{14} & 0\\
        0 & \mathbb{B}_{22} & 0 &\mathbb{B}_{24} & 0\\
        0 & \mathbb{B}_{32} & 0 & \mathbb{B}_{34} & 0 \\
        0 & \mathbb{B}_{42} & 0 &\mathbb{B}_{44} & 0\\
        0 & \mathbb{B}_{52} & 0 & \mathbb{B}_{54} & 0
    \end{bmatrix},~B_{1\to5,24}^N \bydef 
    \obpi^{\leq N}\oGGD( \mathbb{B}_{1\to5,24}) \end{align}
    {\small\begin{align}
    &M_1 \bydef \begin{bmatrix}
         B_{12}^N-B_{13}^N &  B_{14}^N-B_{15}^N \\
         B_{22}^N-B_{23}^N & B_{24}^N-B_{25}^N \\
          B_{32}^N- B_{33}^N  & B_{34}^N - B_{35}^N  \\
          B_{42}^N-B_{43}^N  &  B_{44}^N-B_{45} \\
         B_{52}^N-B_{53}^N  & B_{54}^N- B_{55}^N
    \end{bmatrix},~M_2 \bydef \begin{bmatrix}
        \mathbb{Q}_1^2 + \mathbb{Q}_2^2 & \mathbb{Q}_1 \mathbb{Q}_3 + \mathbb{Q}_2 \mathbb{Q}_4 \\
        \mathbb{Q}_3 \mathbb{Q}_1 + \mathbb{Q}_4 \mathbb{Q}_2 & \mathbb{Q}_3^2 + \mathbb{Q}_4^2 + \mathbb{Q}_1^2 + \mathbb{Q}_2^2
    \end{bmatrix}.
\end{align}}
Additionally, we introduce the non-negative function $\mathcal{Z}_{2,2}: [0,\infty) \to (0,\infty)$ for  and the non-negative constants $\mathcal{Z}_{2,j} > 0$ for $j = 1,3,4,5$ as 
\begin{align}
    &\mathcal{Z}_{2,1} \bydef \max(1,\|B^N_{1\to5,24}\|_{X_2}) \left\|\begin{bmatrix}
        -|2\pi \xi|^2  & 0 \\
        0  & 0
    \end{bmatrix}\mathscr{l}^{-1}(\xi)\right\|_{\mathcal{M}_1},~\mathcal{Z}_{2,2}(r) \bydef \|M_1\|_{2,X_2}2\sqrt{10} \kappa \|\mathscr{l}^{-1}\|_{\mathcal{M}_2}r  \\
    &\mathcal{Z}_{2,3} \bydef \|\obpi^{\leq N} M_1 M_2 M_1^* \obpi^{\leq N}\|_{X_2},~\mathcal{Z}_{2,4} \bydef  \sqrt{\|\obpi^{\leq N} M_1 M_2 \bpi^{>N} M_2^* M_1^* \obpi^{\leq N}\|_{X_2}},\label{def : Z2js saddle} \\
    &\mathcal{Z}_{2,5} \bydef \varphi(\|Q_1^2 + Q_2^2\|_{1}, \|Q_1 Q_3 + Q_2 Q_4\|_{1}, \|Q_3 Q_1 + Q_4 Q_2\|_{1}, \|Q_3^2 + Q_4^2 + Q_1^2 + Q_2^2\|_{1}).
\end{align}
Finally, we let $\mathcal{Z}_2(r): [0,\infty) \to (0,\infty)$ be a non-negative function defined as 
\begin{align}
    \mathcal{Z}_2(r) \bydef (1+\sqrt{3})\mathcal{Z}_{2,1} + \mathcal{Z}_{2,2}(r)  + 5\kappa  \sqrt{\varphi(\mathcal{Z}_{2,3},\mathcal{Z}_{2,4},\mathcal{Z}_{2,4},\mathcal{Z}_{2,5})}.
\end{align}
Then, it follows that $\|\mathbb{A}(D\mathbb{F}(\mbf{h} + \overline{\mathbf{x}}) - D\mathbb{F}(\overline{\mathbf{x}}))\|_{H_1} \leq \mathcal{Z}_2(r)r$
\end{lemma}
\begin{proof}
The proof can be found in Appendix \ref{apen : saddle node computations}.
\end{proof}
%
%
%
Next, we examine the $\mathcal{Z}_1$ bound for \eqref{eq:saddle_node_map}. To begin, we introduce some new notations. First, recall the definitions of $v_1, v_2$ from \eqref{def : v1v2}. Then, we introduce
\begin{align}
    &v_3 \bydef q_1 w_1 + q_2 w_2,~\text{and}~v_4 \bydef q_2 w_1
\end{align}
and define 
\begin{align}
    \mathbb{m}(\overline{\mbf{x}}) \bydef \begin{bmatrix}
        D_{\mbf{u}}\mathbb{g}(\overline{\mathbf{u}}) & 0 \\
         D^2_{\mbf{u}}\mathbb{g}(\overline{\mathbf{u}})\overline{\mathbf{w}} & D_{\mbf{u}}\mathbb{g}(\overline{\mathbf{u}})
    \end{bmatrix} \bydef \begin{bmatrix}
        \mathbb{v}_1 & \mathbb{v}_2 & 0 & 0 \\
        -\mathbb{v}_1 & -\mathbb{v}_2 & 0 & 0 \\
        \mathbb{v}_3 & \mathbb{v}_4 & \mathbb{v}_1 & \mathbb{v}_2 \\
        -\mathbb{v}_3 & -\mathbb{v}_4 & -\mathbb{v}_1 & -\mathbb{v}_2
    \end{bmatrix},&\mathbb{m}^N(\overline{\mbf{x}}) = \begin{bmatrix}
   \mathbb{v}_1^N & \mathbb{v}_2^N & 0 & 0 \\
    -\mbb{v}_1^N & -\mbb{v}_2^N & 0 & 0 \\
     \mathbb{v}_3^N & \mathbb{v}_4^N & \mathbb{v}_1^N & \mathbb{v}_2^N \\
    -\mathbb{v}_3^N & -\mathbb{v}_4^N & -\mathbb{v}_1^N & -\mathbb{v}_2^N
    \end{bmatrix} 
\end{align}
where $\mathbb{v}_j$ is the multiplication operator associated to $v_j$ for $j = 1,2,3,4$ and
\begin{align}
    &V_j \bydef \gamma(v_j), V_j^N = \Pi^{\leq N} V_j, v_j^N \bydef \gamma^\dagger(V_j^N).
    \end{align}
For $k = 1,2$, we also define
\begin{align}
    &\Xi_k^N = \Pi^{\leq N} \Xi_k, \xi_k^N = \gamma^\dagger(\Xi_k^N) ~~U_1^N = \Pi^{\leq N} U_1, u_1^N = \gamma^\dagger(U_1^N)~~W_1^N = \Pi^{\leq N} W_1, w_1^N = \gamma^\dagger(W_1^N).
\end{align}
Now, given $u \in L^2(\mathbb{R}^2)$ and $U \in \ell^2(\mathbb{Z}^2)$, define $u^*$ and $U^*$ as the duals in $L^2(\mathbb{R}^2)$ and $\ell^2(\mathbb{Z}^2)$ respectively of $u$ and $U$. More specifically, we have 
\begin{align}
    u^*(v) \bydef (u,v)_2, ~~ U^*(V) \bydef (U,V)_2 ~~~~ \text{ for all } v \in L^2(\mathbb{R}^2) \text{ and } V \in \ell^2(\mathbb{Z}^2).
    \label{dual definition}
\end{align}
 Moreover, we slightly abuse notation and, given $u \in L^2$ and $U \in \ell^2$, we denote $u^* \mathbb{l} : \mathcal{H} \to \mathbb{R}$ and $U^* l : \mathscr{h} \to \mathbb{R}$ the following linear operators
\begin{align}\label{def : dual compose L}
    u^* \mathbb{l}v \bydef (u, \mathbb{l}v)_2, ~~ U^* lV \bydef (U, lV)_2  ~~~~ \text{ for all } v \in L^2(\mathbb{R}^2) \text{ and } V \in \ell^2(\mathbb{Z}^2).
\end{align}
Using the above notation, we introduce
\begin{align}
    \mathbb{M}(\overline{\mathbf{x}})
    &\bydef
    \begin{bmatrix}
        -1 & 0 & \boldsymbol{\xi}^* \mathbb{l}_{\overline{\rho}}^* \\
        D_{\rho} \mathbb{f}(\overline{\rho},\overline{\mathbf{u}})  & D_{\mbf{u}}\mathbb{g}(\overline{\mathbf{u}}) & 0 \\
        D_{\rho,\mbf{u}}\mathbb{f}(\overline{\rho},\overline{\mathbf{u}})\overline{\mathbf{w}} & D^2_{\mbf{u}}\mathbb{g}(\overline{\mathbf{u}})\overline{\mathbf{w}} & D_{\mbf{u}}\mathbb{g}(\overline{\mathbf{u}})
    \end{bmatrix}, \\
    \mathbb{M}^N(\overline{\mathbf{x}}) &\bydef
    \begin{bmatrix}
        -1 & 0 & (\boldsymbol{\xi}^N)^* \mathbb{l}_{\overline{\rho}}^* \\
        D_{\rho} \mathbb{f}^N(\overline{\rho},\overline{\mathbf{u}})  & D_{\mbf{u}}\mathbb{g}^N(\overline{\mathbf{u}}) & 0 \\
        D_{\rho,\mbf{u}}\mathbb{f}^N(\overline{\rho},\overline{\mathbf{u}})\overline{\mathbf{w}} & D^2_{\mbf{u}}\mathbb{g}^N(\overline{\mathbf{u}})\overline{\mathbf{w}} & D_{\mbf{u}}\mathbb{g}^N(\overline{\mathbf{u}})
    \end{bmatrix},
    M^N(\overline{\mathbf{X}}) \bydef \oGGD(\mathbb{M}^N(\overline{\mathbf{x}}))
\end{align}
We further introduce notation for clarify
\begin{align}
    \Delta U \bydef (\Delta U)_{n \in \mathbb{Z}} = (-\tilde{n}^2 u_n)_{n \in \mathbb{Z}} .
\end{align}
Now, we state the following lemma for the $\mathcal{Z}_1$ bound. The motivation is similar to that done in Lemma \ref{lem : Z_full_1}. In fact, the $\mathcal{Z}_u$ bound is mostly unchanged as the extra equation does not play a major role. The primary difference will be in the $Z_1$ bound, where the estimate will require us to treat the augmented system directly.
\begin{lemma}\label{lem : Z_full_1 saddle}
Let $\mathcal{Z}_{\infty,1},\mathcal{Z}_{\infty,2},\mathcal{Z}_{\infty,3},\mathcal{Z}_{\infty},\mathcal{Z}_u,Z_1,$ and $\mathcal{Z}_1 > 0$ be constants satisfying
{\small\begin{align} 
&\mathcal{Z}_{\infty,1} \bydef |\om| \sqrt{\|\Xi_1 - \Xi_1^N\|_{2}^2 + \|\Xi_2 - \Xi_2^N\|_{2}^2},~
\mathcal{Z}_{\infty,2} \bydef  |\om|  \sqrt{\|\Delta(\overline{U}_1 - \overline{U}_1^N)\|_{2}^2 + \|\Delta(\overline{W}_1 - \overline{W}_1^N)\|_{2}^2}, \\
&\mathcal{Z}_{\infty,3} \bydef \|\mathbb{l}_{\overline{\rho}}^{-1}\|_{2}\varphi\left(\left(\sum_{j = 1}^2 \|V_j - V_j^N\|_{1}^2\right)^{\frac{1}{2}},0,\left(\sum_{j = 3}^4 \|V_j - V_j^N\|_{1}^2\right)^{\frac{1}{2}},\left(\sum_{j = 1}^2 \|V_j - V_j^N\|_{1}^2\right)^{\frac{1}{2}} \right),
\end{align}}
\begin{align}
&\mathcal{Z}_{\infty} \bydef \varphi(0,\mathcal{Z}_{\infty,1},\mathcal{Z}_{\infty,2},\mathcal{Z}_{\infty,3}),~\mathcal{Z}_{u} \geq \left\|\mathbb{B}\begin{bmatrix}
    0 & 0 \\
    0 & \mathbb{m}^N(\overline{\mbf{x}})
\end{bmatrix}\left(\mathbb{L}_{\overline{\rho}}^{-1} - \oGG^\dagger(L_{\overline{\rho}}^{-1})\right)\right\|_{H_2}, \\
&Z_1 \geq \left\|I - B(I + M^N(\overline{\mathbf{X}})L_{\overline{\rho}}^{-1})\right\|_{X_2},~
\mathcal{Z}_1 \bydef Z_1 +  \mathcal{Z}_{u} +   \|M_1\|_{2,X_2}\mathcal{Z}_{\infty}.
\end{align}
Then, it follows that $ \|I -{\mathbb{A}}D\mathbb{F}(\overline{\mathbf{x}})\|_{H_1} \leq \mathcal{Z}_1.$
\end{lemma}
\begin{proof}
Following similar steps to those performed in \eqref{lem : Z_full_1}, we obtain
\begin{align}
    \|I_{d} - \mathbb{A}D\mathbb{F}(\overline{\mathbf{x}})\|_{H_1} &\leq Z_1 + \mathcal{Z}_u + \|M_1\|_{2,X_2}\|(\mathbb{M}(\overline{\mathbf{x}}) - \mathbb{M}^N(\overline{\mathbf{x}}))\mathbb{L}_{\overline{\rho}}^{-1}\|_{H_2}.
\end{align}
For the final term, we can use $\varphi$ defined in Lemma \ref{lem : full_matrix_estimate} to get
{\footnotesize\begin{align}
    \|(\mathbb{M}(\overline{\mathbf{x}}) - \mathbb{M}^N(\overline{\mathbf{x}}))\mathbb{L}_{\overline{\rho}}^{-1}\|_{H_2} 
    \leq \varphi\left(0, \left\|\boldsymbol{\xi} - \boldsymbol{\xi}^N\right\|_{2},\left\|\begin{bmatrix}
        D_{\rho} \mathbb{f}(\overline{\rho},\overline{\mathbf{u}})  - D_{\rho} \mathbb{f}^N(\overline{\rho},\overline{\mathbf{u}}) \\
        D_{\rho,\mbf{u}}\mathbb{f}(\overline{\rho},\overline{\mathbf{u}})\overline{\mathbf{w}} - D_{\rho,\mbf{u}}\mathbb{f}^N(\overline{\rho},\overline{\mathbf{u}})\overline{\mathbf{w}}^N
    \end{bmatrix}\right\|_{2}, \|\mathbb{l}_{\overline{\rho}}^{-1}\|_{2}\left\|\mathbb{m}(\overline{\mbf{x}}) - \mathbb{m}^N(\overline{\mbf{x}})\right\|_{2}\right).
\end{align}}
Now, we will examine each term required to evaluate this expression. To begin, observe that
\begin{align}
    &\|\boldsymbol{\xi} - \boldsymbol{\xi}^N\|_{2} = |\om| \|\boldsymbol{\Xi} - \boldsymbol{\Xi}^N\|_{2} = |\om| \sqrt{\|\Xi_1 - \Xi_1^N\|_{2}^2 + \|\Xi_2 - \Xi_2^N\|_{2}^2} \bydef \mathcal{Z}_{\infty,1}. \end{align}
This is the second term required. For the third term, notice we can write
   \begin{align} &\left\|\begin{bmatrix}
        D_{\rho} \mathbb{f}(\overline{\rho},\overline{\mathbf{u}})  - D_{\rho} \mathbb{f}^N(\overline{\rho},\overline{\mathbf{u}}) \\
        D_{\rho,\mbf{u}}\mathbb{f}(\overline{\rho},\overline{\mathbf{u}})\overline{\mathbf{w}} - D_{\rho,\mbf{u}}\mathbb{f}^N(\overline{\rho},\overline{\mathbf{u}})\overline{\mathbf{w}}^N
    \end{bmatrix}\right\|_{2} 
    = |\om| \left\|\begin{bmatrix}
        D_{\rho} f(\overline{\rho},\overline{\mathbf{U}})  - D_{\rho} f^N(\overline{\rho},\overline{\mathbf{U}}) \\
        D_{\rho,\mbf{U}}f(\overline{\rho},\overline{\mathbf{U}})\overline{\mathbf{W}} - D_{\rho,\mbf{U}}f^N(\overline{\rho},\overline{\mathbf{U}})\overline{\mathbf{W}}^N
    \end{bmatrix}\right\|_{2} \\
    &\hspace{+1.5cm}= |\om| \sqrt{\|D_{\rho} f(\overline{\rho},\overline{\mathbf{U}})  - D_{\rho} f^N(\overline{\rho},\overline{\mathbf{U}}) \|_{2}^2 + \|D_{\rho,\mbf{U}}f(\overline{\rho},\overline{\mathbf{U}})\overline{\mathbf{W}} - D_{\rho,\mbf{U}}f^N(\overline{\rho},\overline{\mathbf{U}})\overline{\mathbf{W}}^N\|_{2}^2} \\
    &\hspace{+1.5cm}= |\om|  \sqrt{\|\Delta (U_1 - U_1^N)\|_{2}^2 + \|\Delta (W_1 - W_1^N)\|_{2}^2} \bydef \mathcal{Z}_{\infty,2}.
    \end{align}
Finally, we estimate the fourth term.
    {\scriptsize\begin{align}
&\|\mathbb{l}_{\overline{\rho}}^{-1}\|_{2}\left\|\mathbb{m}(\overline{\mbf{x}}) - \mathbb{m}^N(\overline{\mbf{x}})\right\|_{2} \leq \|\mathbb{l}_{\overline{\rho}}^{-1}\|_{2}\varphi\left(\left(\sum_{j = 1}^2 \| V_j-V_j^N\|_{1}^2\right)^{\frac{1}{2}},0,\left(\sum_{j = 3}^4 \|V_j-V_j^N\|_{1}^2\right)^{\frac{1}{2}},\left(\sum_{j = 1}^2 \|V_j-V_j^N\|_{1}^2\right)^{\frac{1}{2}} \right) \bydef \mathcal{Z}_{\infty,3}.
\end{align}}
This concludes the proof.
\end{proof}
Next, we must compute the $Z_1$ and $\mathcal{Z}_u$ bounds. We begin with the $Z_1$ bound which, as mentioned previously, is now different as we must treat the extra equation.
\begin{lemma}\label{lem : Z1 periodic saddle}
    Let
    \begin{align}
        Z_1 \bydef \sqrt{Z_{0}^2 + 2\left(\sqrt{Z_{1,1}^2 + Z_{1,2}^2} + \sqrt{Z_{1,3}^2 + Z_{1,4}^2}\right)^2}
    \end{align}
    where $Z_{1,1}$ and $Z_{1,2}, \left(\frac{\mathscr{l}_{ij}}{l_{\mathrm{den}}}\right)_{2N}$ are defined as in Lemma \ref{lem : Z1 periodic patterns} and
    \begin{align*}
        &Z_{0} \bydef \|\obpi^{\leq2N} - \obpi^{\leq 3N}B(I_d + M^N(\overline{X}) L^{-1})\obpi^{\leq2N}\|_{X_2}\\
        &Z_{1,3} \bydef \left(\frac{l_{22}}{l_{\mathrm{den}}}\right)_{2N} \| V_3^N\|_{1} + \left(\frac{l_{21}}{l_{\mathrm{den}}}\right)_{2N} \|V_4^N\|_{1},~~ 
        Z_{1,4} &\bydef \left(\frac{l_{12}}{l_{\mathrm{den}}}\right)_{2N} \|V_3^N\|_{1} + \left(\frac{l_{11}}{l_{\mathrm{den}}}\right)_{2N} \|V_4^N\|_{1}. 
    \end{align*}
Then, it follows that $\|I_d - B(I_d + Mg^N(\overline{X}) L_{\overline{\rho}}^{-1})\|_{X_2} \leq Z_1$.
\end{lemma}

\begin{proof}
To begin, we introduce a truncation.
\begin{align}
    &\|I_d - B(I_d + M^N(\overline{X}) L_{\overline{\rho}}^{-1})\|_{X_2}^2 \\
    &\leq \|\obpi^{>2N} - B(I_d + M^N(\overline{X}) L_{\overline{\rho}}^{-1})\obpi^{>2N}\|_{X_2}^2 + \|\obpi^{\leq2N} - B(I_d + M^N(\overline{X}) L_{\overline{\rho}}^{-1})\bpi^{\leq2N}\|_{X_2}^2 \\
    &= \|\obpi^{>2N} - B \obpi^{>2N} - BM^N(\overline{X}) \obpi^{<2N} L_{\overline{\rho}}^{-1}\|_{X_2}^2 + Z_{0}^2 \\
    &= \|\obpi^{> N} M^N(\overline{X}) L_{\overline{\rho}}^{-1} \obpi^{<2N}\|_{X_2}^2 + Z_0^2 \\
    &= \left\|\begin{bmatrix} \bpi^{> N} & 0 \\
    0 & \bpi^{>N}\end{bmatrix}\begin{bmatrix}
        Dg^N(\overline{\rho},\overline{\mbf{U}}) & 0 \\
        D^2g^N(\overline{\rho},\overline{\mbf{U}}) \overline{\mbf{W}} & Dg^N(\overline{\rho},\overline{\mbf{U}})
    \end{bmatrix} \begin{bmatrix}
        l_{\overline{\rho}}^{-1} & 0 \\
        0 & l_{\overline{\rho}}^{-1}
    \end{bmatrix}\begin{bmatrix} \bpi^{>2 N} & 0 \\
    0 & \bpi^{>2N}\end{bmatrix}\right\|_{2}^2 + Z_0^2.
\end{align}
Now, we use the structure of $\begin{bmatrix}
        Dg^N(\overline{\rho},\overline{\mbf{U}}) & 0 \\
        D^2g^N(\overline{\rho},\overline{\mbf{U}}) \overline{\mbf{W}} & Dg^N(\overline{\rho},\overline{\mbf{U}})
    \end{bmatrix}$  to get
{\small\begin{align}
    &\left\|\begin{bmatrix} \bpi^{> N} & 0 \\
    0 & \bpi^{>N}\end{bmatrix} \begin{bmatrix}
        Dg^N(\overline{\rho},\overline{\mbf{U}}) & 0 \\
        D^2g^N(\overline{\rho},\overline{\mbf{U}}) \overline{\mbf{W}} & Dg^N(\overline{\rho},\overline{\mbf{U}})
    \end{bmatrix} \begin{bmatrix}
        l_{\overline{\rho}}^{-1} & 0 \\
        0 & l_{\overline{\rho}}^{-1}
    \end{bmatrix}\begin{bmatrix} \bpi^{> 2N} & 0 \\
    0 & \bpi^{>2N}\end{bmatrix}\right\|_{2} \\
    &= \left\| \begin{bmatrix}
        \bpi^{> N}Dg^N(\overline{\rho},\overline{\mbf{U}}) l_{\overline{\rho}}^{-1}\bpi^{>2N} & 0 \\
        \bpi^{>N} D^2g^N(\overline{\rho},\overline{\mbf{U}})\overline{\mbf{W}} l_{\overline{\rho}}^{-1} \bpi^{>2N} & \bpi^{> N}Dg^N(\overline{\rho},\overline{\mbf{U}}) l_{\overline{\rho}}^{-1}\bpi^{>2N}
    \end{bmatrix}\right\|_{2} \\
    &\leq \left\| \begin{bmatrix}
        \bpi^{> N}Dg^N(\overline{\rho},\overline{\mbf{U}}) l_{\overline{\rho}}^{-1}\bpi^{>2N} & 0 \\
        0 & \bpi^{> N}Dg^N(\overline{\rho},\overline{\mbf{U}}) l_{\overline{\rho}}^{-1}\bpi^{>2N}
    \end{bmatrix}\right\|_{2} + \left\| \begin{bmatrix}
        0 & 0 \\
        \bpi^{>N} D^2g^N(\overline{\rho},\overline{\mbf{U}})\overline{\mbf{W}} l_{\overline{\rho}}^{-1} \bpi^{>2N} & 0
    \end{bmatrix}\right\|_{2} \\
    &=\left\| 
        \bpi^{> N}Dg^N(\overline{\rho},\overline{\mbf{U}}) l_{\overline{\rho}}^{-1}\bpi^{>2N}\right\|_{2} + \|\bpi^{>N} D^2g^N(\overline{\rho},\overline{\mbf{U}})\overline{\mbf{W}} l_{\overline{\rho}}^{-1} \bpi^{>2N}\|_{2} \\
    &\leq \sqrt{2Z_{1,1}^2 + 2Z_{1,2}^2} + \|\bpi^{>N} D^2g^N(\overline{\rho},\overline{\mbf{U}})\overline{\mbf{W}} l_{\overline{\rho}}^{-1} \bpi^{>2N}\|_{2}.
\end{align}}
The final term above can be estimated in a very similar way to $Z_{1,1}$ and $Z_{1,2}$. We obtain
\begin{align}
    \|\bpi^{>N} D^2g^N(\overline{\rho},\overline{\mbf{U}})\overline{\mbf{W}} l_{\overline{\rho}}^{-1} \bpi^{>2N}\|_{2} \leq \sqrt{2Z_{1,3}^2 + 2Z_{1,4}^2}.
\end{align}
Therefore, we have
\begin{align}
    \|I_d - B(I_d + M^N(\overline{X})L_{\overline{\rho}}^{-1})\|_{2} \leq \sqrt{Z_0^2 + 2\left(\sqrt{Z_{1,1}^2 + Z_{1,2}^2} + \sqrt{Z_{1,3}^2 + Z_{1,4}^2}\right)^2} \bydef Z_1
\end{align}
as desired.
\end{proof}
Finally, we compute the bound $\mathcal{Z}_u$ for the saddle node map \eqref{eq:saddle_node_map}. 
\begin{lemma}\label{lem : old Zu saddle}
Let $\left(\mathcal{Z}_{u,k,j}\right)_{k \in \{1,2\}, j \in \{1,2,3,4\}}$ be defined as Lemma \ref{lem : old Zu}. Let $\left(\mathcal{Z}_{u,k,j}\right)_{k \in \{3,4\}, j \in \{1,2,3,4\}}$ bounds satisfying
\begin{align}
&\mathcal{Z}_{u,3,1} \geq \|\mathbb{1}_{\mathbb{R}\setminus\om} \mathbb{l}_{22}\mathbb{l}_{den}^{-1} \mathbb{v}_3^N\|_{2},~\mathcal{Z}_{u,4,1} \geq \|\mathbb{1}_{\om} (\mathbb{l}_{22}\mathbb{l}_{den}^{-1} - \Gamma^\dagger(l_{22}l_{\mathrm{den}}^{-1}))\mathbb{v}_3^N\|_{2}\\
&\mathcal{Z}_{u,3,2} \geq \|\mathbb{1}_{\mathbb{R}\setminus\om}\mathbb{l}_{21}\mathbb{l}_{den}^{-1}\mathbb{v}_4^N\|_{2},~\mathcal{Z}_{u,4,2} \geq \|\mathbb{1}_{\om} (\mathbb{l}_{21}\mathbb{l}_{den}^{-1} - \Gamma^\dagger(l_{21}l_{\mathrm{den}}^{-1}))\mathbb{v}_4^N\|_{2}
    \\
    &\mathcal{Z}_{u,3,3} \geq \|\mathbb{1}_{\mathbb{R}\setminus\om}\mathbb{l}_{12}\mathbb{l}_{den}^{-1}\mathbb{v}_3^N\|_{2},~\mathcal{Z}_{u,4,3} \geq \|\mathbb{1}_{\om}(\mathbb{l}_{12}\mathbb{l}_{den}^{-1} - \Gamma^\dagger(l_{12}l_{\mathrm{den}}^{-1}))\mathbb{v}_3^N\|_{2}\\
    &\mathcal{Z}_{u,3,4} \geq \|\mathbb{1}_{\mathbb{R}\setminus\om}\mathbb{l}_{22}\mathbb{l}_{den}^{-1}\mathbb{v}_4^N\|_{2},~\mathcal{Z}_{u,4,4} \geq \|\mathbb{1}_{\om}(\mathbb{l}_{11}\mathbb{l}_{den}^{-1} - \Gamma^\dagger(l_{11}l_{\mathrm{den}}^{-1}))\mathbb{v}_4^N\|_{2}
\end{align} 
Moreover, given $k \in \{1,2,3,4\}$, define  $\mathcal{Z}_{u,k} \bydef \sqrt{2}\sqrt{(\mathcal{Z}_{u,k,1} + \mathcal{Z}_{u,k,2})^2 + (\mathcal{Z}_{u,k,3}+\mathcal{Z}_{u,k,4})^2} $.
Then, it follows that $\mathcal{Z}_{u,1}, \mathcal{Z}_{u,2}, \mathcal{Z}_{u,3},$ and $\mathcal{Z}_{u,4}$ satisfy
\begin{align} 
&\mathcal{Z}_{u,1} \geq \|\mathbb{1}_{\mathbb{R} \setminus \om} D_{\mbf{u}}\mathbb{g}^N(\overline{\mbf{u}}) \mathbb{l}_{\overline{\rho}}^{-1}\|_{2},~\mathcal{Z}_{u,2} \geq \|\mathbb{1}_{\om}D_{\mbf{u}}\mathbb{g}^N(\overline{\mbf{u}})(\bGam^\dagger(l_{\overline{\rho}}^{-1}) - \mathbb{l}_{\overline{\rho}}^{-1})\|_{2} \\
&\mathcal{Z}_{u,3} \geq \|\mathbb{1}_{\mathbb{R} \setminus \om} D^2_{\mbf{u}}\mathbb{g}^N(\overline{\mbf{u}})\overline{\mbf{w}} \mathbb{l}_{\overline{\rho}}^{-1}\|_{2},~\mathcal{Z}_{u,4} \geq \|\mathbb{1}_{\om}D^2_{\mbf{u}}\mathbb{g}^N(\overline{\mbf{u}})\overline{\mbf{w}}(\bGam^\dagger(l_{\overline{\rho}}^{-1}) - \mathbb{l}_{\overline{\rho}}^{-1})\|_{2}.
\end{align}
Finally, defining \begin{align}
    \mathcal{Z}_u \bydef \sqrt{\mathcal{Z}_{u,1}^2 + \mathcal{Z}_{u,2}^2} + \sqrt{\mathcal{Z}_{u,3}^2 + \mathcal{Z}_{u,4}^2},
\end{align}
it follows that $\left\|\mathbb{B}\begin{bmatrix} 
0 & 0 \\
0 & \mathbb{m}^N(\overline{\mbf{x}})\end{bmatrix}\left(\oGGD(L_{\overline{\rho}}^{-1}) - \mathbb{L}_{\overline{\rho}}^{-1}\right)\right\|_{X_2} \leq \|M_1\|_{2,X_2}\mathcal{Z}_u$.
\end{lemma}
\begin{proof}
To begin, observe that
\begin{align}
    &\left\|\mathbb{B}\begin{bmatrix} 
0 & 0 \\
0 & \mathbb{m}^N(\overline{\mbf{x}})\end{bmatrix}\left(\oGGD(L_{\overline{\rho}}^{-1}) - \mathbb{L}_{\overline{\rho}}^{-1}\right)\right\|_{2} \\
&\leq \|M_1\|_{2,X_2} \left\| \begin{bmatrix}
    D_{\mbf{u}} \mathbb{g}^N(\overline{\mbf{u}}) & 0 \\
    D^2_{\mbf{u}} \mathbb{g}^N(\overline{\mbf{u}})\overline{\mbf{w}} & D_{\mbf{u}} \mathbb{g}^N(\overline{\mbf{u}})
\end{bmatrix}\begin{bmatrix}
    \bGam^\dagger(l_{\overline{\rho}}^{-1}) - \mathbb{l}_{\overline{\rho}}^{-1} & 0 \\
    0 & \bGam^\dagger(l_{\overline{\rho}}^{-1}) - \mathbb{l}_{\overline{\rho}}^{-1}
\end{bmatrix}\right\|_{2} \\
&= \|M_1\|_{2,X_2} \left\| \begin{bmatrix}
    D_{\mbf{u}} \mathbb{g}^N(\overline{\mbf{u}})(\bGam^\dagger(l_{\overline{\rho}}^{-1}) - \mathbb{l}_{\overline{\rho}}^{-1}) & 0 \\
    D^2_{\mbf{u}} \mathbb{g}^N(\overline{\mbf{u}})\overline{\mbf{w}}(\bGam^\dagger(l_{\overline{\rho}}^{-1}) - \mathbb{l}_{\overline{\rho}}^{-1}) & D_{\mbf{u}} \mathbb{g}^N(\overline{\mbf{u}})(\bGam^\dagger(l_{\overline{\rho}}^{-1}) - \mathbb{l}_{\overline{\rho}}^{-1})
\end{bmatrix}\right\|_{2} \\
&\leq \|M_1\|_{2,X_2} \biggl(\left\| \begin{bmatrix}
    D_{\mbf{u}} \mathbb{g}^N(\overline{\mbf{u}})(\bGam^\dagger(l_{\overline{\rho}}^{-1}) - \mathbb{l}_{\overline{\rho}}^{-1}) & 0 \\
    0 & D_{\mbf{u}} \mathbb{g}^N(\overline{\mbf{u}})(\bGam^\dagger(l_{\overline{\rho}}^{-1}) - \mathbb{l}_{\overline{\rho}}^{-1})
\end{bmatrix}\right\|_{2} \\
&\hspace{+2cm}+ \left\| \begin{bmatrix}
    0 & 0 \\
    D^2_{\mbf{u}} \mathbb{g}^N(\overline{\mbf{u}})\overline{\mbf{w}}(\bGam^\dagger(l_{\overline{\rho}}^{-1}) - \mathbb{l}_{\overline{\rho}}^{-1}) & 0\end{bmatrix}\right\|_{2}\biggr) \\
    &= \|M_1\|_{2,X_2} \left( \|D_{\mbf{u}} \mathbb{g}^N(\overline{\mbf{u}})(\bGam^\dagger(l_{\overline{\rho}}^{-1}) - \mathbb{l}_{\overline{\rho}}^{-1})\|_{2} + \|D^2_{\mbf{u}} \mathbb{g}^N(\overline{\mbf{u}})\overline{\mbf{w}}(\bGam^\dagger(l_{\overline{\rho}}^{-1}) - \mathbb{l}_{\overline{\rho}}^{-1})\|_{2}\right) \\
    &\leq \|M_1\|_{2,X_2} \mathcal{Z}_u
\end{align}
as desired.
\end{proof}
We now estimate $\mathcal{Z}_{u,3}$ and $\mathcal{Z}_{u,4}$.
\begin{lemma}
Let $a, E, C(d),$ and $C_j$ for $j = 1,2,3,4$ be defined as in Lemma \ref{lem : Zu old exact}.
Then, let $(\mathcal{Z}_{u,k,j})_{k \in \{3,4\},j \in \{1,2,3,4\}} > 0$ be defined as
\begin{align}
    &\mathcal{Z}_{u,3,1}^2 \bydef |\om| \frac{C_{1}^2e^{-2ad}}{a} (V_3^N, V_3^N * E)_2,~\mathcal{Z}_{u,4,1}^2 \bydef \mathcal{Z}_{u,3,1}^2 + e^{-4ad} C(d) C_1^2 |\om|(V_3^N,V_3^N * E)_2 \\
    &\mathcal{Z}_{u,3,2}^2 \bydef |\om| \frac{C_2^2 e^{-2ad}}{a} (V_4^N, V_4^N * E)_2,~\mathcal{Z}_{u,4,2}^2 \bydef \mathcal{Z}_{u,3,2}^2 + e^{-4ad} C(d) C_2^2 |\om|(V_4^N,V_4^N * E)_2 \\
    &\mathcal{Z}_{u,3,3}^2 \bydef |\om| \frac{C_3^2 e^{-2ad}}{a} (V_3^N, V_3^N * E)_2,~\mathcal{Z}_{u,4,3}^2 \bydef \mathcal{Z}_{u,3,3}^2 + e^{-4ad} C(d) C_3^2 |\om|(V_3^N,V_3^N * E)_2 \\
    &\mathcal{Z}_{u,3,4}^2 \bydef |\om| \frac{C_4^2e^{-2ad}}{a} (V_4^N,V_4^N*E)_2,~\mathcal{Z}_{u,4,4}^2 \bydef \mathcal{Z}_{u,3,4}^2 + e^{-4ad} C(d) C_4^2 |\om|(V_4^N,V_4^N * E)_2.
\end{align}
If $\mathcal{Z}_{u,1},\mathcal{Z}_{u,2},\mathcal{Z}_{u,3},\mathcal{Z}_{u,4},$ and $\mathcal{Z}_u$ are defined as in Lemma \ref{lem : old Zu saddle}, then it follows that \\ $\left\|\mathbb{B}\begin{bmatrix} 
0 & 0 \\
0 & \mathbb{m}^N(\overline{\mbf{x}})\end{bmatrix}\left(\oGGD(L_{\overline{\rho}}^{-1}) - \mathbb{L}_{\overline{\rho}}^{-1}\right)\right\|_{X_2} \leq \|M_1\|_{2,X_2}\mathcal{Z}_u$.
\end{lemma}
\begin{proof}
The proof follows from Lemma 6.5 of \citep{unbounded_domain_cadiot}. In particular, each $\mathcal{Z}_{u,k,j}$ can be computed using the results of the aforementioned lemma.
\end{proof}
\subsection{Proofs of Saddle Nodes}\label{sec : proofs of saddle nodes}
In this section, we present the computer assisted proofs of Saddle Node bifurcations in \eqref{eq : gen_model}.
\begin{figure}[H]
    \centering
    \epsfig{figure=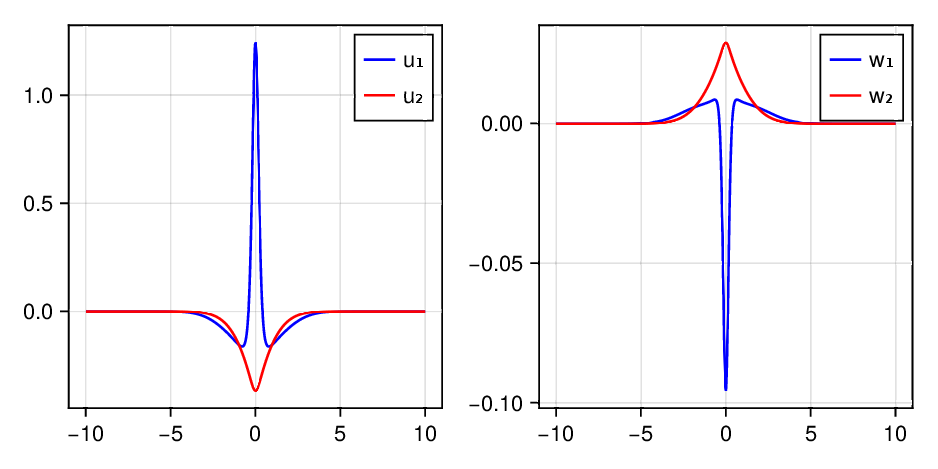, width = 0.5\textwidth}
    \caption{Approximation of a saddle-node bifurcation for the Glycolysis equation.}\label{fig:saddle gly}
\end{figure}
\begin{theorem}\label{th : saddle node in gyl}
Choose $\overline{\rho} = 0.010384836980401362$ and let $c = 1, b = 0.4154639603133325, j = 0.21, a = 0, h = 0$. Moreover, let $r_0 \bydef 4 \times 10^{-8}$. Then there exists a unique solution $\tilde{\mbf{x}} = (\tilde{\rho},\tilde{\mathbf{u}},\tilde{\mathbf{w}})$ to \eqref{eq : gen_model} in $\overline{B_{r_0}(\mbf{x}_{gyls})} \subset H_1$ and we have that $\|\tilde{\mbf{x}}-\mbf{x}_{gyls}\|_{H_1} \leq r_0$. 
\end{theorem}
\begin{proof}
Choose $N_0 = 800, N = 620 , d = 40$. Then, we perform the full construction to build $\overline{\mathbf{x}} = \ogd(\rho,\overline{\mathbf{U}},\overline{\mathbf{W}})$. Then, we define $\mbf{x}_{glys} \bydef \overline{\mathbf{x}}$. Next, we construct $B^N$ and use Lemma \ref{corr : banach algebra} to compute $\kappa$. This allows us to compute the $\mathcal{Z}_2(r)$ bound defined in Section \ref{sec : Bounds for Saddle Nodes}. 
Finally, using \citep{julia_blanco_fassler}, we choose $r_0 \bydef 4 \times 10^{-8}$ and define
\begin{align}
    \mathcal{Y}_0 \bydef 3.31 \times 10^{-9} \text{,}~\mathcal{Z}_{2}(r_0) \bydef    2.51 \times 10^4 \text{,}~Z_1 \bydef 0.112\text{,}~\mathcal{Z}_u \bydef 1.56 \times 10^{-4}.
    \end{align}
and prove that these values satisfy Theorem \ref{th: radii polynomial saddle}. 
\end{proof}
\section{Control of the spectrum of \texorpdfstring{$D\mathbb{f}(\tilde{\mathbf{u}})$}{DFtildeu}}\label{sec : control of the spectrum}

Following up our analysis of the previous section, the existence proof of a saddle-node bifurcation now depends on the rigorous control of the fully imaginary eigenvalues of $D\mathbb{f}(\tilde{\mathbf{u}})$. Moreover, we aim at controlling the eigenvalues with positive real part in order to conclude about stability change when passing the saddle node bifurcation.  In this section, we present a computer-assisted approach to do so, which highly relies on the methodology developed in \citep{cadiot2025stabilityanalysislocalizedsolutions}. In fact, we will obtain tight enclosures of the eigenvalues of $D\mathbb{f}(\tilde{\mathbf{u}})$. This will allow us to conclude about the existence of saddle-node bifurcations.


First, we recall some properties and results established in \citep{cadiot2025stabilityanalysislocalizedsolutions}. Let $\sigma(D\mathbb{f}(\tbu))$ be the spectrum of $D\mathbb{f}(\tbu)$. Then, Lemma 2.2 of \citep{cadiot2025stabilityanalysislocalizedsolutions} provides that $\sigma(D\mathbb{f}(\tbu))$ is the following disjoint union
\begin{align}\label{eq : decomposition spectrum}
    \sigma(D\mathbb{f}(\tbu)) = \sigma_{ess}(D\mathbb{f}(\tbu))\cup Eig(D\mathbb{f}(\tbu)),
\end{align}
where $\sigma_{ess}(D\mathbb{f}(\tbu))$ is the essential spectrum of $D\mathbb{f}(\tbu)$ and $Eig(D\mathbb{f}(\tbu)$ its eigenvalues. In fact, using the semi-linearity of \eqref{eq : gen_model}  (cf. Assumptions 1 and 2 in \citep{cadiot2025stabilityanalysislocalizedsolutions}) and Weyl's perturbation theory (cf. \citep{kato2013perturbation} Theorem 5.35 in Chapter IV for instance), we have that $\sigma_{ess}(D\mathbb{f}(\tbu))$ can be computed explicitly as follows
\begin{align}
    \sigma_{ess}(D\mathbb{f}(\tbu)) = \sigma_{ess}(\mathbb{l}_{\tilde{\rho}})= \{ \lambda \in \sigma(\mathscr{l}_{\tilde{\rho}}(\xi)), ~ \text{ for } \xi \in \R\}.
\end{align}
In other words, $\sigma_{ess}(D\mathbb{f}(\tbu))$ corresponds to the union of spectrum of the symbol  $\mathscr{l}_{\tilde{\rho}}(\xi)$, for all $\xi \in \R$. The next lemma provides an explicit enclosure of the real part of such a set.

\begin{lemma}
Assume that $\nu_1 + \nu_2 > 0$. Let $\lambda_{sup}$ be defined as 
    \begin{equation}
        \lambda_{sup} \bydef \sup_{\xi \in \R} \text{Re}\left(\frac{\mathscr{l}_{11}(\xi) + \mathscr{l}_{22}(\xi)  + \sqrt{(\mathscr{l}_{11}(\xi) - \mathscr{l}_{22}(\xi))^2 + 4 \nu_2 \nu_3}}{2}\right),
    \end{equation}
    where Re stands for the real part.
    Then, we have that $\lambda_{sup} <0$ and  that $\left\{\text{Re}(z), ~ z \in \sigma_{ess}(D\mathbb{f}(\tbu))\right\} = (-\infty, \lambda_{sup}]$.
\end{lemma}
\begin{proof}
    The proof relies on the fact that, given $\xi \in \R$, the eigenvalues of $\mathscr{l}_{\tilde{\rho}}(\xi)$ are given by 
    \begin{equation}\label{eq : lambda pm}
       \lambda_{\pm}(\xi) \bydef \frac{\mathscr{l}_{11}(\xi) + \mathscr{l}_{22}(\xi)  \pm \sqrt{(\mathscr{l}_{11}(\xi) - \mathscr{l}_{22}(\xi))^2 + 4 \nu_2 \nu_3}}{2}.
    \end{equation}
    Consequently we simply need to prove that $\lambda_{sup}$ is strictly negative.
Given $\xi \in \R$, if $(\mathscr{l}_{11}(\xi) - \mathscr{l}_{22}(\xi))^2 + 4 \nu_2 \nu_3 \leq 0$, then $\frac{1}{2}\left(\mathscr{l}_{11}(\xi) + \mathscr{l}_{22}(\xi)\right)$ is the real part of $\lambda_{\pm}(\xi)$. But using that $\nu_2 > 0$ by definition and that $\nu_1 + \nu_2 > 0$ by assumption, we get
\[
    \mathscr{l}_{11}(\xi) + \mathscr{l}_{22}(\xi) \leq \nu_1 - \nu_2 <0. 
    \]
    This implies that  $\text{Re}(\lambda_{\pm}(\xi)) \leq \nu_1 - \nu_2 <0$ for all $\xi \in \R$ such that  $(\mathscr{l}_{11}(\xi) - \mathscr{l}_{22}(\xi))^2 + 4 \nu_2 \nu_3 \leq 0$. Note that this also provides a bound for $\lambda_{sup}$ in this case.

    Now, suppose that $\xi \in \R$ is such that $(\mathscr{l}_{11}(\xi) - \mathscr{l}_{22}(\xi))^2 + 4 \nu_2 \nu_3 > 0$. Notice that 
    \begin{align*}
        (\mathscr{l}_{11}(\xi) - \mathscr{l}_{22}(\xi))^2 + 4 \nu_2 \nu_3 = (\mathscr{l}_{11}(\xi) + \mathscr{l}_{22}(\xi))^2 - 4 \mathscr{l}_{den}(\xi).
    \end{align*}
    But using Assumption \eqref{assumption : fourier}, we have that $l_{den}(\xi) \geq \sigma_0 > 0$ for all $\xi \in \R$. Combining the above, this implies that 
    \begin{align*}
       \sqrt{(\mathscr{l}_{11}(\xi) + \mathscr{l}_{22}(\xi))^2 - 4 \mathscr{l}_{den}(\xi)} \leq \sqrt{(\mathscr{l}_{11}(\xi) + \mathscr{l}_{22}(\xi))^2 - 4 \sigma_0} < |\mathscr{l}_{11}(\xi) + \mathscr{l}_{22}(\xi)|.
    \end{align*}
    In particular, we get that $\lambda_{\pm}(\xi) < 0$ whenever $(\mathscr{l}_{11}(\xi) - \mathscr{l}_{22}(\xi))^2 + 4 \nu_2 \nu_3 > 0$.  
    Overall we have obtained that $\lambda_{sup}$ must be strictly negative.
\end{proof}

In the rest of the paper, we assume that $\nu_1 + \nu_2~ > ~0$.
Then, the above lemma provides that the real part of the essential spectrum of $D\mathbb{f}(\tbu)$ is  strictly negative. Consequently, in order to control the spectrum on the imaginary axis, it remains to enclose the eigenvalues (cf. \eqref{eq : decomposition spectrum}). This is the goal of the next section.

\subsection{Control of the eigenvalues}

In this section, given $\delta>0$, we define $\sigma_{\delta}$ be the following set 
\begin{align}\label{def : sigma delta}
    \sigma_{\delta} \bydef \{\lambda \in \mathbb{C}, ~ |\lambda| \leq \delta, ~ Re(\lambda) \geq 0\}.
\end{align}
Then, we are interested in controlling the eigenvalues of  $D\mathbb{f}(\tbu)$ in $\sigma_{\delta}$. In fact, we want to find $\delta_0$ such that if  $\lambda \in Eig(D\mathbb{f}(\tbu))$ with $Re(\lambda) \geq 0$, then $\lambda \in \sigma_{\delta_0}$. We first provide the following preliminary result.

\begin{lemma}\label{lem : bound linv with lamb}
Assume that $\nu_1 + \nu_2 > 0, \nu_2 > 0, $ and $\nu_2 \nu_3 < 0$. 
Let $c_1,c_2,\mathscr{k} > 0$ be defined as 
\begin{align} 
c_1 \bydef \sqrt{-\nu_2 \nu_3}, ~~  c_2 \bydef \sqrt{-\frac{\nu_3}{\nu_2}}, ~~ \mathscr{k} \bydef \max\left(c_2,\frac{1}{c_2}\right).
\end{align}
Also let $\Omega_{\eta}$ be the set
\begin{align}
    \Omega_{\eta} \bydef \left\{ \lambda \in \mathbb{C}, ~ \mathrm{Re}(\lambda) \geq 0 ~ \text{and} ~ \|(\mathscr{l}(\xi) - \lambda I_d)^{-1}\|_{2} \geq \eta \text{ for every } \xi \in \mathbb{R} \right\}.
\end{align}
 Let $\mathscr{R} : \R \to \R$ be the function defined as 
\begin{align}
    \mathscr{R}(\eta) \bydef \sqrt{\left(\frac{\mathscr{k}}{\eta} - |\nu_1|\right)^2 + \left(\frac{\mathscr{k}}{\eta} + c_1\right)^2}.
\end{align}
Then, it follows that $\Omega_{\eta} \subset \overline{B_{\mathscr{R}(\eta)}(0)}$. 
\end{lemma}
\begin{proof}
Let $\mathscr{K} \bydef \mathrm{diag}(1,\mathscr{k})$. Then, it follows that
\begin{align}
    \tilde{\mathscr{l}}(\xi) \bydef \mathscr{K}^{-1} \mathscr{l}(\xi) \mathscr{K} = \begin{bmatrix}
        -\tilde{\rho} |2\pi \xi|^2 + \nu_1 & c_1 \\
        -c_1 & -|2\pi \xi|^2 - \nu_2
    \end{bmatrix}.
\end{align}
Now, let $\lambda \in \mathbb{C}$ such that $Re(\lambda) \geq 0$. Since $\mathscr{l}(\xi)$ has negative eigenvalues, we obtain that $\mathscr{l}(\xi) - \lambda I_d$ is invertible and 
\begin{align}
    \|(\mathscr{l}(\xi) - \lambda I_d)^{-1}\|_{2} \leq \|\mathscr{K}\|_{2} \|\mathscr{K}^{-1}\|_{2}\|(\tilde{\mathscr{l}}(\xi) - \lambda I_d)^{-1}\|_{2} \leq \mathscr{k} \|(\tilde{\mathscr{l}}(\xi) - \lambda I_d)^{-1}\|_{2}.
\end{align}
We then use of Cauchy-Schwarz inequality and obtain
\begin{align}
    \|(\tilde{\mathscr{l}}(\xi) - \lambda I_d)^{-1}\|_{2} \leq \frac{1}{\sqrt{(\mathrm{Re}(\lambda) + |\nu_1|)^2 + (\mathrm{Im}(\lambda) - c_1)^2}}
\end{align}
for all $\xi \in \R$.
Using the above, it follows that
\begin{align}
    \|(\mathscr{l}(\xi) - \lambda I_d)^{-1}\|_{2} \leq \frac{\mathscr{k}}{\sqrt{(\mathrm{Re}(\lambda) + |\nu_1|)^2 + (\mathrm{Im}(\lambda) - c_1)^2}} \leq \frac{\mathscr{k}}{\max(\mathrm{Re}(\lambda) + |\nu_1|, \mathrm{Im}(\lambda) - c_1)}.
\end{align}
Now, suppose by way of contradiction that $\Omega_{\eta}$ is not a subset of $\overline{B_{\mathscr{R}(\eta)}(0)}$. That is, suppose that $\mathrm{Re}(\lambda) + |\nu_1| > \frac{s}{\eta}$ or $\mathrm{Im}(\lambda) - c_1 > \frac{s}{\eta}$. Then, $\|(\mathscr{l}(\xi) - \lambda I_d)^{-1}\|_{2} < \eta$ for all $\xi \in \mathbb{R}$, and hence $\lambda \notin \Omega_{\eta}$. This concludes the proof.
\end{proof}
The previous lemma can now be applied to our case. In particular, we are able to determine $\delta_0$ in such a way that $\delta_0$ contains all the eigenvalues with positive real part.
\begin{lemma}
Let $r_0$ be given in Theorem \ref{th : saddle node in gyl}, and let  $\lambda \in  Eig(D\mathbb{f}(\tbu))$ with $Re(\lambda) \geq 0$, then 
    \begin{align}\label{def : delta zero}
        |\lambda| \leq \mathscr{R} \left(\sqrt{2(\|\bar{V}_1\|_1^2 +\|\bar{V}_2\|_1^2 )} + \sqrt{5}\kappa \sqrt{\mathcal{Z}_{2,3}} r_0 + 3 \kappa \|\mathscr{l}_{\tilde{\rho}}^{-1}\|_{\mathcal{M}_2 }r_0^2\right) \bydef \delta_0.
    \end{align}
    where $\mathscr{R}(\cdot)$ is given in Lemma  \ref{lem : bound linv with lamb}.
\end{lemma}

\begin{proof}
Let $\mathbf{u} \in L^2_e$ be an associated eigenfunction to $\lambda$ such that $\|\mathbf{u}\|_2 = 1$. In fact, we have that 
\begin{align*}
    \mathbb{l}_{\tilde{\rho}}\mathbf{u} - \lambda \mathbf{u} + D\mathbb{g}(\tilde{\mathbf{u}})\mathbf{u} = 0.
\end{align*}
In fact, this implies that 
\begin{align*}
    \frac{1}{\|(\mathbb{l}_{\tilde{\rho}}-\lambda I)^{-1}\|_2} \leq \|(\mathbb{l}_{\tilde{\rho}}-\lambda I)\mathbf{u}\|_2 \leq \|D\mathbb{g}(\tilde{\mathbf{u}})\mathbf{u}\|_2 \leq \|D\mathbb{g}(\tilde{\mathbf{u}})\|_2
\end{align*}
since $\|\mathbf{u}\|_2 = 1.$
Now, using the proof of Lemma \ref{lem : Z2 patterns}, we notice that 
\begin{align*}
    \|D\mathbb{g}(\tilde{\mathbf{u}})\|_2 &\leq \|D\mathbb{g}(\bar{\mathbf{u}})\|_2 + \|D\mathbb{g}(\tilde{\mathbf{u}}) - D\mathbb{g}(\bar{\mathbf{u}})\|_2\\
    &\leq \sqrt{2(\|\bar{V}_1\|_1^2 +\|\bar{V}_2\|_1^2 )} + \sqrt{5}\kappa \sqrt{\mathcal{Z}_{2,3}} r_0 + 3 \kappa \|\mathscr{l}_{\tilde{\rho}}^{-1}\|_{\mathcal{M}_2 }r_0^2.
\end{align*}
Choosing $\eta \bydef \sqrt{2(\|\bar{V}_1\|_1^2 +\|\bar{V}_2\|_1^2 )} + \sqrt{5}\kappa \sqrt{\mathcal{Z}_{2,3}} r_0 + 3 \kappa \|\mathscr{l}_{\tilde{\rho}}^{-1}\|_{\mathcal{M}_2 }r_0^2$, we conclude the proof using Lemma \ref{lem : bound linv with lamb}.
\end{proof}

Letting $\delta_0$ be defined as in \eqref{def : delta zero}, then we have that $\sigma_{\delta_0}$ contains all the eigenvalues of $D\mathbb{f}(\tilde{\mathbf{u}})$ on the imaginary axis, as well as, more generally, the ones with positive real part. Now, we focus on such a subset, and enclose the eigenvalues in its neighborhood (our goal being to prove that $0$ is the only eigenvalue on the imaginary axis).

Now in order to obtain a precise control on the set $\sigma_{\delta_0} \cap \sigma(D\mathbb{f}(\tbu))$, we use the spectrum of $Df(\bar{\mathbf{U}})$, as presented in \citep{cadiot2025stabilityanalysislocalizedsolutions}. For this purpose, we first compute a pseudo-diagonalization of $Df(\bar{\mathbf{U}})$, that is $P$ and $P^{-1}$ such that 
\begin{align}
    \mathcal{D} = P^{-1} Df(\bar{\mathbf{U}}) P
\end{align}
is close to being diagonal. In fact, we make the following choice for the construction of $P$
\begin{align}
    P = P^N + \bpi^{>N}
\end{align}
where $P^N = \bpi^{\leq N} P^N  \bpi^{\leq N}$, is essentially a square matrix of size $N+1$. Then, assuming that $P^N : \bpi^{\leq N}\ell^2 \to  \bpi^{\leq N}\ell^2$ is invertible, we denote $(P^N)^{-1 } : \bpi^{\leq N}\ell^2 \to  \bpi^{\leq N}\ell^2$ its inverse and get that $P^{-1}$ is given as follows 
\begin{align}
    P^{-1}= (P^N)^{-1} + \bpi^{>N}.
\end{align}
Then, using similar notations as in \citep{cadiot2025stabilityanalysislocalizedsolutions}, we define $S$ as the diagonal matrix having the same diagonal as $\mathcal{D}$, that is 
\begin{align*}
    (SU)_n = \lambda_n U_n, \text{ where } \lambda_n \bydef \mathcal{D}_{n,n} \text{ for all } n \in \mathbb{N}_0.
\end{align*}
Moreover, we define $R$ as 
\begin{align*}
    R \bydef \mathcal{D} - S.
\end{align*}
Using such a pseudo-diagonalization, \citep{cadiot2025stabilityanalysislocalizedsolutions} provides a Gershgorin-type of approach for enclosing the eigenvalues away from the essential spectrum. More specifically, we recall the following result which will allow us to construct disks around $(\lambda_n)_{n \in \mathbb{N}_0}$ containing the eigenvalues of $D\mathbb{f}(\tilde{\mathbf{u}})$. In particular, we will be able to count the eigenvalues, taking into account multiplicity. Before stating the central result of this section, we provide the following preliminary result.
\begin{lemma}\label{lem : ift_L_inverse for spectrum}
Suppose that $\tilde{\rho} \neq 1$ and define $\mathscr{l}_{den, \mu}$ as 
 \begin{align}
    \mathscr{l}_{den, \mu}(\xi) \bydef (-\tilde{\rho} (2\pi \xi)^2 + \nu_1 - \mu)( -(2\pi\xi)^2 - \nu_2 - \mu) - \nu_2 \nu_3 \text{ for all } \xi \in \R.
 \end{align}
 In fact, $ \mathscr{l}_{den, \mu}(\xi)$ is the determinant of $\mathscr{l}_{\tilde{\rho}}(\xi) - \mu I_d$. Moreover, given $\mu \in \mathbb{C}$, let us define
\begin{align}
     \xi_{\pm}(\mu) &\bydef \frac{-((\tilde{\rho}+1)\mu + \tilde{\rho} \nu_2 - \nu_1) \pm \sqrt{((\tilde{\rho}-1)\mu + \tilde{\rho} \nu_2 + \nu_1)^2  + 4\tilde{\rho} \nu_2 \nu_3}}{2\tilde{\rho}}\\
     \alpha_{\pm}(\mu) &\bydef \sqrt{\frac{|\xi_{\pm}(\mu)| - \text{Re}(\xi_{\pm}(\mu))}{2}} - i \operatorname{sgn}(\text{Im}(\xi_{\pm}(\mu))) \sqrt{\frac{|\xi_{\pm}(\mu)| + \text{Re}(\xi_{\pm}(\mu))}{2}}.
\end{align}
Finally, recalling $\delta_0$ defined in \eqref{def : delta zero}, let $\mathcal{I}(\delta_0) \bydef \{ ix, ~ |x| \leq \delta_0  \} \cup \left\{ \frac{\pm \sqrt{-4\tilde{\rho}\nu_2 \nu_3} - \tilde{\rho} \nu_2 - \nu_1}{\tilde{\rho} -1} \right\}$.
Then, we have that
    \begin{equation*}
\left|\mathcal{F}^{-1}\left(\frac{d_1|2\pi\xi|^2 + d_2}{\mathscr{l}_{den, \mu}(\xi) }\right)(x)\right|\leq \tilde{C}_0e^{-\tilde{a}|x|},
    \end{equation*}
    for all $\mu \in \sigma_{\delta_0}$,
    with $\tilde{a}$ and $\tilde{C}_0$ given as 
    \begin{align*}
        \tilde{a} &\bydef \inf_{\mu \in \mathcal{I}(\delta_0)} \min\left\{\sqrt{\frac{|\xi_{-}(\mu)| - \text{Re}(\xi_{-}(\mu))}{2}}, ~ \sqrt{\frac{|\xi_{+}(\mu)| - \text{Re}(\xi_{+}(\mu))}{2}}\right\} \\
        \tilde{C}_0 &\bydef  \begin{cases} 
        \sup_{\mu  \in \mathcal{I}(\delta_0)} \left|\frac{d_1}{2\tilde{\rho} \alpha_{\pm}(\mu)}\right| & \mu  \in \left\{ \frac{\pm \sqrt{-4\tilde{\rho}\nu_2 \nu_3} - \tilde{\rho} \nu_2 - \nu_1}{\tilde{\rho} -1} \right\} \\
        \sup_{\mu  \in \mathcal{I}(\delta_0)} \left| \frac{1}{\tilde{\rho}(\xi_+(\mu) - \xi_-(\mu))}\left(\frac{d_1 \xi_+(\mu) + d_2}{2\alpha_+}  - \frac{d_1 \xi_-(\mu) + d_2}{2\alpha_-} \right) \right| & \mu \in \{ix, ~ |x| \leq \delta_0\}
        \end{cases}.
    \end{align*}
\end{lemma}

\begin{proof}
First, note that $\xi_{\pm}$ provides the roots of $l_{den,\mu}$:
\begin{align*}
    l_{den,\mu} = \tilde{\rho}( (2\pi\xi)^2 - \xi_+(\mu))((2\pi\xi)^2 - \xi_-(\mu)). 
\end{align*}
Using the above factorization, we get
\begin{align*}
    \frac{d_1|2\pi\xi|^2 + d_2}{\mathscr{l}_{den, \mu}(\xi)} &= \frac{d_1(|2\pi\xi|^2-\xi_+) + d_2 + d_1 \xi_+}{\rho ((2\pi\xi)^2 - \xi_-)((2\pi\xi)^2 - \xi_+)}\\
    &=   \frac{d_1}{\tilde{\rho} ( (2\pi\xi)^2 - \xi_-)} + \frac{d_2 + d_1\xi_+}{\tilde{\rho}(\xi_+ - \xi_-)} \left( \frac{1}{(2\pi\xi)^2 - \xi_+} - \frac{1}{(2\pi\xi)^2 - \xi_-} \right)\\
    & = \frac{1}{\tilde{\rho}(\xi_+ - \xi_-)} \left(\frac{d_1 \xi_+ + d_2}{(2\pi\xi)^2 - \xi_+} - \frac{d_1 \xi_- + d_2}{(2\pi\xi)^2 - \xi_-}\right).
\end{align*}
Now, note that $\alpha_{\pm} = \sqrt{-\xi_{\pm}}$ such that $\text{Re}(\alpha_{\pm}) > 0$. Using that $\mathcal{F}^{-1}(\frac{1}{(2\pi\xi)^2 +z^2})(x) = \frac{e^{-z|x|}}{2z}$, we obtain that
{\small\begin{align*}
    \mathcal{F}^{-1}\left(\frac{d_1|2\pi\xi|^2 + d_2}{\mathscr{l}_{den, \mu}(\xi) }\right)(x) = \frac{1}{\rho(\xi_+(\mu) - \xi_-(\mu))}\left(\frac{d_1 \xi_+(\mu) + d_2}{2\alpha_+(\mu)} e^{-\alpha_+(\mu) |x|} - \frac{d_1 \xi_-(\mu) + d_2}{2\alpha_-(\mu)} e^{-\alpha_-(\mu)|x|} \right).
\end{align*}}
The above implies that  
\begin{equation*}
\left|\mathcal{F}^{-1}\left(\frac{d_1|2\pi\xi|^2 + d_2}{\mathscr{l}_{den, \mu}(\xi) }\right)(x)\right|\leq C(\mu) e^{-a(\mu) |x|},
    \end{equation*}
    where 
    \begin{align}
        C(\mu) &= \left| \frac{1}{\tilde{\rho}(\xi_+(\mu) - \xi_-(\mu))}\left(\frac{d_1 \xi_+(\mu) + d_2}{2\alpha_+(\mu)}  - \frac{d_1 \xi_-(\mu) + d_2}{2\alpha_-(\mu)} \right) \right|\label{def : Cmu} \\
        a(\mu) &= \min\{\text{Re}(\alpha_-(\mu)), ~ \text{Re}(\alpha_+(\mu)) \}.\label{def : amu}
    \end{align}
    Now we prove that $\inf_{\mu \in \sigma_{\delta_0}} a(\mu) = \inf_{\mu \in \mathcal{I}(\delta_0)} a(\mu)$ and $\sup_{\mu \in \sigma_{\delta_0}} C(\mu) = \sup_{\mu \in \mathcal{I}(\delta_0)} C(\mu)$. Let $D(\mu)$ be defined as $D(\mu) = ((\tilde{\rho}-1)\mu + \tilde{\rho} \nu_2 + \nu_1)^2  + 4\tilde{\rho} \nu_2 \nu_3.$ Moreover, let $I_0 = \{\mu, ~ \text{Re}(\mu) \geq 0\} \setminus \{\mu, ~ D(\mu) = 0\}$. Then, notice that $\xi_{\pm}$ is holomorphic on $I_0$. This is because in $I_0$, we removed the $D(\mu) = 0$ terms, and the square root is analytic (hence holomorphic) away from $0$. Hence $\alpha_{\pm}$ is harmonic on $I_0$ (when seen as a 2D function with the real and imaginary parts of $\mu$), as the real part of an holomorphic function. Using the minimum principle, we obtain that $a(\mu)$ is minimized on the boundary of $I_0$, that is on the imaginary axis or at $\{\mu, ~ D(\mu) = 0\}$. Using a similar reasoning, we obtain that $C(\mu)$ is maximized   on the imaginary axis or at $\{\mu, ~ D(\mu) = 0\}$. Indeed, we have that 
\begin{align*}
    C(\mu)^2= \frac{1}{\tilde{\rho}^2} \left|\frac{h(\xi_+) - h(\xi_-)}{\xi_+ - \xi_-}\right|^2, \text{ where } h(z) = \frac{d_1 z + d_2}{2\sqrt{-z}},
\end{align*}
and we can use the maximum modulus principle. We conclude the proof noticing that $\{\mu, ~ D(\mu) = 0\} = \left\{ \frac{\pm \sqrt{-4\tilde{\rho}\nu_2 \nu_3} - \tilde{\rho} \nu_2 - \nu_1}{\tilde{\rho} -1} \right\}$.
\par We also remark that in the case that $D(\mu) = 0$, we have a simplification in the $\tilde{C}_0$ formula. In this case, we let $\xi_{+}(\mu) = \xi_{-}(\mu) = \xi(\mu)$ and $\alpha_{+}(\mu) = \alpha_{-}(\mu) = \alpha(\mu)$. Then, we get
{\footnotesize\begin{align}
    \left| \frac{1}{\tilde{\rho}(\xi_+(\mu) - \xi_-(\mu))}\left(\frac{d_1 \xi_+(\mu) + d_2}{2\alpha_+(\mu)}  - \frac{d_1 \xi_-(\mu) + d_2}{2\alpha_-(\mu)} \right) \right| &= \left| \frac{1}{\tilde{\rho}(\xi(\mu) - \xi(\mu))}\left(\frac{d_1 \xi(\mu) + d_2}{2\alpha(\mu)}  - \frac{d_1 \xi(\mu) + d_2}{2\alpha(\mu)} \right) \right| \\
    &=\left| \frac{1}{\tilde{\rho}(\xi(\mu) - \xi(\mu))}\left(\frac{d_1 (\xi(\mu) - \xi(\mu)) }{2\alpha(\mu)}\right) \right| \\
    &=\left|\frac{d_1}{2\tilde{\rho} \alpha(\mu)}\right|,
\end{align}}
which we can use to do the evaluation. To compute $\tilde{a}$ and $\tilde{C}_0$, we use a combination of rigorous numerics along with the result of Theorem \ref{th : saddle node in gyl}. In particular, as our computations rely on using $\tilde{\rho}$, which we proved to exist in Theorem \ref{th : saddle node in gyl}, we can replace $\tilde{\rho} = \overline{\rho} + h$ where $|h| \leq r_0$. The technical details for this computation can be found in Appendix \ref{apen : tildea tilde C0}.
\end{proof}
 Now, we recall the central result of \citep{cadiot2025stabilityanalysislocalizedsolutions}, using the notations of the present manuscript. This result allows to construct Gershgorin disks centered at  $(\lambda_k)$ containing the eigenvalues of $D\mathbb{f}(\tilde{\mathbf{u}})$ in $\sigma_{\delta_0}$.
\begin{lemma}\label{lem : link of spectrum unbounded to U0}
  Let $\delta_0$ be defined as in \eqref{def : delta zero} and let $\mathcal{J} = \sigma_{\delta_0}$. Furthermore, let $t \in \mathbb{C}$ be such that $S+tI_d$ is invertible. Moreover, let $r_0$ and $\tilde{\mathbf{u}}$ be given as in Theorem \ref{th : saddle node in gyl}.
  Now, let us introduce various constants satisfying the following 
  \begin{align*}
     \mathcal{Z}_{u,1} &\geq   \sup_{\mu \in \mathcal{J}} \left\|\out \left(\mathbb{l}_{\tilde{\rho}} - \mu I_d\right)^{-1} D\mathbb{g}(\bar{\mathbf{u}})\right\|_2\\
\mathcal{Z}_{u,2} & \geq  \sup_{\mu \in \mathcal{J}} \| \cha \left(\Gamma^\dagger\left((l_{\tilde{\rho}}- \mu I_d )^{-1}\right) - (\mathbb{l}_{\tilde{\rho}}- \mu I_d)^{-1}\right)D\mathbb{g}(\bar{\mathbf{u}}) \|_2\\
\mathcal{Z}_{u,3} & \geq  \sup_{\mu \in \mathcal{J}}  \|\bpi^{\leq N} (S+ tI_d)^{-1}P^{-1}(l_{\tilde{\rho}}-\mu I_d)\|_2 \| \cha \left(\Gamma^\dagger\left((l_{\tilde{\rho}}- \mu I_d )^{-1}\right) - (\mathbb{l}_{\tilde{\rho}}- \mu I_{d})^{-1}\right)D\mathbb{g}(\bar{\mathbf{u}}) \|_2\\
     \mathcal{C}_1   &\geq \sup_{\mu \in \mathcal{J}} \frac{1}{r_0}  \left\| \left(\mathbb{l}_{\tilde{\rho}} - \mu I_d\right)^{-1} \left(D\mathbb{g}(\bar{\mathbf{u}})-D\mathbb{g}(\tilde{\mathbf{u}})\right)\right\|_2\\
     \mathcal{C}_2   &\geq  \sup_{\mu \in \mathcal{J}} \frac{1}{r_0}  \|\bpi^{\leq N} (S+ tI_d)^{-1}P^{-1}(l_{\tilde{\rho}}-\mu I_d)\|_2 \left\| \cha \left(\mathbb{l}_{\tilde{\rho}} - \mu I_d\right)^{-1} \left(D\mathbb{g}(\bar{\mathbf{u}})-D\mathbb{g}(\tilde{\mathbf{u}})\right)\right\|_2\\
     Z_{1,1} &\geq \sup_{\mu \in \mathcal{J}}  \|\bpi^{>N} (l_{\tilde{\rho}} -\mu I_d)^{-1}R \bpi^{\leq N}\|_2, ~~~~ Z_{1,2} \geq  \sup_{\mu \in \mathcal{J}} \|\bpi^{>N} (l_{\tilde{\rho}} -\mu I_d)^{-1}R\bpi^{>N}\|_2\\
       Z_{1,3} &\geq\| \bpi^{\leq N}(S+ tI_d)^{-1}R\bpi^{\leq N}\|_2, ~~~~ Z_{1,4} \geq \| \bpi^{\leq N}(S+ tI_d)^{-1}R\bpi^{>N}\|_2.
  \end{align*}
  If 
  \begin{equation}\label{eq : condition C1 r0}
      \mathcal{C}_1r_0 <1,
  \end{equation}
 we define $\beta_1$ as 
$\beta_1 \bydef\frac{\mathcal{Z}_{u,1}+\mathcal{C}_1r_0}{1-\mathcal{C}_1r_0}.$
If in addition
\begin{equation}\label{eq : condition for kappa1 and C1 r0}
    1-Z_{1,2}-\mathcal{Z}_{u,2} - (1+\beta_1^2)^\frac{1}{2}\mathcal{C}_1r_0 > 0,
\end{equation}
 then we define
\begin{equation}\label{eq : value for epsilon infinity n}
\begin{aligned}
\beta_2 &\bydef \frac{Z_{1,1} + (\mathcal{Z}_{u,2} + (1+\beta_1^2)^\frac{1}{2}\mathcal{C}_1r_0)\|P^N\|_2}{1-Z_{1,2}-\mathcal{Z}_{u,2} - (1+\beta_1^2)^\frac{1}{2}\mathcal{C}_1r_0}\\
      \epsilon_n^{(\infty)} &\bydef |\lambda_n + t|\left(Z_{1,3} + Z_{1,4} \beta_2 + \left(\mathcal{Z}_{u,3} + \mathcal{C}_2r_0(1+\beta_1^2)^\frac{1}{2} \right)\left(\|P^N\|_2 +   \beta_2 \right)   \right)\\
      \epsilon_n^{(q)} &\bydef |\lambda_n + t|\left(Z_{1,3} + Z_{1,4} \frac{Z_{1,1} + 2\mathcal{Z}_{u,2}\|P^N\|_2}{1-Z_{1,2}-2\mathcal{Z}_{u,2}} +  2\mathcal{Z}_{u,3}\left(\|P^N\|_2 +   \frac{Z_{1,1} + 2\mathcal{Z}_{u,2}\|P^N\|_2}{1-Z_{1,2}-2\mathcal{Z}_{u,2}} \right) \right) \\
      \epsilon_n &\bydef \max\left\{\epsilon_n^{(\infty)}, \epsilon_n^{(q)}\right\}
\end{aligned}
\end{equation}
for all $n \in \mathbb{N}_0$.

 Let  $k \in \mathbb{N}$ and $I \subset \mathbb{N}_0$ such that $|I| = k$. If  $\left(\cup_{n \in I} \overline{B_{\epsilon_n}(\lambda_n)}\right) \bigcap\left( \cup_{n \in \mathbb{N}_0 \setminus I} \overline{B_{\epsilon_n}(\lambda_n)}\right) = \varnothing$ and $\cup_{n \in I} \overline{B_{\epsilon_n}(\lambda_n)} \subset \mathcal{J}$, then there are exactly $k$ eigenvalues of $D\mathbb{f}(\tilde{\mathbf{u}})$ in $\cup_{n \in I} \overline{B_{\epsilon_n}(\lambda_n)} \subset  \mathcal{J}$ counted with multiplicity.
 \end{lemma}
\begin{proof}
The proof can be found in \cite{cadiot2025stabilityanalysislocalizedsolutions}.
\end{proof}

\begin{remark}
    The bounds $\mathcal{Z}_{u,1}, \mathcal{Z}_{u,2}, \mathcal{Z}_{u,3}$ introduced in the previous lemma are not the same as the ones introduced in Lemma \ref{lem : old Zu}. However, they are closely related and can be computed in a very similar fashion. Indeed, their computation is based on the proof of Lemma \ref{lem : old Zu}, combined with the estimates of Lemma \ref{lem : ift_L_inverse for spectrum}. Note that the additional factor $2$ in the inequalities for  $\mathcal{Z}_{u,1}, \mathcal{Z}_{u,2}, \mathcal{Z}_{u,3}$ allows to provide upper bounds for  $\mathcal{Z}_{u,1}^{(q)}, \mathcal{Z}_{u,2}^{(q)}, \mathcal{Z}_{u,3}^{(q)}$ given in Lemma 4.2 of \citep{cadiot2025stabilityanalysislocalizedsolutions} (see the proof of Theorem 5.2 in the aforementioned paper).  Details for the computation of the bounds of Lemma \ref{lem : link of spectrum unbounded to U0} are given in Appendix \ref{apen : link of spectrum}, and are implemented  explicitly in the code \citep{julia_blanco_fassler}. 
    \end{remark}
\begin{remark}
    In practice, we usually choose $t$ slightly larger than $\lambda_{sup}$ (e.g. $t = \lambda_{sup} - 0.01$). This initial choice of $t$ may be quite coarse, leading to a poor enclosure of the spectrum. In particular, one may be unable to draw conclusions about spectral localization if the Gershgorin disks overlap or intersect the imaginary axis.
To address this issue, one can apply Lemma \ref{lem : link of spectrum unbounded to U0} using the initial value of $t$, thereby obtaining an enclosure of the largest eigenvalue with positive real part. One can then re-choose $t$ and re-run the proof to obtain a sharper enclosure of the eigenvalues with positive real part. Iterating this procedure yields progressively tighter enclosures of the eigenvalues, resulting in more precise Gershgorin disks.

\end{remark}
\subsection{Existence proof of a saddle-node}

Using the above analysis, we are able to rigorously enclose the spectrum of $D\mathbb{f}(\tilde{\rho},\tilde{\mathbf{u}})$, for a given zero $(\tilde{\rho},\tilde{\mathbf{u}}) \in H_1$ of $\mathbb{f}$. This is particularly useful for verifying that the zeros of \eqref{eq:saddle_node_map} are indeed saddle-nodes, following  Definition \ref{lem : saddle node def}. In fact, we obtain the following result.
\begin{theorem}
    The zero $(\tilde{\rho},\tilde{\mathbf{u}})$ of $\mathbb{f}$ established in Theorem \ref{th : saddle node in gyl} is a saddle-node. Moreover, $\sigma(D\mathbb{f}(\tilde{\rho},\tilde{\mathbf{u}})) \setminus \{0\}$ is contained in the left-half part of the complex plane.
\end{theorem}

\begin{proof}
    By construction, we know that $\tilde{\mathbf{x}}$ in Theorem \ref{th : saddle node in gyl} is an isolated zero of \eqref{eq:saddle_node_map}. Using Definition \ref{lem : saddle node def}, it remains to show that $D_u\mathbb{f}(\tilde{\rho}, \tilde{\mathbf{u}})$ does not possess eigenvalues on the imaginary axis except than zero. For this matter, we implement the bounds of Lemma \ref{lem : link of spectrum unbounded to U0} in \citep{julia_blanco_fassler}. We obtain the all the eigenvalues of $D_u\mathbb{f}(\tilde{\rho},\tilde{\mathbf{u}})$ are contained on the left-half part of the complex plane. Moreover, we verify that the disks obtained in Lemma \ref{lem : link of spectrum unbounded to U0} do not cross the imaginary axis, except one of them which contains zero. Since zero is an eigenvalue, using Theorem \ref{th : saddle node in gyl}, we obtain that $D_u\mathbb{f}(\tilde{\rho}, \tilde{\mathbf{u}})$ does not possess eigenvalues on the imaginary axis. 
\end{proof}

The above theorem not only allows to validate the existence of a regular saddle-node bifurcation (by eliminating a possible degenerate Hopf bifurcation), but it also allows to control the change of stability. In fact, going back to the parabolic system \eqref{eq:original_system} and using the analysis of Section I.7.1 of \citep{hansjorb_bifurcation}, we obtain that around $(\tilde{\rho},\tilde{\mathbf{u}})$, the number of unstable directions of  $D_u\mathbb{f}(\tilde{\rho},\tilde{\mathbf{u}})$ goes from zero to one (that is the branch of  stationary solutions looses stability after the saddle-node). This is a direct consequence of \citep{hansjorb_bifurcation} and the fact that a simple eigenvalue crosses zero with non-zero speed (see equation I.7.30). This is of particular interest when investigating the local dynamics (in time) of \eqref{eq:original_system}. We push further this idea and investigate the  stability (with respect to even perturbations) of some of the localized solutions for which the existence was established in Section \ref{sec : patterns}.

\subsection{Stability analysis}
In this section, we focus on the stability of some localized stationary solutions of Section \ref{sec : patterns}. In particular, given $\tilde{\mathbf{u}}$ proven in Section \ref{sec : patterns}, that is there exists $r_0$ such that $\|\tilde{\mathbf{u}}-\bar{\mathbf{u}}\|_{\mathcal{H}} \leq r_0$, then we use Lemma \ref{lem : link of spectrum unbounded to U0} to control the spectrum of $D\mathbb{f}(\tilde{\mathbf{u}}) : L^2_e \to L^2_e$.
In fact, this allows us to conclude about the (nonlinear) stability of $\tilde{\mathbf{u}}$ with respect to even perturbations in $L^2_e$. In the cases for which $\tilde{\mathbf{u}}$ is unstable, we count the number of unstable directions. More specifically, we obtain the following results for which computational details are given in \citep{julia_blanco_fassler}.

\begin{theorem}
    The solution $\tilde{u}_{1} \in \mathcal{H}_e$ obtained in Theorem \ref{th : existence proofs}  is stable with respect to even perturbations. In particular, the elements of the spectrum of $D\mathbb{f}(\tilde{\mathbf{u}}) : L^2_e \to L^2_e$ possess a strictly negative real part.
\end{theorem}

\begin{proof}
    We prove that all the disks constructed thanks to Lemma \ref{lem : link of spectrum unbounded to U0} are strictly contained on the left-part of the complex plane. 
\end{proof}
\begin{theorem}
    The solution $\tilde{u}_{5} \in \mathcal{H}_e$ obtained in Theorem \ref{th : existence proofs} is unstable. In particular, $D\mathbb{f}(\tilde{\mathbf{u}}) : L^2_e \to L^2_e$ possesses exactly one positive eigenvalue.
\end{theorem}

\begin{proof}
We construct the disks of  Lemma \ref{lem : link of spectrum unbounded to U0} and prove that one of them, say $\overline{B_{\epsilon_0}(\lambda_0)}$, is on strictly included on the right-part of the complex plane. More specifically, we get $\epsilon_0 = 6.24 \times 10^{-3}$ and $\lambda_0 = 0.708705$.
Then, we prove that $\overline{B_{\epsilon_0}(\lambda_0)}$  is disjoint from the rest of the disks. Finally, we conclude the proof by proving the rest of the disks is strictly contained on the left-part of the complex plane. 
\end{proof}

\section{Acknowledgments}

The authors wish to thank Professor Nicolas Verschueren van Rees for his insights and fruitful discussions.  MC was supported by the ANR project CAPPS: ANR-23-CE40-0004-01 and by the FMJH :  ANR-22-EXES-0013.

\appendix

\renewcommand{\theequation}{A.\arabic{equation}}
\setcounter{equation}{0}
\section{Computing the Bound \texorpdfstring{$\mathcal{Z}_2$}{Z2} for Saddle Nodes}\label{apen : saddle node computations}
We prove Lemma \ref{lem : bound Z_2 saddle}. In particular, we compute the $\mathcal{Z}_2$ bound.
\begin{proof}
Let $\mbf{h} \bydef (\eta,\mbf{h}_1,\mbf{h}_2) \bydef (\eta,h_{1,1},h_{1,2},h_{2,1},h_{2,2}) \in B_r(\overline{\mathbf{x}})$. Now, observe that
{\footnotesize\begin{align}
    &\|\mathbb{A}(D\mathbb{F}(\mbf{h} + \overline{\mathbf{x}}) - D\mathbb{F}(\overline{\mathbf{x}}))\|_{H_1} \\
    &\leq \left\| \mathbb{B}\begin{bmatrix}
        0 & 0 & 0\\
        D_{\rho} \mathbb{f}(\eta + \overline{\rho},\mbf{h}_1 + \overline{\mathbf{u}}) - D_{\rho} \mathbb{f}(\overline{\rho},\overline{\mathbf{u}}) & 0 & 0\\
        D_{\rho,\mbf{u}} \mathbb{f}(\eta + \overline{\rho},\mbf{h}_1 + \overline{\mathbf{u}}) (\mbf{h}_2 + \overline{\mathbf{w}}) - D_{\rho,\mbf{u}}\mathbb{f}(\overline{\rho},\overline{\mathbf{u}}) \overline{\mathbf{w}}) & 0 & 0
    \end{bmatrix}\right\|_{H_1,H_2} \\
    &\hspace{+0.5cm}+ \left\|\mathbb{B}\begin{bmatrix}
        0 & 0 & 0 \\
        0 & D_{\mbf{u}} \mathbb{f}(\eta+\overline{\rho},\mbf{h}_1 + \overline{\mathbf{u}}) - D_{\mbf{u}} \mathbb{f}(\overline{\rho},\overline{\mathbf{u}}) & 0 \\
        0 & D^2_{\mbf{u}} \mathbb{f}(\eta + \overline{\rho},\mbf{h}_1 + \overline{\mathbf{u}}) (\mbf{h}_2 + \overline{\mathbf{w}}) - D^2_{\mbf{u}}\mathbb{f}(\overline{\rho},\overline{\mathbf{u}}) \overline{\mathbf{w}} & D_{\mbf{u}} \mathbb{f}(\eta + \overline{\rho},\mbf{h}_1 + \overline{\mathbf{u}})  - D_{\mbf{u}}\mathbb{f}(\overline{\rho},\overline{\mathbf{u}}) 
    \end{bmatrix}\right\|_{H_1,H_2} \\
    &\leq \left\| \mathbb{B}_{1\to5,24}\right\|_{H_2}\left(
        \left\|\begin{bmatrix} 
        D_{\rho} \mathbb{f}(\eta + \overline{\rho},\mbf{h}_1 + \overline{\mathbf{u}}) - D_{\rho} \mathbb{f}(\overline{\rho},\overline{\mathbf{u}})\\
        D_{\rho,\mbf{u}} \mathbb{f}(\eta + \overline{\rho},\mbf{h}_1 + \overline{\mathbf{u}}) (\mbf{h}_2 + \overline{\mathbf{w}}) - D_{\rho,\mbf{u}}\mathbb{f}(\overline{\rho},\overline{\mathbf{u}}) \overline{\mathbf{w}})\end{bmatrix}\right\|_{2}\right) \\
        &\hspace{+2cm}+ \left\|\mathbb{B}\left(\begin{bmatrix}
        0 & 0 & 0 & 0 & 0 \\
        0 & \eta \Delta & 0 & 0 & 0 \\
        0 & 0 & 0 & 0 & 0 \\
        0 & \eta \Delta & 0 & \eta \Delta & 0 \\
        0 & 0 & 0   & 0 & 0
    \end{bmatrix} + D\mathbb{G}(\mbf{h}_1 + \overline{\mathbf{u}},\mbf{h}_2 + \overline{\mathbf{w}}) - D\mathbb{G}(\overline{\mathbf{u}}, \overline{\mathbf{w}})\right)\right\|_{H_1,H_2}\\
    &\leq \left\| \mathbb{B}_{1\to5,2\to5}\right\|_{H_2}\left(
        \left\|\begin{bmatrix} 
        D_{\rho} \mathbb{f}(\eta + \overline{\rho},\mbf{h}_1 + \overline{\mathbf{u}}) - D_{\rho} \mathbb{f}(\overline{\rho},\overline{\mathbf{u}})\\
        D_{\rho,\mbf{u}} \mathbb{f}(\eta + \overline{\rho},\mbf{h}_1 + \overline{\mathbf{u}}) (\mbf{h}_2 + \overline{\mathbf{w}}) - D_{\rho,\mbf{u}}\mathbb{f}(\overline{\rho},\overline{\mathbf{u}}) \overline{\mathbf{w}})\end{bmatrix}\right\|_{2}\right) \\
        &\hspace{+2cm}+ \sqrt{3}\|\mathbb{B}_{1\to5,24}\|_{H_2}|\eta|\left\| \begin{bmatrix}
        \Delta & 0 \\
        0 & 0
    \end{bmatrix}\right\|_{\mathcal{H},2} + \left\|\mathbb{B}\left(D\mathbb{G}(\mbf{h}_1 + \overline{\mathbf{u}},\mbf{h}_2 + \overline{\mathbf{w}}) - D\mathbb{G}(\overline{\mathbf{u}}, \overline{\mathbf{w}})\right)\right\|_{H_1,H_2}\label{first step in Z2 saddle}
\end{align}}
where $\mathbb{B}_{1\to5,24}$ is defined in \eqref{def : Z2 saddle definitions} and we used the facts that
{\footnotesize\begin{align}
    &D_{\mbf{u}}\mathbb{f}(\eta+\overline{\rho},\mbf{h}_1 + \overline{\mathbf{u}}) - D_{\mbf{u}}\mathbb{f}(\overline{\rho},\overline{\mathbf{u}}) = \begin{bmatrix}
        \eta \Delta & 0 \\
        0 & 0
    \end{bmatrix} + D_{\mbf{u}}\mathbb{g}(\mbf{h}_1 + \overline{\mathbf{u}}) - D_{\mbf{u}}\mathbb{g}(\overline{\mathbf{u}}) \\
    &D^2_{\mbf{u}} \mathbb{f}(\eta + \overline{\rho},\mbf{h}_1 + \overline{\mathbf{u}}) (\mbf{h}_2 + \overline{\mathbf{w}}) - D^2_{\mbf{u}}\mathbb{f}(\overline{\rho},\overline{\mathbf{u}}) \overline{\mathbf{w}} = \begin{bmatrix}
        \eta \Delta & 0 \\
        0 & 0
    \end{bmatrix} + D^2_{\mbf{u}} \mathbb{g}(\mbf{h}_1 + \overline{\mathbf{u}}) (\mbf{h}_2 + \overline{\mathbf{w}}) - D^2_{\mbf{u}}\mathbb{g}(\overline{\mathbf{u}}) \overline{\mathbf{w}}.
\end{align}}
For the first term in \eqref{first step in Z2 saddle}, we use the Fourier transform to get
{\small\begin{align}
    \left\|\begin{bmatrix} 
        D_{\rho} \mathbb{f}(\eta + \overline{\rho},\mbf{h}_1 + \overline{\mathbf{u}}) - D_{\rho} \mathbb{f}(\overline{\rho},\overline{\mathbf{u}})\\
        D_{\rho,\mbf{u}} \mathbb{f}(\eta + \overline{\rho},\mbf{h}_1 + \overline{\mathbf{u}}) (\mbf{h}_2 + \overline{\mathbf{w}}) - D_{\rho,\mbf{u}}\mathbb{f}(\overline{\rho},\overline{\mathbf{u}}) \overline{\mathbf{w}})\end{bmatrix}\right\|_{2} = \left\|\begin{bmatrix}
         \Delta h_{1,1} \\
        0 \\
         \Delta h_{2,1} \\ 0
    \end{bmatrix}\right\|_{2}
 &= \left\|\begin{bmatrix}
         \Delta & 0 & 0 & 0\\
        0  & 0 & 0 & 0 \\
        0 & 0 & \Delta & 0 \\
        0 & 0 & 0 & 0
    \end{bmatrix}\begin{bmatrix}
        h_{1,1} \\
        h_{1,2} \\
        h_{2,1} \\
        h_{2,2}
    \end{bmatrix}\right\|_{2}
\\&\hspace{-2cm}= \left\|\begin{bmatrix}
        -|2\pi \xi|^2 & 0 & 0 & 0\\
        0  & 0 & 0 & 0 \\
        0 & 0 & -|2\pi \xi|^2 & 0 \\
        0 & 0 & 0 & 0 
    \end{bmatrix}\begin{bmatrix}
        \hat{h}_{1,1} \\
        \hat{h}_{1,2} \\
        \hat{h}_{2,1} \\
        \hat{h}_{2,2}
    \end{bmatrix}\right\|_{2} 
    \\
    &\hspace{-2cm}\leq \left\|\begin{bmatrix}
        -|2\pi \xi|^2 & 0 \\
        0  & 0
    \end{bmatrix}\mathscr{l}_{\overline{\rho}}^{-1}(\xi)\right\|_{\mathcal{M}_1}\left\|\mathscr{l}_{\overline{\rho}}(\xi)\hat{\mbf{h}}\right\|_{2}   \\
    &\hspace{-2cm}\leq \left\|\begin{bmatrix}
        -|2\pi \xi|^2 & 0 \\
        0  & 0
    \end{bmatrix}\mathscr{l}_{\overline{\rho}}^{-1}(\xi)\right\|_{\mathcal{M}_1}r \bydef \frac{\mathcal{Z}_{2,1} r}{\|\mathbb{B}_{1\to5,24}\|_{H_2}}.
\end{align}}
The quantity $\left\|\begin{bmatrix}
        -|2\pi \xi|^2 & 0 \\
        0  & 0
    \end{bmatrix}\mathscr{l}_{\overline{\rho}}^{-1}(\xi)\right\|_{\mathcal{M}_1}$ can be computed using standard Calculus techniques. Next, we examine the second term in \eqref{first step in Z2 saddle}. 
\begin{align}
     \sqrt{3}\|\mathbb{B}_{1\to5,24}\|_{H_2}|\eta|\left\| \begin{bmatrix}
        \Delta & 0 \\
        0 & 0
    \end{bmatrix}\right\|_{\mathcal{H},2} \leq \sqrt{3}\|\mathbb{B}_{1\to5,24}\|_{H_2}\left\| \begin{bmatrix}
        \Delta & 0 \\
        0 & 0
    \end{bmatrix}\right\|_{\mathcal{H},2} r.
\end{align}
Now, let $\mbf{z} \bydef (z_1,z_2) \in \mathcal{H}$ such that $\|\mbf{z}\|_{\mathcal{H}} = 1$. Then,
\begin{align}
    \left\| \begin{bmatrix}
        \Delta & 0 \\
        0 & 0
    \end{bmatrix}\right\|_{\mathcal{H},2} &= \left\| \begin{bmatrix}
        \Delta & 0 \\
        0 & 0
    \end{bmatrix}\begin{bmatrix}
        z_1 \\ z_2
    \end{bmatrix}\right\|_{2} \leq \left\| \begin{bmatrix}
        -|2\pi \xi|^2 & 0 \\
        0 & 0
    \end{bmatrix}\mathscr{l}_{\overline{\rho}}^{-1}(\xi)\right\|_{\mathcal{M}_1}  = \frac{\mathcal{Z}_{2,1}}{\|\mathbb{B}_{1\to5,24}\|_{H_2}}.
\end{align}
\par Finally, we examine the third term in \eqref{first step in Z2 saddle}. To do so, we first introduce some notation. Let
\begin{align}
    \mathbb{G}_3(\mbf{x}) \bydef \begin{bmatrix}
        0 \\ u_1^2 u_2 \\ - u_1^2 u_2 \\ u_1 (2u_2 w_1 + u_1 w_2) \\ -u_1(2u_2 w_1 + u_1 w_2)
    \end{bmatrix}.
\end{align}
Then, observe that we can write
\begin{align}
    \mathbb{G}_3(\mbf{u},\mbf{w})
  \bydef \mathbb{G}_{3,1}\mbf{x}( \mathbb{G}_{3,2} \mbf{x}\mathbb{G}_{3,3} \mbf{x} + \mathbb{G}_{3,4} \mbf{x} \mathbb{G}_{3,5} \mbf{x})
\end{align}
where
\begin{align}
    &\mathbb{G}_{3,1} = \bydef \begin{bmatrix}
        0 & 0 & 0 & 0 & 0 \\0 &1 & 0 & 0 & 0 \\
        0 & -1 & 0 & 0 & 0 \\
        0 & 1 & 0 & 0 & 0 \\
        0 & -1 & 0 & 0 & 0
    \end{bmatrix} ,~\mathbb{G}_{3,2} \bydef \begin{bmatrix}
        0 & 0 & 0 & 0 & 0 \\
        0 & 1 & 0 & 0 & 0 \\
        0 & 1 & 0 & 0 & 0 \\
        0 & 0 & 2 & 0 & 0 \\
        0 & 0 & 2 & 0 & 0 
    \end{bmatrix},~\mathbb{G}_{3,3} \bydef \begin{bmatrix}
        0 & 0 & 0 &0 & 0 \\
        0 &0 & 1 & 0 & 0 \\
        0 &0 & 1 & 0 & 0 \\
        0 &0 & 0 & 1 & 0 \\
        0 &0 & 0 & 1 & 0 
    \end{bmatrix}\\
    &\mathbb{G}_{3,4}  \bydef \begin{bmatrix}
        0 & 0 & 0 & 0 & 0 \\0 & 0 & 0 & 0 & 0 \\
        0 &0 & 0 & 0 & 0 \\
        0 &1 & 0 & 0 & 0 \\
        0 &1 & 0 & 0 & 0
    \end{bmatrix}
    ,\mathbb{G}_{3,5} \bydef \begin{bmatrix}
        0 & 0 & 0 & 0 & 0 \\0 &0 & 0 & 0 & 0 \\
        0 &0 & 0 & 0 & 0 \\
        0 &0 & 0 & 0 & 1 \\
        0 &0 & 0 & 0 & 1 
    \end{bmatrix}.
\end{align}
Note that the dependency on $\mbf{x}$ is not needed since the above term does not depend on $\rho$; however, we write it for convenience.
Now, let $\mbf{y}_1, \mbf{y}_2,\mbf{y}_3 \in H_1$. Then, we write $G_{3} : (H_1)^3 \to (\mathcal{H}_e)^2$ as
\begin{align}
    G_{3}(\mbf{y}_1,\mbf{y}_2,\mbf{y}_3) \bydef \mathbb{G}_{3,1} \mbf{y}_1 (\mathbb{G}_{3,2} \mbf{y}_2\mathbb{G}_{3,3} \mbf{y}_3 + \mathbb{G}_{3,4} \mbf{y}_2 \mathbb{G}_{3,5} \mbf{y}_3)
\end{align}
for all $\mbf{y}_1,\mbf{y}_2,\mbf{y}_3 \in H_1$. Now, let $\mbf{z} \bydef (\alpha,\mbf{z}_1,\mbf{z}_2)\bydef (\alpha,z_{1,1},z_{1,2},z_{2,1},z_{2,2}) \in H_1, \|\mbf{z}\|_{H_1} = 1$. Similar to the proof of Lemma \ref{lem : Z2 patterns}, observe that
{\small\begin{align}
    (D\mathbb{G}(\mbf{h}_1 + \overline{\mathbf{u}},\mbf{h}_2 + \overline{\mathbf{w}}) - D\mathbb{G}(\overline{\mathbf{u}},\overline{\mathbf{w}})) \mbf{z}= D^2\mathbb{G}(\overline{\mathbf{u}},\overline{\mathbf{w}})(\mbf{h},\mbf{z}) + G_3(\mbf{z},\mbf{h},\mbf{h}) + G_3(\mbf{h},\mbf{z},\mbf{h}) + G_3(\mbf{h},\mbf{h},\mbf{z}).
\end{align}}
Generally, by definition of $G_3$, we have
$\|\mathbb{B} G_3(\mbf{y}_1,\mbf{y}_2,\mbf{y}_3)\|_{H_2} \leq \left\|M_1\right\|_{2,X_2} \|G_3(\mbf{y}_1,\mbf{y}_2,\mbf{y}_3)\|_{H_2}$
where $M_1$ is defined in \eqref{def : Z2 saddle definitions}.
Now, we examine the specific cases to obtain
\begin{align}
    &\|G_3(\mbf{z},\mbf{h},\mbf{h})\|_{H_2} 
    = \sqrt{2}\left( \|z_{1,1} h_{1,1} h_{1,2}\|_{2}^2 + \|z_{1,1}(2h_{1,2} h_{2,1} + h_{1,1} h_{2,2})\|_{2}^2\right)^{\frac{1}{2}} \\
    &\leq 2\sqrt{10}\kappa \|\mathscr{l}^{-1}\|_{\mathcal{M}_2} \|\mbf{z}\|_{H_1} \|\mbf{h}\|_{H_1}^2 \leq 2\sqrt{10} \kappa \|\mathscr{l}^{-1}\|_{\mathcal{M}_2} r^2 \bydef \frac{\mathcal{Z}_{2,2}(r) r}{\|M_1\|_{2,X_2}}.
\end{align}
The other estimates proceed similarly, and have the same result. That is,  
\begin{align}
    \|G_3(\mbf{h},\mbf{z},\mbf{h})\|_{H_2} \leq \frac{\mathcal{Z}_{2,2}(r)}{\|M_1\|_{2,X_2}}, ~\|G_3(\mbf{h},\mbf{h},\mbf{z})\|_{H_2} \leq \frac{\mathcal{Z}_{2,2}(r)r}{\|M_1\|_{2,X_2}} 
\end{align}
where $\mathcal{Z}_{2,2}(r)$ is defined as in \eqref{def : Z2js saddle}.
Notice that due to our choice of estimate for $\|z_{i,j} h_{k,l} h_{m,n}\|_{2},$ we obtain the same result for each quantity. This significantly simplifies our computation compared to the approach used in \citep{gs_cadiot_blanco}. It also allow us to use the same bound for the saddle node bounds as we did for patterns. This is one of the reasons we chose this approach in this work. 
\par It now remains to study $\left\|\mathbb{B}D^2\mathbb{G}(\overline{\mathbf{u}},\overline{\mathbf{w}})(\mbf{h},\mbf{z}) \right\|_{H_2}$. Observe first the following
{\small\begin{align}
    \|\mathbb{B}D^2\mathbb{G}(\overline{\mathbf{u}},\overline{\mathbf{w}})(\mbf{h},\mbf{z})\|_{H_2}
    &\leq \left\|\mathbb{B}\begin{bmatrix}
        0 & 0 & 0 & 0 & 0 \\
        0 &\mathbb{q}_1 & \mathbb{q}_2 & 0 & 0 \\ 0 &
        -\mathbb{q}_1 & -\mathbb{q}_2 & 0 & 0 \\
        0 & \mathbb{q}_3 & \mathbb{q}_4 & \mathbb{q}_1 & \mathbb{q}_2 \\
        0 & -\mathbb{q}_3 & -\mathbb{q}_4 & -\mathbb{q}_1 & -\mathbb{q}_2
    \end{bmatrix}\begin{bmatrix}
        0 & 0 & 0 & 0 & 0 \\0 &\mathbb{h}_{1,1} & 0 & 0 & 0 \\
        0 &\mathbb{h}_{1,2} & \mathbb{h}_{1,1} & 0 & 0 \\
        0 &\mathbb{h}_{2,1} & 0 & \mathbb{h}_{1,1} & 0 \\
        0 &\mathbb{h}_{2,2} & \mathbb{h}_{2,1} & \mathbb{h}_{1,2} & \mathbb{h}_{1,1}
    \end{bmatrix} \begin{bmatrix}
        0 \\z_{1,1} \\
        z_{1,2} \\
        z_{2,1} \\
        z_{2,2}
    \end{bmatrix}\right\|_{H_2}\\
    &\leq \left\|M_1\begin{bmatrix}
       \mathbb{Q}_1 & \mathbb{Q}_2 & 0 & 0 \\\mathbb{Q}_4 & \mathbb{Q}_5 & \mathbb{Q}_1 & \mathbb{Q}_2 \end{bmatrix}\right\|_{2,X_2}\left\|\begin{bmatrix}
        \mathbb{h}_{1,1} & 0 & 0 & 0 \\
        \mathbb{h}_{1,2} & \mathbb{h}_{1,1} & 0 & 0 \\
        \mathbb{h}_{2,1} & 0 & \mathbb{h}_{1,1} & 0 \\
        \mathbb{h}_{2,2} & \mathbb{h}_{2,1} & \mathbb{h}_{1,2} & \mathbb{h}_{1,1}
    \end{bmatrix} \begin{bmatrix}
        z_{1,1} \\
        z_{1,2} \\
        z_{2,1} \\
        z_{2,2}
    \end{bmatrix} \right\|_{2} \\
    &\leq 5\kappa r \sqrt{\left\|M_1 M_2 M_1^*\right\|_{X_2}} 
\end{align}}
where the last step followed from the definition of $\mathbb{B}$ and properties of $\Gamma^\dagger$, similar algebraic steps to those performed in Lemma \ref{lem : Z2 patterns}, and $M_1,M_2,M_3$ are defined in \eqref{def : Z2 saddle definitions}. Hence, by Parseval's Identity, we need to analyze the following norms
{\small\begin{align}
    &\|\obpi^{\leq N} M_1 M_2 M_1^* \obpi^{\leq N}\|_{X_2} ,~\|\obpi^{\leq N} M_1 M_2 \bpi^{>N}\|_{2,X_2},~\|\bpi^{>N} M_2 M_1^* \obpi^{\leq N}\|_{X_2,2},~\text{and}~\|\bpi^{>N} M_2 \bpi^{>N}\|_{2}.\label{Z_2 renamed matrix saddle}
\end{align}}
The first term can be computed directly and is defined as $\mathcal{Z}_{2,3}$ in \eqref{def : Z2js saddle}. Each of the remaining terms can be computed using similar steps to those used in Lemma \ref{lem : Z2 patterns}. Indeed,  
{\small\begin{align}
    &\|\obpi^{\leq N} M_1 M_2 \bpi^{>N}\|_{2,X_2} = \|\bpi^{>N} M_2 M_1^* \obpi^{\leq N}\|_{X_2,2} = \sqrt{\|\obpi^{\leq N} M_1 M_2 \bpi^{>N} M_2^* M_1^* \obpi^{\leq N}\|_{X_2}} \bydef \mathcal{Z}_{2,4} \\
    &\|\bpi^{>N} M_2 \bpi^{>N}\|_{2} \leq \varphi(\|Q_1^2 + Q_2^2\|_{1}, \|Q_1 Q_3 + Q_2 Q_4\|_{1}, \|Q_3 Q_1 + Q_4 Q_2\|_{1}, \|Q_3^2 + Q_4^2 + Q_1^2 + Q_2^2\|_{1}) \bydef \mathcal{Z}_{2,5}
\end{align}}
where we used the properties of the adjoint and Young's inequality.
Therefore, we obtain
\begin{align}
    \left\|\mathbb{B}D^2\mathbb{G}(\overline{\mathbf{u}},\overline{\mathbf{w}})(\mbf{h},\mbf{z}) \right\|_{H_2} \leq 5\kappa r \sqrt{\varphi(\mathcal{Z}_{2,3},\mathcal{Z}_{2,4},\mathcal{Z}_{2,4},\mathcal{Z}_{2,5})}.
\end{align}
This concludes the computation of the $\mathcal{Z}_2$ bound.
\end{proof}

\section{Integral Computation using residues}

The goal is to compute $$\int_\mathbb{R} \frac{(a_1\xi^2 + a_2)^2}{\ell_\text{den}(\xi)^2} d\xi$$

To do this, we observe that we can write $\ell_\text{den}(\xi) = \rho|2\pi|^4 (\xi^2 - y_1^2)(\xi^2 - y_2^2)$ where $y_1^2, y_2^2$ are the roots of $x \mapsto 16\pi^4\rho x^2 + 4\pi^2(\rho\nu_2 - \nu_1) x - \nu_2(\nu_1 + \nu_3)$.

Now observe that the function $f(\xi) = \frac{(a_1\xi^2 + a_2)^2}{\rho|2\pi|^4(\xi^2 - y_1^2)^2(\xi^2 - y_2^2)^2}$ posses 4 poles of order 2 at $\pm y_1, \pm y_2$. Without loss of generality, assume $\mathrm{Im}(y_1), \mathrm{Im}(y_2) > 0$ (if not, take $-y_1$ or $-y_2$, Assumption \ref{assumption : fourier} also guarantees that $\mathrm{Im}(y_1), \mathrm{Im}(y_1) \neq 0$). Using Residue Theory, we have that 
$$\int_\mathbb{R} f(\xi)d\xi = 2\pi i \left(\mathrm{Res}(f, y_1) + \mathrm{Res}(f, y_2)\right)$$
It remains to compute $\mathrm{Res}(f, y_i)$, for $i = 1,2$. Direct computations lead to
\begin{align*}
    \mathrm{Res}(f, y_i) = \lim_{\xi \rightarrow y_i} \frac{d}{d\xi} \left((\xi - y_i)^2 f(\xi)\right)
    &= \lim_{\xi \rightarrow y_i} \frac{d}{d\xi} \left[\frac{(a_1\xi^2 + a_2)^2}{(2\pi)^8\rho^2 (\xi + y_i)^2(\xi^2 - y_j^2)^2}\right] \\ 
    &= \frac{1}{\rho^2|2\pi|^8}\lim_{\xi \rightarrow y_i} \frac{d}{d\xi} \left[\frac{(a_1\xi^2 + a_2)^2}{(\xi + y_i)^2(\xi^2 - y_j^2)^2}\right],
\end{align*}
where $j = 2$, if $i = 1$ and $j = 1$, if $i = 2$. It just remains to take compute the derivative. In fact, we obtain that
$\frac{d}{d\xi} (\xi + y_i)^2(\xi^2 - y_j^2)^2 = 2(\xi + y_i)(\xi^2 - y_j^2)\left(3\xi^2 + 2y_i \xi - y_2^2\right), \frac{d}{d\xi} (a_1\xi^2 + a_2)^2 = 4a_1\xi (a_1\xi^2 + a_2)$. Applying the quotient rule, we  get:
{\footnotesize
\begin{align*}
    \mathrm{Res}(f, y_i) &= \frac{1}{\rho^2|2\pi|^8}\lim_{\xi \rightarrow y_i} \frac{4a_1\xi(a_1\xi^2 + a_2)(\xi + y_1)^2(\xi^2 - y_j^2)^2 - (a_1\xi^2 + a_2)^2\cdot 2(\xi+y_i)(\xi^2 - y_j^2)(3\xi^2 + 2y_i\xi - y_j^2)}{(\xi + y_i)^4(\xi^2 - y_j^2)^2} \\
    &= \frac{1}{\rho^2|2\pi|^8}\lim_{\xi \rightarrow y_i} 2(a_1\xi^2 + a_2)\frac{2a_1\xi(\xi+y_i)(\xi^2 - y_2^2)^2 - (a_1\xi^2 + a_2)(3\xi^2 -2y_i\xi - y_2^2)}{(\xi+y_i)^3(\xi^2 - y_j^2)^3} \\
     &= \frac{2(a_1y_i^2 + a_2)(4a_1y_i^2(y_i^2 - y_j^2) - (a_1y_1^2 + a_2)(5y_1^2 - y_2^2))}{\rho^2|2\pi|^8\cdot 8y_i^3(y_i^2 - y_j^2)^2}.
\end{align*}
}
\normalsize
These formulas are then implemented and evaluated using interval arithmetic in \citep{julia_blanco_fassler}.
\renewcommand{\theequation}{C.\arabic{equation}}
\setcounter{equation}{0}

\renewcommand{\theequation}{C.\arabic{equation}}
\setcounter{equation}{0}
\section{Computing the Bounds for Controlling the Spectrum}\label{apen : spectrum computations}
In this appendix, we provide the details for the computations performed in Section \ref{sec : control of the spectrum}. Specifically, we derive explicit formulas which can directly be implemented using rigorous numerics (see \cite{julia_blanco_fassler}).
\subsection{Computing \texorpdfstring{$\tilde{a}$}{tildea} and \texorpdfstring{$\tilde{C}_0$}{tildeC0}}\label{apen : tildea tilde C0}
We first discuss the computation of $\tilde{a}$ and $\tilde{C}_0$ required for Lemma \ref{lem : ift_L_inverse for spectrum}. Recall that these constants are defined via a minimum (for $\tilde{a}$) and a maximum (for $\tilde{C}_0$) over the set $\mathcal{I}(\delta_0) = \{ ix, ~ |x| \leq \delta_0  \} \cup \left\{ \frac{\pm \sqrt{-4\tilde{\rho}\nu_2 \nu_3} - \tilde{\rho} \nu_2 - \nu_1}{\tilde{\rho} -1} \right\}.$ The computation is achieved directly thanks to interval arithmetic on Julia \cite{julia_interval}. Indeed, we evaluate quantities rigorously on small intervals, and obtain upper and lower bounds.   
For this matter, we consider a partition of unity  $\mathcal{I}_k = [p_k,p_{k+1}]$ for $k = 1,\dots,K$ and $p_k \in \mathbb{R}, |p_k| \leq \delta_0$  such that
\begin{align}
    \{ ix, ~ |x| \leq \delta_0  \} =  \bigcup_{k = 1}^K i\mathcal{I}_k.
\end{align}
Then, we define
\begin{align}
    \tilde{a}_{K+1} \bydef a\left(\frac{\sqrt{-4\nu_2\nu_3} - \tilde{\rho}\nu_2 - \nu_1}{\tilde{\rho} -1}\right), ~&
    \tilde{a}_{K+2} \bydef a\left(\frac{-\sqrt{-4\nu_2\nu_3} - \tilde{\rho}\nu_2 - \nu_1}{\tilde{\rho} -1}\right),\\
\tilde{a}_k \bydef a(i\mathcal{I}_{k}) &\text{ for all } k \in \{1, \dots, K\}
\end{align}
and
\begin{align}
    (\tilde{C}_{0})_{K+1} \bydef C\left(\frac{\sqrt{-4\nu_2\nu_3} - \tilde{\rho}\nu_2 - \nu_1}{\tilde{\rho} -1}\right), ~
    &(\tilde{C}_0)_{K+2} \bydef C\left(\frac{-\sqrt{-4\nu_2\nu_3} - \tilde{\rho}\nu_2 - \nu_1}{\tilde{\rho} -1}\right), \\
    (\tilde{C}_0)_k \bydef C(i\mathcal{I}_k) &\text{ for all } k \in \{1, \dots, K\}.
\end{align}
Then, we directly obtain
\begin{align}
    \tilde{a} = \min_{k = 1,\dots,K+2}\tilde{a}_k, ~~ \tilde{C}_0 \bydef \max_{k = 1,\dots,K+2} (\tilde{C}_0)_k.
\end{align}
Using interval arithmetic, we are able to enclose rigorously each $\tilde{a}_k$ and each $(\tilde{C}_0)_k$, allowing us to compute a rigorous (and sharp) lower bound for $\tilde{a}$ and upper bound for $\tilde{C}_0$. 
\subsection{Lemma \ref{lem : link of spectrum unbounded to U0}}\label{apen : link of spectrum} 
Next, we compute the bounds required to apply Lemma \ref{lem : link of spectrum unbounded to U0}. For each component of the estimates, our objective is to provide an explicit formula which can be implemented and controlled thanks to interval arithmetic. We first provide two preliminary lemmas, introducing notably the constants $\mathcal{K}_1$ and $\mathcal{K}_2$ which will be used along our estimations.
\begin{lemma}\label{lem : K1}
Let $t \in \mathbb{C}$ be such that $S + t I_d$ is invertible. Moreover,  let $\mu \in \sigma_{\delta_0}$. Then, 
\begin{align}
    \|\Pi^{\leq N} (S + tI_d)^{-1} P^{-1}(l_{\tilde{\rho}} - \mu I_d)\|_{2} \leq \mathcal{K}_1\|\Pi^{\leq N} (S + tI_d)^{-1} (P^N)^{-1}(l_{\tilde{\rho}} + \delta_0 I_d)\Pi^{\leq N}\|_{2},
\end{align}
where
\begin{align}
    \mathcal{K}_1 \bydef 1 + \frac{\|I_d - (P^N)^{-1}P^N\|_{2}\|(P^N)^{-1}\|_{2}}{1 - \|I_d - (P^N)^{-1}P^N\|_{2}}.
\end{align}
\begin{proof}
To begin, since $(S + t I_d)^{-1}$ is diagonal, it follows that $\Pi^{\leq N}(S+tI_d)^{-1} = \Pi^{\leq N} (S + tI_d)^{-1} \Pi^{\leq N}$. Moreover, by definition of $P^{-1}$, we have $\Pi^{\leq N} P^{-1}  = \Pi^{\leq N} (P^N)^{-1} \Pi^{\leq N}$. Finally, since $l_{\tilde{\rho}} - \mu I_d$ is also diagonal, we get
\begin{align}
    \|\Pi^{\leq N} (S + tI_d)^{-1} P^{-1}(l_{\tilde{\rho}} - \mu I_d)\|_{2} &\leq \mathcal{K}_1 \|\Pi^{\leq N} (S + tI_d)^{-1} (P^N)^{-1}(l_{\tilde{\rho}} - \mu I_d)\Pi^{\leq N}\|_{2} \\
    &\hspace{-2cm}\leq \|\Pi^{\leq N} (S + tI_d)^{-1} (P^N)^{-1}(l_{\tilde{\rho}} + \delta_0 I_d)\Pi^{\leq N}\|_{2} \|(l_{\tilde{\rho}} + \delta_0 I_d)^{-1} (l_{\tilde{\rho}} - \mu I_d)\|_{2}.
\end{align}
Now, using that
\begin{align} 
\|(l_{\tilde{\rho}} + \delta_0 I_d)^{-1} (l_{\tilde{\rho}} - \mu I_d)\|_{2} \leq 1,
\end{align} 
since $\mu \in \sigma_{\delta_0}$, we obtain
\begin{align}
    \|\Pi^{\leq N} (S + tI_d)^{-1} P^{-1}(l_{\tilde{\rho}} - \mu I_d)\|_{2} &\leq \mathcal{K}_1\|\Pi^{\leq N} (S + tI_d)^{-1} (P^N)^{-1}(l_{\tilde{\rho}} + \delta_0 I_d)\Pi^{\leq N}\|_{2} 
\end{align}
\end{proof}
\end{lemma}
\begin{lemma}\label{lem : K2}
We estimate
\begin{align}
    \sup_{\mu \in \mathcal{J}}\|\Pi^{> N} (l_{\tilde{\rho}} - \mu I_d)^{-1}\|_{2} &\leq \sup_{\mu \in \mathcal{J}}\varphi\left(\left(\frac{l_{22}-\mu}{l_{den,\mu}}\right)_N,\left(\frac{l_{12}}{l_{den,\mu}}\right)_N,\left(\frac{l_{21}}{l_{den,\mu}}\right)_N,\left(\frac{l_{11}-\mu}{l_{den,\mu}}\right)_N\right) \bydef \mathcal{K}_2,
\end{align}
where, given a function $q : \R \to \R$ and $N \in \mathbb{N}$, $(q)_{{N}}$ is defined as in Lemma \ref{lem : Z1 periodic patterns}.
\end{lemma}
\begin{proof}
Observe that
\begin{align}
    (l_{\tilde{\rho}} - \mu I_d)^{-1} = \begin{bmatrix}
        (l_{22}-\mu ) l_{den,\mu}^{-1} & -l_{12} l_{den,\mu}^{-1} \\
        -l_{21} l_{den,\mu}^{-1} & (l_{11} - \mu)l_{den,\mu}^{-1}
    \end{bmatrix}.
\end{align}
The proof directly follows from the proof of Lemma \ref{lem : Z1 periodic patterns}.
\end{proof}
We are now in a position to estimate the bounds $Z_{1,k}$. These bounds are fully estimated thanks to computations involving Fourier coefficients operators. In particular, we provide upper bounds depending on matrix norms. 
\begin{lemma}
Let $\mathcal{K}_1$ be defined as in Lemma \ref{lem : K1} and $\mathcal{K}_2$ as in Lemma \ref{lem : K2}.
Let $Z_{1,k}$ be bounds defined as in Lemma \ref{lem : link of spectrum unbounded to U0}. Then, letting
\begin{align}
    &Z_{1,1} \bydef \mathcal{K}_2  \sqrt{\|(P^N)^* Dg(\overline{\mathbf{U}})^*\Pi^{> N} Dg(\overline{\mathbf{U}}) P^N\|_{2}},\\
    &Z_{1,2} \bydef \sqrt{2}\mathcal{K}_2 \sqrt{\|\overline{V}_1\|_{1}^2 + \|\overline{V}_2\|_{1}^2},\\
    &Z_{1,3} \bydef \mathcal{K}_1 \|\bpi^{\leq N} (S + tI_d)^{-1} R\bpi^{\leq N}\|_{2},\\
    &Z_{1,4} \bydef \mathcal{K}_1\sqrt{\|\Pi^{\leq N} (S + t I_d)^{-1} (P^N)^{-1} Dg(\overline{U})\Pi^{> N}Dg(\overline{U})^* (P^N)^{-*} (S + tI_d)^{-*} \Pi^{\leq N}\|_{2}},
\end{align}
it follows that $Z_{1,k}$ for each $k = 1,2,3,4$ satisfy the bound required in Lemma \ref{lem : link of spectrum unbounded to U0}.
\end{lemma}
\begin{proof}
Note that $Z_{1,3}$ can be computed as is with the factor $\mathcal{K}_1$ to account for the error we get when using $(P^N)^{-1}$ instead of $P^{-1}$. For $Z_{1,4}$, recall that $\Pi^{\leq N} R\Pi^{> N} = (P^N)^{-1} Dg(\overline{\mathbf{U}})\Pi^{> N}$. We show this result in detail.
\begin{align}
    \Pi^{\leq N} R\Pi^{> N} &= \Pi^{\leq N} (\mathcal{D} - S)\Pi^{> N} \\
    &= \Pi^{\leq N} P^{-1} Df(\overline{\mathbf{U}}) P \Pi^{> N} - \Pi^{\leq N} S \Pi^{> N} \\
    &=\Pi^{\leq N} P^{-1} L P \Pi^{> N}+ \Pi^{\leq N} P^{-1} Dg(\overline{\mathbf{U}}) P \Pi^{> N} - 0.
\end{align}
where we used that $S$ is diagonal. Since $L$ is also diagonal, and $P$ is defined through $P^N$ and a tail which is the identity, we can also say 
\begin{align}
    \Pi^{\leq N} P^{-1} L P \Pi^{> N}+ \Pi^{\leq N} P^{-1} Dg(\overline{\mathbf{U}}) P \Pi^{> N} &= 0 + \Pi^{\leq N} P^{-1} Dg(\overline{\mathbf{U}})P\Pi^{> N}.
\end{align}
Finally, since $P \Pi^{> N} = \Pi^{> N}$, we have
\begin{align}
    \Pi^{\leq N} R\Pi^{> N} &=\Pi^{\leq N} P^{-1} Dg(\overline{\mathbf{U}}) P\Pi^{> N} = \Pi^{\leq N} P^{-1} Dg(\overline{\mathbf{U}}) \Pi^{> N}.
\end{align}
Hence, we obtain
\begin{align}
    \|\Pi^{\leq N} (S + t I_d)^{-1} R\Pi^{> N}\|_{2}^2 &= \|\Pi^{\leq N} (S + t I_d)^{-1} (P^N)^{-1} Dg(\overline{\mathbf{U}})\Pi^{> N}\|_{2}^2 \\
    &\hspace{-2cm}=\mathcal{K}_1^2\|\Pi^{\leq N} (S + t I_d)^{-1} (P^N)^{-1} Dg(\overline{\mathbf{U}})\Pi^{> N}Dg(\overline{\mathbf{U}})^* (P^N)^{-*} (S + tI_d)^{-*} \Pi^{\leq N}\|_{2} \\
    &\hspace{-2cm}\bydef Z_{1,4}^2,
\end{align}
which can now be evaluated on the computer. We now move to $Z_{1,2}$. For a matrix $M$, let $\mathrm{diag}(M)$ denote the diagonal matrix with the diagonal entries of $M$. Again using properties of $R$, observe that
\begin{align}
    \sup_{\mu \in \mathcal{J}} \|\Pi^{> N} (l_{\tilde{\rho}} - \mu I_d)^{-1} R \Pi^{> N}\|_{2} \leq \sup_{\mu \in \mathcal{J}}\|\Pi^{> N} (l_{\tilde{\rho}} - \mu I_d)^{-1}\|_{2}  \|\Pi^{> N} R \Pi^{> N}\|_{2}& \\
    = \sup_{\mu \in \mathcal{J}}\|\Pi^{> N} (l_{\tilde{\rho}} - \mu I_d)^{-1}\|_{2}  \|\Pi^{> N} (P^{-1} L P + P^{-1} Dg(\overline{\mathbf{U}})P - S) \Pi^{> N}\|_{2}&\\
    = \sup_{\mu \in \mathcal{J}}\|\Pi^{> N} (l_{\tilde{\rho}} - \mu I_d)^{-1}\|_{2}  \|\Pi^{> N} (P^{-1} L P + P^{-1} Dg(\overline{\mathbf{U}})P - P^{-1} LP - P^{-1}\mathrm{diag}(Dg(\overline{\mathbf{U}}))P) \Pi^{> N}\|_{2}&\\
    = \sup_{\mu \in \mathcal{J}}\|\Pi^{> N} (l_{\tilde{\rho}} - \mu I_d)^{-1}\|_{2}  \|\Pi^{> N} (P^{-1} Dg(\overline{\mathbf{U}})P- P^{-1}\mathrm{diag}(Dg(\overline{\mathbf{U}}))P) \Pi^{> N}\|_{2}&
    \\
    = \sup_{\mu \in \mathcal{J}}\|\Pi^{> N} (l_{\tilde{\rho}} - \mu I_d)^{-1}\|_{2}   \|\Pi^{> N} (Dg(\overline{\mathbf{U}})-\mathrm{diag}(Dg(\overline{\mathbf{U}})))\Pi^{> N}\|_{2}&
    \end{align}
where we used that $\Pi^{> N} P^{-1} = \Pi^{> N}$ and $P \Pi^{>N} = \Pi^{> N}$. 
Continuing, 
    \begin{align}
    \sup_{\mu \in \mathcal{J}} \|\Pi^{> N} (l_{\tilde{\rho}} - \mu I_d)^{-1} R \Pi^{> N}\|_{2}&\leq \sup_{\mu \in \mathcal{J}}\|\Pi^{> N} (l_{\tilde{\rho}} - \mu I_d)^{-1}\|_{2} \| \Pi^{> N}Dg(\overline{\mathbf{U}})\Pi^{> N}\|_{2} \\
    &\leq \sup_{\mu \in \mathcal{J}}\|\Pi^{> N} (l_{\tilde{\rho}} - \mu I_d)^{-1}\|_{2} \|Dg(\overline{\mathbf{U}})\|_{2} \\
    &= \sqrt{2}\sup_{\mu \in \mathcal{J}}\|\Pi^{> N} (l_{\tilde{\rho}} - \mu I_d)^{-1}\|_{2} \sqrt{\|\overline{V}_1\|_{1}^2 + \|\overline{V}_2\|_{1}^2} \\
    &\bydef Z_{1,2}.
\end{align}
Finally, for $Z_{1,1}$, we estimate
{\small\begin{align}
    \sup_{\mu \in \mathcal{J}} \|\Pi^{> N} (l_{\tilde{\rho}} - \mu I_d)^{-1} R \Pi^{\leq N}\|_{2}^2 &\leq \sup_{\mu \in \mathcal{J}}\|\Pi^{> N} (l_{\tilde{\rho}} - \mu I_d)^{-1}\|_{2}^2  \|\Pi^{> N} R \Pi^{\leq N}\|_{2}^2 \\
    &= \sup_{\mu \in \mathcal{J}}\|\Pi^{> N} (l_{\tilde{\rho}} - \mu I_d)^{-1}\|_{2}^2  \|\Pi^{> N} (P^{-1} L P + P^{-1} Dg(\overline{\mathbf{U}}) P - S) \Pi^{\leq N}\|_{2}^2 \\
    &= \sup_{\mu \in \mathcal{J}}\|\Pi^{> N} (l_{\tilde{\rho}} - \mu I_d)^{-1}\|_{2}^2  \|\Pi^{> N} Dg(\overline{\mathbf{U}}) P^N \|_{2}^2
    \end{align}}
where we used the fact that $S$ is diagonal, that $L$ is diagonal along with properties of $P$. Additionally, we used $\Pi^{> N} P^{-1} = \Pi^{> N}$, and $P \Pi^{\leq N} = P^N$. Continuing, 
    \begin{align}
    \sup_{\mu \in \mathcal{J}} \|\Pi^{> N} (l_{\tilde{\rho}} - \mu I_d)^{-1} R \Pi^{\leq N}\|_{2}^2&= \sup_{\mu \in \mathcal{J}}\|\Pi^{> N} (l_{\tilde{\rho}} - \mu I_d)^{-1}\|_{2}^2  \|(P^N)^* Dg(\overline{\mathbf{U}})^*\Pi^{> N} Dg(\overline{\mathbf{U}}) P^N\|_{2} \\
    &\bydef Z_{1,1}^2.
\end{align}
This concludes the proof.
\end{proof}
Next, we provide formulas for the bounds $\mathcal{Z}_{u,k}$ bounds for $k = 1,2,3, 4$. These formulas have been derived in previous work (cf. \cite{unbounded_domain_cadiot}) and we recall them for completeness. 
\begin{lemma}\label{lem : Ju old exact}
Let $\tilde{a}$ be defined as in \eqref{def : amu}. Moreover, let $\tilde{E} \in \ell^2$ and\\ $\tilde{C}(d),\tilde{C}_1,\tilde{C}_2,\tilde{C}_3,\tilde{C}_4 > 0$ be defined as
\begin{align}
    &\tilde{E} \bydef \gamma(\mathbb{1}_{\om} \cosh(2\tilde{a}x)),~ \tilde{C}(d) \bydef 4d + \frac{4e^{-\tilde{a}d}}{\tilde{a}(1-e^{-\frac{3\tilde{a}d}{2}})} + \frac{2}{\tilde{a}(1-e^{-2\tilde{a}d})},\\
    &\tilde{C}_1 \bydef \tilde{C}_0(-1,-\nu_2),~ \tilde{C}_2 \bydef \tilde{C}_0(0,\nu_2),~C_3 \bydef \tilde{C}_0(0,\nu_3),~\text{and}~ \tilde{C}_4 \bydef \tilde{C}_0(-\tilde{\rho},\nu_1)\label{def : E Cd Cj mu}
\end{align}
where $\tilde{C}_0(d_1,d_2)$ is defined as in \eqref{def : Cmu}. Then, let $(\mathcal{Z}_{u,k,j})_{k \in \{1,2\},j \in \{1,2,3,4\}} > 0$ be defined as
\begin{align}
&\mathcal{Z}_{u,1,1}^2 \bydef |\om| \frac{\tilde{C}_1^2e^{-2\tilde{a}d}}{\tilde{a}} (V_1^N, V_1^N * \tilde{E})_2,~\mathcal{Z}_{u,2,1}^2 \bydef \mathcal{Z}_{u,1,1}^2 + e^{-4\tilde{a}d} \tilde{C}(d) \tilde{C}_1^2 |\om|(V_1^N,V_1^N * \tilde{E})_2 \\
    &\mathcal{Z}_{u,1,2}^2 \bydef |\om| \frac{\tilde{C}_2^2 e^{-2\tilde{a}d}}{\tilde{a}} (V_2^N, V_2^N * \tilde{E})_2,~\mathcal{J}_{u,2,2}^2 \bydef \mathcal{Z}_{u,1,2}^2 + e^{-4\tilde{a}d} \tilde{C}(d) \tilde{C}_2^2 |\om|(V_2^N,V_2^N * \tilde{E})_2 \\
    &\mathcal{Z}_{u,1,3}^2 \bydef |\om| \frac{\tilde{C}_3^2 e^{-2\tilde{a}d}}{\tilde{a}} (V_1^N, V_1^N * \tilde{E})_2,~\mathcal{Z}_{u,2,3}^2 \bydef \mathcal{Z}_{u,1,3}^2 + e^{-4\tilde{a}d} \tilde{C}(d) \tilde{C}_3^2 |\om|(V_1^N,V_1^N * \tilde{E})_2 \\
    &\mathcal{Z}_{u,1,4}^2 \bydef |\om| \frac{\tilde{C}_4^2e^{-2\tilde{a}d}}{\tilde{a}} (V_2^N,V_2^N*\tilde{E})_2,~\mathcal{Z}_{u,2,4}^2 \bydef \mathcal{Z}_{u,1,4}^2 + e^{-4\tilde{a}d} \tilde{C}(d) \tilde{C}_4^2 |\om|(V_2^N,V_2^N * \tilde{E})_2.
\end{align}
Let 
\begin{align}
    \mathcal{Z}_{u,k} \bydef \sqrt{2}\sqrt{(\mathcal{Z}_{u,k,1} + \mathcal{Z}_{u,k,2})^2 +(\mathcal{Z}_{u,k,3} + \mathcal{Z}_{u,k,4})^2}
\end{align}
for $k = 1,2$ and $\mathcal{Z}_{u,3} \bydef \mathcal{K}_1 \mathcal{Z}_{u,2}$. Then, $\mathcal{Z}_{u,1},\mathcal{Z}_{u,2}, $ and $\mathcal{Z}_{u,3}$ satisfy their respective bounds from Lemma \ref{lem : link of spectrum unbounded to U0}.
\end{lemma}
\begin{proof}
The proof follows from Lemma 6.5 of \cite{unbounded_domain_cadiot}. In particular, each $\mathcal{Z}_{u,k,j}$ can be computed using the results of the aforementioned lemma.
\end{proof}
Next, we compute the $\mathcal{C}_k$ bounds for $k = 1,2$. 
\begin{lemma}\label{lem : scriptC12 comp}
Let $\mathcal{C}_1$ be defined as 
\begin{align}
    &\mathcal{C}_1 \bydef \sup_{\mu \in \mathcal{J}}\|(\mathscr{l}_{\tilde{\rho}} - \mu I_d)^{-1}\|_{\mathcal{M}_1} (2\sqrt{5} \kappa \sqrt{\mathcal{Z}_{2,3}} + 3\kappa \|\mathscr{l}_{\overline{\rho}}^{-1}\|_{\mathcal{M}_2} r_0), \\
    &\mathcal{C}_2 \bydef \mathcal{K}_1 \mathcal{C}_1
\end{align}
Then, it follows that $\mathcal{C}_1$ and  $\mathcal{C}_2$ satisfy their respective definition in Lemma \ref{lem : link of spectrum unbounded to U0}. 
\end{lemma}
\begin{proof}
First, recall that $\tilde{\mathbf{u}} \in B_{r_0}(\overline{\mathbf{u}})$. Hence, there exists an $\mathbf{h} \in B_{r_0}(0)$ such that $\tilde{\mathbf{u}} = \overline{\mathbf{u}} + \mathbf{h}$. We now proceed from the definition
{\small\begin{align}
    \|(\mathbb{l}_{\tilde{\rho}} - \mu I_d)^{-1} (D\mathbb{g}(\overline{\mathbf{u}}) - D\mathbb{g}(\tilde{\mathbf{u}}))\|_{2} &\leq \|(\mathbb{l}_{\tilde{\rho}} - \mu I_d)^{-1}\|_{2} \|D\mathbb{g}(\overline{\mathbf{u}}) - D\mathbb{g}(\overline{\mathbf{u}}+\mathbf{h})\|_{2} \\
    &= \|(\mathscr{l}_{\tilde{\rho}} - \mu I_d)^{-1}\|_{\mathcal{M}_1} (\|D^2\mathbb{g}(\overline{\mathbf{u}})(\mathbf{h},\mathbf{z})\|_{2} + \|2g_3(\mathbf{z},\mathbf{h},\mathbf{h}) + g_3(\mathbf{h},\mathbf{h},\mathbf{z})\|_{2}) \\
    &\leq \|(\mathscr{l}_{\tilde{\rho}} - \mu I_d)^{-1}\|_{\mathcal{M}_1} (\|D^2\mathbb{g}(\overline{\mathbf{u}})(\mathbf{h},\mathbf{z})\|_{2} + 3\kappa \|\mathscr{l}_{\tilde{\rho}}^{-1}\|_{\mathcal{M}_2}r_0^2)\label{C1 split}
\end{align}}
for some $\mathbf{z} \in L_e^2, \|\mathbf{z}\|_{2} = 1$ and $g_3 : (\mathcal{H}_e)^3 \to \mathcal{H}_e$ is defined as in \eqref{def : g3}. Notice that the $\mathcal{M}_1$ norm has $\tilde{\rho}$ whereas the $\mathcal{M}_2$ norm has $\overline{\rho}$. This is due to the factor $\|\mathscr{l}_{\overline{\rho}}^{-1}\|_{\mathcal{M}_2}$ coming from $\kappa$, which was computed with $\overline{\rho}$. Now, observe that
\begin{align}
    \|D^2\mathbb{g}(\overline{\mathbf{u}})(\mathbf{h},\mathbf{z})\|_{2} &= \left\| \begin{bmatrix}
        \mathbb{q}_1 & \mathbb{q}_2 \\
        -\mathbb{q}_1 & -\mathbb{q}_2
    \end{bmatrix}\begin{bmatrix}
        \mathbb{h}_1 & \mathbb{h}_2 \\
        \mathbb{h}_2 & \mathbb{h}_1
    \end{bmatrix}\begin{bmatrix}
    z_1 \\ z_2 \end{bmatrix}\right\|_{2} \\
    &\leq \left\|\begin{bmatrix}
        I_d \\ -I_d
    \end{bmatrix}\begin{bmatrix}
        \mathbb{q}_1 & \mathbb{q}_2
    \end{bmatrix}\right\|_{2} \left\|\begin{bmatrix}
        \mathbb{h}_1 & \mathbb{h}_2 \\
        \mathbb{h}_2 & \mathbb{h}_1
    \end{bmatrix}\begin{bmatrix}
    z_1 \\ z_2 \end{bmatrix}\right\|_{2} \\
    &= \sqrt{5} \kappa \sqrt{\left\| \begin{bmatrix}
        I_d \\ - I_d
    \end{bmatrix}(\mathbb{q}_1^2 + \mathbb{q}_2^2)\begin{bmatrix}
        I_d & - I_d
    \end{bmatrix}\right\|_{2}} r_0 \\
    &\leq 2\sqrt{5} \kappa \sqrt{\|\mathbb{q}_1^2 + \mathbb{q}_2^2\|_{2}}r_0 \\
    &= 2\sqrt{5} \kappa \sqrt{\|Q_1^2 + Q_2^2\|_{1}}r_0 \\
    &\bydef 2\sqrt{5}\kappa \sqrt{\mathcal{Z}_{2,3}} r_0,
\end{align}
where $\mathcal{Z}_{2,3}$ is defined as in Lemma \ref{lem : Z2 patterns}. Now, we get
{\small\begin{align}
    \sup_{\mu \in \mathcal{J}} \frac{1}{r_0}\|(\mathbb{l}_{\tilde{\rho}} - \mu I_d)^{-1} (D\mathbb{g}(\overline{\mathbf{u}}) - D\mathbb{g}(\tilde{\mathbf{u}}))\|_{2} &\leq \sup_{\mu \in \mathcal{J}} \frac{1}{r_0}\|(\mathscr{l}_{\tilde{\rho}} - \mu I_d)^{-1}\|_{\mathcal{M}_1} (2\sqrt{5} \kappa \sqrt{\mathcal{Z}_{2,3}} + 3\kappa \|\mathscr{l}_{\overline{\rho}}^{-1}\|_{\mathcal{M}_2} r_0)r_0 \\
    &= \sup_{\mu \in \mathcal{J}}\|(\mathscr{l}_{\tilde{\rho}} - \mu I_d)^{-1}\|_{\mathcal{M}_1} (2\sqrt{5} \kappa \sqrt{\mathcal{Z}_{2,3}} + 3\kappa \|\mathscr{l}_{\overline{\rho}}^{-1}\|_{\mathcal{M}_2} r_0).
\end{align}}
The computation of $\mathcal{C}_2$  follows similarly with an extra factor of $\mathcal{K}_1$. 
\end{proof}
Finally, in order to apply Lemma \ref{lem : scriptC12 comp}, we must compute $\sup_{\mu \in \mathcal{J}}\|(\mathscr{l}_{\tilde{\rho}} - \mu I_d)^{-1}\|_{\mathcal{M}_1}$. This computation is technical and it will be  the goal of the next section.
\subsubsection{Estimating the \texorpdfstring{$\mathcal{M}_1$}{M1} norm for control of the spectrum}\label{apen : M1 int estimate}
In this appendix, we estimate 
\begin{align}
    \sup_{\mu \in \mathcal{J}} \|(\mathscr{l}_{\tilde{\rho}}(\xi) - \mu I_d)^{-1}\|_{\mathcal{M}_1},
\end{align}
which is required to apply Lemma \ref{lem : scriptC12 comp}.
To begin, observe that
\begin{align}
    (\mathscr{l}_{\tilde{\rho}}(\xi) - \mu I_d)^{-1} &\bydef \begin{bmatrix}
        \frac{\mathscr{l}_{22}(\xi) - \mu}{\mathscr{l}_{den,\mu}(\xi)} & \frac{-\mathscr{l}_{12}(\xi)}{\mathscr{l}_{den,\mu}(\xi)} \\
        \frac{-\mathscr{l}_{21}(\xi)}{\mathscr{l}_{den,\mu}(\xi)} & \frac{\mathscr{l}_{11}(\xi) - \mu}{\mathscr{l}_{den,\mu}(\xi)}.
    \end{bmatrix}
\end{align}
Now, observe that we can write
\begin{align}
    \frac{1}{\mathscr{l}_{den,\mu}(\xi)} = \frac{1}{\mathrm{Re}(\mathscr{l}_{den,\mu}(\xi)) + i \mathrm{Im}(\mathscr{l}_{den,\mu}(\xi))} = \frac{\mathrm{Re}(\mathscr{l}_{den,\mu}(\xi)) - i\mathrm{Im}(\mathscr{l}_{den,\mu}(\xi))}{|\mathscr{l}_{den,\mu}(\xi)|^2}.
\end{align}
Now, let 
\begin{align}
    \delta_{ij} \bydef \begin{cases}
        1 & i = j \\
        0 & \mathrm{else}
    \end{cases}.
\end{align}
Then, we can also write
\begin{align}
&\frac{\mathscr{l}_{ij}(\xi) - \mu \delta_{ij}}{\mathscr{l}_{den,\mu}(\xi)} \\
&=\frac{(\mathscr{l}_{ij}(\xi) - \mathrm{Re}(\mu) \delta_{ij} - i\mathrm{Im}(\mu) \delta_{ij})(\mathrm{Re}(\mathscr{l}_{den,\mu}(\xi)) - i\mathrm{Im}(\mathscr{l}_{den,\mu}(\xi)))}{|\mathscr{l}_{den,\mu}(\xi)|^2} \\
&= \frac{(\mathscr{l}_{ij}(\xi) - \mathrm{Re}(\mu) \delta_{ij})\mathrm{Re}(\mathscr{l}_{den,\mu}(\xi)) - \mathrm{Im}(\mu) \delta_{ij} \mathrm{Im}(\mathscr{l}_{den,\mu}(\xi))}{|\mathscr{l}_{den,\mu}(\xi)|^2} \\
&\hspace{+4cm}- i \frac{(\mathscr{l}_{ij}(\xi) - \mathrm{Re}(\mu) \delta_{ij})\mathrm{Im}(\mathscr{l}_{den,\mu}(\xi)) + \mathrm{Im}(\mu) \delta_{ij} \mathrm{Re}(\mathscr{l}_{den,\mu}(\xi))}{|\mathscr{l}_{den,\mu}(\xi)|^2}\label{before squaring}
\end{align}
Then, observe that
{\tiny\begin{align}
    &[(\mathscr{l}_{ij}(\xi) - \mathrm{Re}(\mu) \delta_{ij})\mathrm{Re}(\mathscr{l}_{den,\mu}(\xi)) - \mathrm{Im}(\mu) \delta_{ij} \mathrm{Im}(\mathscr{l}_{den,\mu}(\xi))]^2 \\
    &=(\mathscr{l}_{ij}(\xi) - \mathrm{Re}(\mu) \delta_{ij})^2\mathrm{Re}(\mathscr{l}_{den,\mu}(\xi))^2 - 2\delta_{ij} (\mathscr{l}_{ij}(\xi) - \mathrm{Re}(\mu) \delta_{ij})\mathrm{Im}(\mu)  \mathrm{Re}(\mathscr{l}_{den,\mu}(\xi))\mathrm{Im}(\mathscr{l}_{den,\mu}(\xi)) + \delta_{ij} \mathrm{Im}(\mu)^2 \mathrm{Im}(\mathscr{l}_{den,\mu}(\xi))^2
\end{align}}
and 
{\tiny\begin{align}
    &[(\mathscr{l}_{ij}(\xi) - \mathrm{Re}(\mu) \delta_{ij})\mathrm{Im}(\mathscr{l}_{den,\mu}(\xi)) + \mathrm{Im}(\mu) \delta_{ij} \mathrm{Re}(\mathscr{l}_{den,\mu}(\xi))]^2 \\
    &=(\mathscr{l}_{ij}(\xi) - \mathrm{Re}(\mu) \delta_{ij})^2\mathrm{Im}(\mathscr{l}_{den,\mu}(\xi))^2 + 2\delta_{ij} (\mathscr{l}_{ij}(\xi) - \mathrm{Re}(\mu) \delta_{ij})\mathrm{Im}(\mu)  \mathrm{Re}(\mathscr{l}_{den,\mu}(\xi))\mathrm{Im}(\mathscr{l}_{den,\mu}(\xi)) + \delta_{ij} \mathrm{Im}(\mu)^2 \mathrm{Re}(\mathscr{l}_{den,\mu}(\xi))^2
\end{align}}
Hence, we return to \eqref{before squaring} and get
\begin{align}
\left|\frac{\mathscr{l}_{ij}(\xi) - \mu \delta_{ij}}{\mathscr{l}_{den,\mu}(\xi)}\right|^2 
&=\frac{(\mathscr{l}_{ij}(\xi) - \mathrm{Re}(\mu) \delta_{ij})^2 |\mathscr{l}_{den,\mu}(\xi)|^2 + \delta_{ij} \mathrm{Im}(\mu)^2 |\mathscr{l}_{den,\mu}(\xi)|^2}{|\mathscr{l}_{den,\mu}(\xi)|^4} \\
&=\frac{[(\mathscr{l}_{ij}(\xi) - \mathrm{Re}(\mu) \delta_{ij})^2 + \mathrm{Im}(\mu)^2 \delta_{ij}] |\mathscr{l}_{den,\mu}(\xi)|^2}{|\mathscr{l}_{den,\mu}(\xi)|^4}\\
&=\frac{(\mathscr{l}_{ij}(\xi) - \mathrm{Re}(\mu) \delta_{ij})^2 + \mathrm{Im}(\mu)^2 \delta_{ij}}{|\mathscr{l}_{den,\mu}(\xi)|^2}.
\end{align}
We then let
\begin{align}
    \mathscr{f}_{ij}(\xi,\mu) \bydef \frac{(\mathscr{l}_{ij}(\xi) - \mathrm{Re}(\mu) \delta_{ij})^2 + \mathrm{Im}(\mu)^2 \delta_{ij}}{|\mathscr{l}_{den,\mu}(\xi)|^2}.\label{def : fij}
\end{align}
Hence, we must estimate
\begin{align}
    \sup_{\mu \in \mathcal{J}} \|(\mathscr{l}_{\tilde{\rho}} - \mu I_d)^{-1}\|_{\mathcal{M}_1} &= \sup_{\mu \in \mathcal{J}}\sqrt{\sum_{i,j \in \{1,2\}} \left(\sup_{\xi \in \mathbb{R}} \sqrt{\mathscr{f}_{ij}(\xi,\mu)}\right)^2} \\
    &= \sup_{\mu \in \mathcal{J}}\sqrt{\sum_{i,j \in \{1,2\}} \sup_{\xi \in \mathbb{R}} |\mathscr{f}_{ij}(\xi,\mu)|} \\
    &=\sqrt{\sum_{i,j \in \{1,2\}} \sup_{\mu \in \mathcal{J}}\sup_{\xi \in \mathbb{R}} |\mathscr{f}_{ij}(\xi,\mu)|}
\end{align}
Our goal is now to estimate the supremum of the functions $\mathscr{f}_{ij}$ over $\xi$ and $\mu$. For that matter, we first provide estimates when $\xi \geq M$. This will later on simplify to computing a supremum over a bounded domain, which can be achieved thanks to interval arithmetic (using a partition of unity).
\begin{lemma}
Let $C = C(\epsilon) > 0$ be defined as
\begin{align}
    C \bydef \frac{1}{(2\pi)^8 \tilde{\rho}^2} + \epsilon
\end{align}
for any $\epsilon > 0$.
Choose $M = M(\epsilon)$ as
\begin{align}
    M \geq \max\left\{ 1, \sqrt{\frac{\sum_{k \in \{0,2,4,6\}} |\mathscr{b}_k|}{(2\pi)^8 \tilde{\rho}^2 \left(1 - \frac{1}{1 + (2\pi)^8 \tilde{\rho} \epsilon}\right)}}\right\}.
\end{align} 
where $\mathscr{b}_k$ are upper bounds for the coefficients of $|\mathscr{l}_{den,\mu}|^2$. Then, it follows that
\begin{align}
    \frac{1}{|\mathscr{l}_{den,\mu}(\xi)|^2} \leq \frac{C}{\xi^8} ~ \text{for all} ~\xi \in [M,\infty).
\end{align} 
\end{lemma}
\begin{proof}
Since $\mathscr{l}_{den,\mu}$ only has even coefficients and is 4th order, we write
\begin{align}
    |\mathscr{l}_{den,\mu}(\xi)|^2 = (2\pi)^8\tilde{\rho}^2 \xi^8 + \mathscr{a}_{6}(\mu) \xi^6 + \mathscr{a}_4(\mu) \xi^4 + \mathscr{a}_2(\mu) \xi^2 + \mathscr{a}_0(\mu)
\end{align}
for some $\mathscr{a}_k(\mu) \in \mathbb{R}$ and $k = 0,2,4,6$. Additionally, note that $(2\pi)^8 \tilde{\rho}^2 > 0$. We now choose $C$ such that
\begin{align}
    C \bydef \frac{1}{(2\pi)^8 \tilde{\rho}^2} + \epsilon
\end{align}
for some $\epsilon > 0$. To satisfy the inequality we want, we must have
{\small\begin{align}
    \left((2\pi)^8 \tilde{\rho}^2 - \frac{1}{C}\right)\xi^8 = (2\pi)^8 \tilde{\rho}^2 \left(1 - \frac{1}{1 + (2\pi)^8 \tilde{\rho} \epsilon}\right) \xi^8 \geq |\mathscr{a}_6(\mu)| \xi^6 + |\mathscr{a}_4(\mu)| \xi^4 + |\mathscr{a}_2(\mu)| \xi^2 + |\mathscr{a}_0(\mu)|.
\end{align}}
Now, if we assume $M \geq 1$ (hence $\xi \geq 1$ on the interval we want to bound), we have
\begin{align}
    |\mathscr{a}_6(\mu)| \xi^6 + |\mathscr{a}_4(\mu)| \xi^4 + |\mathscr{a}_2(\mu)| \xi^2 + |\mathscr{a}_0(\mu)| &\leq (|\mathscr{a}_6(\mu)|  + |\mathscr{a}_4(\mu)|  + |\mathscr{a}_2(\mu)|  + |\mathscr{a}_0(\mu)|)\xi^6 \\
    &\leq (|\mathscr{b}_6|  + |\mathscr{b}_4|  + |\mathscr{b}_2|  + |\mathscr{b}_0|)\xi^6
\end{align}
where $|\mathscr{a}_k(\mu)| \leq |\mathscr{b}_k|$ for each $k$ and all $\mu \in \mathcal{J}$. Hence, it suffices to solve
\begin{align}
    &(2\pi)^8 \tilde{\rho}^2 \left(1 - \frac{1}{1 + (2\pi)^8 \tilde{\rho} \epsilon}\right) \xi^8 \geq \xi^6 \sum_{k \in \{0,2,4,6\}} |\mathscr{b}_k| \\
    &(2\pi)^8 \tilde{\rho}^2 \left(1 - \frac{1}{1 + (2\pi)^8 \tilde{\rho} \epsilon}\right) \xi^2 \geq  \sum_{k \in \{0,2,4,6\}} |\mathscr{b}_k| \\
    & \xi^2 \geq \frac{\sum_{k \in \{0,2,4,6\}} |\mathscr{b}_k|}{(2\pi)^8 \tilde{\rho}^2 \left(1 - \frac{1}{1 + (2\pi)^8 \tilde{\rho} \epsilon}\right)} \\
    &\xi \geq \sqrt{\frac{\sum_{k \in \{0,2,4,6\}} |\mathscr{b}_k|}{(2\pi)^8 \tilde{\rho}^2 \left(1 - \frac{1}{1 + (2\pi)^8 \tilde{\rho} \epsilon}\right)}}.
\end{align}
Hence, it suffices to choose 
\begin{align}
    M \geq \max\left\{ 1, \sqrt{\frac{\sum_{k \in \{0,2,4,6\}} |\mathscr{b}_k|}{(2\pi)^8 \tilde{\rho}^2 \left(1 - \frac{1}{1 + (2\pi)^8 \tilde{\rho} \epsilon}\right)}}\right\}
\end{align}
and the inequality will hold.
\end{proof}
We now provide the lemma which computes the desired supremum.
\begin{lemma}
Let $\mathscr{f}_{ij}$ be defined as in \eqref{def : fij}. Let $\mathscr{l}_{ij}(\xi) \bydef (2\pi)^2 d_{1,ij} \xi^2 + d_{2,ij}$. Let
\begin{align}
    \mathscr{B}_{ij}^{\infty} \bydef C \left((2\pi)^4 \frac{d_{1,ij}^2}{M^4} + (2\pi)^2 \frac{|d_{1,ij}d_{2,ij}| + 2\delta_0 |d_{1,ij}|\delta_{ij}}{M^6} + \frac{d_{2,ij}^2 + 2\delta_0 d_{2,ij} \delta_{ij} + \delta_0^2 \delta_{ij}}{M^8}\right)
\end{align}Then, it follows that
\begin{align}
    \sup_{\mu \in \mathcal{J}}\sup_{\xi \in \mathbb{R}} |\mathscr{f}_{ij}(\xi,\mu)| &\leq \max\left\{ \left(\sup_{\mu \in \mathcal{J}}  \sup_{\xi \in [-M,M]} |\mathscr{f}_{ij}(\xi,\mu)|\right), \mathscr{B}_{ij}^{\infty}\right\}
\end{align}
\end{lemma}
\begin{proof}
The supremum in $\xi$ can be broken into parts.
\begin{align}
    \sup_{\xi \in \mathbb{R}} |\mathscr{f}_{ij}(\xi,\mu)| &\leq \max\left\{\sup_{\xi \in [-M,M]} |\mathscr{f}_{ij}(\xi,\mu)|, \sup_{\xi \in (-\infty,-M]} |\mathscr{f}_{ij}(\xi,\mu)|, \sup_{\xi \in [M,\infty)} |\mathscr{f}_{ij}(\xi,\mu)|\right\} \\
    &=\max\left\{\sup_{\xi \in [-M,M]} |\mathscr{f}_{ij}(\xi,\mu)|,2\sup_{\xi \in [M,\infty)} |\mathscr{f}_{ij}(\xi,\mu)|\right\}
\end{align}
where the last step followed from the fact that $\mathscr{f}_{ij}$ is even in $\xi$. Now, observe that we can write
\begin{align}
    \sup_{\mu \in \mathcal{J}}\sup_{\xi \in [M,\infty)} |\mathscr{f}_{ij}(\xi,\mu)| &= \sup_{\mu \in \mathcal{J}}\sup_{\xi \in [M,\infty)} \left|\frac{(\mathscr{l}_{ij}(\xi) - \mathrm{Re}(\mu)\delta_{ij})^2 + \mathrm{Im}(\mu)^2 \delta_{ij}}{|\mathscr{l}_{den,\mu}(\xi)|^2}\right| \\
    &\leq \sup_{\mu \in \mathcal{J}}\sup_{\xi \in [M,\infty)} \frac{C}{\xi^8} \left| \mathscr{l}_{ij}(\xi)^2 - 2\mathrm{Re}(\mu) \mathscr{l}_{ij}(\xi) \delta_{ij} + \mathrm{Re}(\mu)^2 \delta_{ij} + \mathrm{Im}(\mu)^2 \delta_{ij}\right| \\
    &\leq \sup_{\mu \in \mathcal{J}}\sup_{\xi \in [M,\infty)} \frac{C}{\xi^8} \left( |\mathscr{l}_{ij}(\xi)|^2 + 2|\mathrm{Re}(\mu)| \mathscr{l}_{ij}(\xi) \delta_{ij} + |\mu|^2 \delta_{ij}\right) \\
    &\leq \sup_{\xi \in [M,\infty)} \frac{C}{\xi^8} \left( |\mathscr{l}_{ij}(\xi)|^2 + 2\delta_0 \mathscr{l}_{ij}(\xi) \delta_{ij} + \delta_0^2 \delta_{ij}\right)
\end{align}
where we used that $|\mathrm{Re}(\mu)| \leq |\mu| \leq \delta_0$ as $\mu \in \mathcal{J}$. Now, we work with the supremum on $\xi$.
\begin{align}
    &\sup_{\xi \in [M,\infty)} \frac{C}{\xi^8} \left( |\mathscr{l}_{ij}(\xi)|^2 + 2\delta_0 \mathscr{l}_{ij}(\xi) \delta_{ij} + \delta_0^2 \delta_{ij}\right) \\
    &\leq C\sup_{\xi \in [M,\infty)} \left( \frac{(2\pi)^4 d_{1,ij}^2}{\xi^4} + \frac{(2\pi)^2(|d_{1,ij} d_{2,ij}|+2\delta_0|d_{1,ij}|\delta_{ij})}{\xi^6} + \frac{d_{2,ij}^2 + 2\delta_0 d_{2,ij} \delta_{ij} + \delta_0^2 \delta_{ij}}{\xi^8}\right) \\
    &\leq C \left((2\pi)^4 \frac{d_{1,ij}^2}{M^4} + (2\pi)^2 \frac{|d_{1,ij}d_{2,ij}| + 2\delta_0 |d_{1,ij}|\delta_{ij}}{M^6} + \frac{d_{2,ij}^2 + 2\delta_0 d_{2,ij} \delta_{ij} + \delta_0^2 \delta_{ij}}{M^8}\right)
\end{align}
as desired.
\end{proof}

\renewcommand{\theequation}{D.\arabic{equation}}
\setcounter{equation}{0}

\section{Details for the existence proofs of Theorem \ref{th : existence proofs}}\label{apen : proof of patterns}

The goal of this appendix is to provide to the interested reader more details about the CAPs done in Section~\ref{sec : Bounds of patterns}. In particular, we include in the following propositions the values of the bounds $\mathcal{Y}_0, \mathcal{Z}_2(r_0), Z_1, \mathcal{Z}_u$, the values for $N$ and $N_0$ the thresholds for our projection operators for linear operators and sequences respectively and the domain size $d$ that our approximate solution lives on.

\begin{proposition}[\bf Localized solution in Glycolsis]\label{prop : spike in gyl}
Let $\delta = 0.45, c = a, b = 0.44, j = 0.5, a = 0, h = 0$. Moreover, let $r_0 \bydef 2 \times 10^{-7}$. Then there exists a unique solution $\tilde{\mbf{u}}$ to \eqref{eq : gen_model} in $\overline{B_{r_0}(\mbf{u}_{gyl})} \subset \mathcal{H}_{e}$ and we have that $\|\tilde{\mbf{u}}-\mbf{u}_{gyl}\|_{\mathcal{H}} \leq r_0$. 
\end{proposition}
\begin{proof}
Choose $N_0 = 1400, N = 1100, d = 30$. Then, we perform the full construction to build $\overline{\mathbf{u}} = \bgam^\dagger(\overline{\mathbf{U}})$. Then, we define $\mbf{u}_{gly} \bydef \overline{\mathbf{u}}$. Next, we construct $B^N$, and use Lemma \ref{corr : banach algebra} to find
\begin{align}
    \left\|\begin{bmatrix}
        B_{11}^N - B_{12}^N \\
        B_{21}^N - B_{22}^N
    \end{bmatrix}\right\|_{2} \leq 14.36,~\kappa \bydef 14.2.
\end{align}
This allows us to compute the $\mathcal{Z}_2(r)$ bound defined in Section \ref{sec : Bounds of patterns}. 
Finally, using \citep{julia_blanco_fassler}, we choose $r_0 \bydef 2 \times 10^{-7}$ and define
\begin{align}
    \mathcal{Y}_0 \bydef 5.07 \times 10^{-8} \text{,}~\mathcal{Z}_{2}(r_0) \bydef 952.07    \text{,}~Z_1 \bydef 0.142\text{,}~\mathcal{Z}_u \bydef 3.911 \times 10^{-5}
    \end{align}
and prove that these values satisfy Theorem \ref{th: radii polynomial}. 
\end{proof}
\begin{proposition}[\bf Localized solution in Schnakenberg]\label{prop : Pulse in schnakenberg}
Let $\delta = 0.011, c = 1, b = 3.6, j = 0, a = 2.8, h = 0$. Moreover, let $r_0 \bydef 3 \times 10^{-10}$. Then there exists a unique solution $\tilde{\mbf{u}}$ to \eqref{eq : gen_model} in $\overline{B_{r_0}(\mbf{u}_{sch})} \subset \mathcal{H}_{e}$ and we have that $\|\tilde{\mbf{u}}-\mbf{u}_{sch}\|_{\mathcal{H}} \leq r_0$. 
\end{proposition}
\begin{proof}
Choose $N_0 = 1300, N = 850, d = 8$. The proof is obtained similarly to that of Proposition \ref{prop : spike in gyl}. In particular, we define
\begin{align}
    \mathcal{Y}_0 \bydef 3.54 \times 10^{-11} ,~\mathcal{Z}_2(r_0) \bydef 8.71 \times 10^{6} ,~ Z_1 \bydef 0.099,~ \mathcal{Z}_u \bydef 1.32 \times 10^{-3}
\end{align}
and prove that these values satisfy Theorem \ref{th: radii polynomial}.
 \end{proof}
 \begin{proposition}[\bf Localized solution in Brusselator]\label{prop : spike in brusselator}
Let $\delta = 0.05, c = 1.51, b = 0, j = 0, a = 0.3, h = 0.51$. Moreover, let $r_0 \bydef 8 \times 10^{-5}$. Then there exists a unique solution $\tilde{\mbf{u}}$ to \eqref{eq : gen_model} in $\overline{B_{r_0}(\mbf{u}_{bruss})} \subset \mathcal{H}_{e}$ and we have that $\|\tilde{\mbf{u}}-\mbf{u}_{bruss}\|_{\mathcal{H}} \leq r_0$. 
\end{proposition}
\begin{proof}
Choose $N_0 = 3200, N = 1750, d = 80$. The proof is obtained similarly to that of Proposition \ref{prop : spike in gyl}. In particular, we define
\begin{align}
    \mathcal{Y}_0 \bydef 1.62 \times 10^{-8},~\mathcal{Z}_2(r_0) \bydef 1912.32 ,~ Z_1 \bydef 0.2621,~ \mathcal{Z}_u \bydef 2.23 \times 10^{-5}
\end{align}
and prove that these values satisfy Theorem \ref{th: radii polynomial}.
 \end{proof}
  \begin{proposition}[\bf Localized solution in Brusselator]\label{prop : multispike in brusselator}
Let $\delta = 0.0095742405424, $\\
$c = 2.3299999237, b = 0, j = 0, a = 16.421100616, h = 1.3299999237$. Moreover, let $r_0 \bydef 9 \times 10^{-11}$. Then there exists a unique solution $\tilde{\mbf{u}}$ to \eqref{eq : gen_model} in $\overline{B_{r_0}(\mbf{u}_{bruss})} \subset \mathcal{H}_{e}$ and we have that $\|\tilde{\mbf{u}}-\mbf{u}_{bruss}\|_{\mathcal{H}} \leq r_0$. 
\end{proposition}
\begin{proof}
Choose $N_0 = 2000, N = 1300, d = 10$. The proof is obtained similarly to that of Proposition \ref{prop : spike in gyl}. In particular, we define
\begin{align}
    \mathcal{Y}_0 \bydef 7.51 \times 10^{-12} ,~\mathcal{Z}_2(r_0) \bydef 5.44 \times 10^{7}  ,~ Z_1 \bydef 0.066,~ \mathcal{Z}_u \bydef 3.71 \times 10^{-3} 
\end{align}
and prove that these values satisfy Theorem \ref{th: radii polynomial}.
 \end{proof}
 \begin{proposition}[\bf Localized solution in Selkov-Schnakenberg]\label{prop : spike in selkov schnakenberg}
Let $\delta = 0.01, c = 1, b = 0.589818, j = 0.005, a = 1.095, h = 0$. Moreover, let $r_0 \bydef 4 \times 10^{-6}$. Then there exists a unique solution $\tilde{\mbf{u}}$ to \eqref{eq : gen_model} in $\overline{B_{r_0}(\mbf{u}_{sel})} \subset \mathcal{H}_{e}$ and we have that $\|\tilde{\mbf{u}}-\mbf{u}_{sel}\|_{\mathcal{H}} \leq r_0$. 
\end{proposition}
\begin{proof}
Choose $N_0 = 4000, N = 3300, d = 25$. The proof is obtained similarly to that of Proposition \ref{prop : spike in gyl}. In particular, we define
\begin{align}
    \mathcal{Y}_0 \bydef 1.32 \times 10^{-6},~\mathcal{Z}_2(r_0) \bydef 40821.1,~ Z_1 \bydef 0.1603,~ \mathcal{Z}_u \bydef 8.66 \times 10^{-4}
\end{align}
and prove that these values satisfy Theorem \ref{th: radii polynomial}.
 \end{proof}
\begin{proposition}[\bf Localized solution in Root-Hair]\label{prop : spike in rh}
Let $\delta = 0.01, c = 2, b = 1.8, j = 3.5, a = 0, h = 1$. Moreover, let $r_0 \bydef 6 \times 10^{-7}$. Then there exists a unique solution $\tilde{\mbf{u}}$ to \eqref{eq : gen_model} in $\overline{B_{r_0}(\mbf{u}_{rh})} \subset \mathcal{H}_{e}$ and we have that $\|\tilde{\mbf{u}}-\mbf{u}_{rh}\|_{\mathcal{H}} \leq r_0$. 
\end{proposition}
\begin{proof}
Choose $N_0 = 300, N = 200, d = 3$. The proof is obtained similarly to that of Proposition \ref{prop : spike in gyl}. In particular, we define
\begin{align}
    \mathcal{Y}_0 \bydef 2.62 \times 10^{-8} ,~\mathcal{Z}_2(r_0) \bydef 6417.5,~ Z_1 \bydef 0.184,~ \mathcal{Z}_u \bydef 1.141 \times 10^{-6}
\end{align}
and prove that these values satisfy Theorem \ref{th: radii polynomial}.
 \end{proof}

\nocite{julia_cadiot_blanco_symmetry}
\nocite{julia_cadiot_Kawahara}
\nocite{julia_cadiot_sh}
\nocite{julia_cadiot_blanco_GS}
\nocite{julia_interval}
\bibliographystyle{abbrv}
\bibliography{biblio}
\end{document}